\documentclass[oneside]{amsart}

\usepackage[T1]{fontenc}
\usepackage{amsmath}
\usepackage{amssymb}
\usepackage{amsthm}
\usepackage[bbgreekl]{mathbbol}
\usepackage{cancel}
\usepackage{enumerate}
\usepackage{geometry}
\usepackage{relsize}
\usepackage{tikz-cd}
\usepackage{url}
\usepackage{verbatim}
\usepackage{xspace}

\usepackage{xcolor}
\usepackage[all]{xy}
\usepackage[hidelinks]{hyperref}
\hypersetup{
  pdftitle={Sheared Witt Vectors},
  pdfauthor={Bhargav Bhatt, Akhil Mathew, and Vadim Vologodsky}
}
\usepackage{cleveref}

\DeclareSymbolFontAlphabet{\mathbb}{AMSb} 
\DeclareSymbolFontAlphabet{\mathbbl}{bbold}

\definecolor{todo}{rgb}{1,0,0}
\definecolor{conditional}{rgb}{0,1,0}
\definecolor{e-mail}{rgb}{0,.40,.80}
\definecolor{reference}{rgb}{.20,.60,.22}
\definecolor{mrnumber}{rgb}{.80,.40,0}
\definecolor{citation}{rgb}{0,.40,.80}

\theoremstyle{definition}
\newtheorem{definition}{Definition}[section]
\newtheorem*{question}{Question}

\newtheorem{construction}[definition]{Construction}
\newtheorem{example}[definition]{Example}
\newtheorem{remark}[definition]{Remark}

\theoremstyle{plain}
\newtheorem{proposition}[definition]{Proposition}
\newtheorem{lemma}[definition]{Lemma}
\newtheorem{corollary}[definition]{Corollary}

\newtheorem{theorem}[definition]{Theorem}

\newtheoremstyle{named}{}{}{\itshape}{}{\bfseries}{.}{.5em}{#1 \thmnote{#3}}
\theoremstyle{named}

\DeclareMathOperator{\spec}{Spec}
\DeclareMathOperator{\Spec}{Spec}
\DeclareMathOperator{\Spf}{Spf}

\newcommand{\cone}{\mathrm{cone}}

\newcommand{\FF}{\mathbb{F}}
\newcommand{\MM}{\mathbb{M}}
\newcommand{\ZZ}{\mathbb{Z}}

\DeclareMathOperator{\Hom}{Hom}

\renewcommand{\phi}{\varphi}

\newcommand{\triplearrows}{\begin{smallmatrix} \to \\ \to \\ \to \end{smallmatrix}}

\newcommand{\quot}{/^{\mathbb{L}}}

\newcommand{\two}{\widetilde{2}}
\newcommand{\zpcycl}{\mathbb{Z}[\zeta_{p^\infty}]^{\wedge}_p}

\newcommand{\et}{\mathrm{\acute{e}t}}
\newcommand{\iso}{\cong}

\newcommand{\nperf}{\mathrm{NPerf}}
\newcommand{\nperftors}{\mathrm{NPerf}}

\DeclareRobustCommand{\dcart}{\texorpdfstring{\ensuremath{\delta}-Cartier}{delta-Cartier}\xspace}
\DeclareRobustCommand{\dcartwo}{\texorpdfstring{\ifmmode\delta\text{-\~{C}artier}\else\ensuremath{\delta}-\~{C}artier\fi}{delta-tilde Cartier}\xspace}

\newcommand{\ddcart}{\delta\mathrm{-CartCAlg}}
\newcommand{\ddcartwo}{\delta\mathrm{-}\widetilde{\mathrm{C}}\mathrm{artCAlg}}
\newcommand{\dieu}{Dieudonn\'e $\delta$}

\newcommand{\deltanilperfect}{\texorpdfstring{$\delta$-nilperfect}{delta-nilperfect}}
\newcommand{\dhl}{\delta\mathrm{-Ring}_{\delta\mathrm{-nilperf}}}
\newcommand{\dring}{\delta\mathrm{-Ring}}
\newcommand{\hV}{\tilde{V}}
\newcommand{\lt}{L^{\mathrm{taut}}}
\newcommand{\dnil}[1]{\sqrt[\delta]{#1}}
\newcommand{\dniltop}[1]{\sqrt[\delta]{#1}^{\,\mathrm{top}}}
\newcommand{\nilperf}[1]{#1^{\delta\text{-}\mathrm{nilperf}}}

\newcommand{\chW}{{}^{s}\!W}
\newcommand{\gh}{\mathrm{gh}}

\newcommand{\hw}{\hat{W}}
\newcommand{\hwbig}{\hw_{\mathrm{big}}}

\newcommand{\wbig}{W_{\mathrm{big}}}
\providecommand{\Sym}{\operatorname{Sym}}
\providecommand{\Pic}{\operatorname{Pic}}
\providecommand{\Ext}{\operatorname{Ext}}
\providecommand{\Norm}{\operatorname{Norm}}
\providecommand{\ev}{\operatorname{ev}}
\providecommand{\Wbig}{W_{\mathrm{big},R}}
\providecommand{\Wrat}{W_{\mathrm{rat},R}}
\providecommand{\Wratplus}{W_{\mathrm{rat},R}^{+}}
\providecommand{\Aone}{\mathbb{A}^1_R}
\providecommand{\RHom}{\mathrm{R}\!\operatorname{Hom}}
\providecommand{\rGamma}{\mathrm{R}\Gamma}
\providecommand{\mono}{\hookrightarrow}
\providecommand{\Z}{\ZZ}
\newcommand{\qperf}{Q^{\mathrm{perf}}}
\newcommand{\ainf}{A_{\inf}}

\newcommand{\gasharp}{\mathbb{G}_a^{\sharp}}
\newcommand{\gahat}{\widehat{\mathbb{G}_a}}
\newcommand{\gahatsharp}{\mathbb{G}_a^{\widehat{\sharp}}}
\newcommand{\gabar}{\overline{\mathbb{G}_a}}
\newcommand{\Prism}{{ {\mathbbl{\Delta}}}}

\title{Sheared Witt vectors}
\author{Bhargav Bhatt}
\address{School of Mathematics, Institute for Advanced Study, 1 Einstein Drive, Princeton, NJ 08540, USA; and Department of Mathematics, Princeton University, Princeton, NJ 08544, USA}
\email{bhargav.bhatt@gmail.com}

\author{Akhil Mathew}
\address{Department of Mathematics, University of Chicago, 5734 S. University Ave., Chicago, IL 60637-1514, USA}
\email{amathew@math.uchicago.edu}

\author{Vadim Vologodsky}
\address{Department of Mathematics, University of Toronto, 40 St. George St., Toronto, ON M5S 2E4, Canada}
\email{vadim.vologodski@utoronto.ca}

\date{\today}
\dedicatory{Dedicated to Vladimir Drinfeld}

\begin{document}

\begin{abstract}
V.~Drinfeld
and E.~Lau introduced a ``decompletion'' of the ring of $p$-typical Witt
vectors, following earlier work of T.~Zink. The goal of this paper is to offer
an exposition of this construction, which we call the sheared Witt vectors, on the category of rings $R$ whose reduction is a perfect $\mathbb{F}_p$-algebra.
\end{abstract}
\maketitle
\tableofcontents

\section{Introduction}

The purpose of this paper is to study a ``decompleted'' variant $\chW$ of the
($p$-typical)
Witt vector functor $W$.

Let us begin with some motivation: why would one seek such a decompletion?
Given a perfect field $k$ (or more generally a perfect $\mathbb{F}_p$-algebra
by \cite[Th.~D]{Lausmoothness}), the Dieudonn\'e
theory
(cf.~\cite[Ch.~V]{DG70})
gives a classification of the category of commutative finite locally free group
schemes over $k$ of $p$-power order via Dieudonn\'e modules, or $W(k)$-modules equipped with
Frobenius and Verschiebung.
It is a natural question to
extend the Dieudonn\'e theory to more general rings.
One then runs into the following problem:  for any $n \geq 1$, the functor
carrying a ring $R$ to
$W(R)/p^n$ does not  commute with filtered colimits. However, the category of
finite locally free group schemes over a filtered colimit of a diagram of rings
$\left\{R_i\right\}_{i \in I}$ is the filtered colimit of the
categories of finite locally free group schemes over each $R_i$.
Thus, a Dieudonn\'e theory over non-perfect rings necessitates a
``decompletion'' of the ring of Witt vectors. (Another reason to care about such a modification of the Witt vectors, and historically our original one, is explained in Remark~\ref{rmk:sheared}.)

Such a functor of ``decompleted'' Witt vectors was studied by T.~Zink
\cite{ZinkDieudonne} and Lau \cite{Laurelations}, for rings $R$ such that
$R_{\mathrm{red}}$ is perfect
and such that there exists $N \geq 1$ such that $x^N = 0$ for all $x \in
\mathrm{Nil}(R)$.
It was recently observed by V.~Drinfeld \cite{Dri25} and E.~Lau \cite{Lau25} that the
construction can be
extended to
all $R$ with $R_{\mathrm{red}}$ perfect, and the
Dieudonn\'e theory for $p$-divisible groups was developed by M.~Hoff and E.~Lau
 \cite{HL26}. The purpose of this text is to
give an exposition of the Drinfeld--Lau--Zink construction, which we will call the
\emph{sheared Witt vectors}, and its basic properties.

\begin{definition}
For any ring $R$, we let $\hw(R) \subset W(R)$ be the collection of Witt vectors
$x = \sum_{i \geq 0} V^i [x_i] \in W(R)$
such that all Witt components $x_i \in R$ of $x$ are
nilpotent  and
such that all but finitely
many vanish. Then $\hw(R) \subset W(R)$ is an ideal, stable under
the Witt vector Frobenius and Verschiebung.
\end{definition}

\begin{definition}[Drinfeld, Lau]
\label{introshearedW}
Let $R$ be a ring whose reduction  $R_{\mathrm{red}}$ is a perfect $\mathbb{F}_p$-algebra; we will call such rings \emph{nilperfect}.
We define the \emph{sheared Witt vectors} $\chW(R)$ via the pullback diagram
of rings
\[ \xymatrix{
\chW(R) \ar[d]  \ar[r] & \varprojlim_F (W(R)/\hw(R))\ar[d]  \\
W(R) \ar[r] &  W(R)/\hw(R)
}.\]
By construction, we have a natural map
$\chW(R) \to W(R)$, and one can show that $\chW(R)$ is derived
(but usually not classically) $p$-complete as an abelian group. We have a short exact sequence
\[ 0 \to \hw(R) \to \chW(R) \to \varprojlim_F (W(R)/\hw(R)) \to 0. \]

We will later extend this construction to derived
$p$-complete rings $R$ such that
$(R/p)_{\mathrm{red}}$ is
perfect; we call such rings \emph{$p$-completely nilperfect}. For example,\footnote{However,
this case does not cover all $R$ of interest. If $R$ is a
semiperfect $\mathbb{F}_p$-algebra, then $\chW(R)$ is a derived $p$-complete ring
with typically unbounded $p$-power torsion, which is usually not classically $p$-complete. Thus,
$\chW( \chW(R))$ is not covered by this case.}
 if $R$ is a $p$-adically complete ring
with bounded $p$-power torsion and such that $(R/p)_{\mathrm{red}}$ is perfect,
then we \emph{define} $\chW(R) = \varprojlim_n \chW(R/p^n)$. 
\end{definition}

\begin{remark}
    $\chW$ is a sheaf of rings on the category of nilperfect rings equipped with the fpqc topology, cf.~\Cref{chWisfpqcsheaf}.  
    The category of nilperfect rings forms a basis for the fpqc topology on all $p$-nilpotent rings. Therefore, one can extend the definition of $\chW$ to all $p$-nilpotent rings by fpqc descent (see Remark~\ref{rmk:cohomologysW}); however, we will not treat this extension in this paper, because of the presence of higher cohomology. 
\end{remark}

\begin{example}
Let $R$ be a local Artinian ring with perfect residue field $k$ of characteristic $p$ and maximal
ideal $\mathfrak{m}$.
Note that one has a natural embedding $W(k) \hookrightarrow W(R)$, which is the unique section of the natural map $W(R) \to W(k)$.
In this case, $\chW(R) \subset W(R)$ is the subring
$W(k) \oplus \hw(\mathfrak{m}) \subset W(R)$. This subring has been considered
by \cite{ZinkDieudonne, Laurelations}. See \Cref{chWofArtinian}.
\end{example}

\begin{example}
    \label{chWofsemiperfect:intro}
Let $R$ be a semiperfect $\mathbb{F}_p$-algebra
written as a quotient $R = P/I$ for $P$ a perfect $\mathbb{F}_p$-algebra and
an ideal $I \subset P$.
In this case,
the map
$W(P) \simeq \chW(P) \twoheadrightarrow \chW(R)$
is a surjection, 
and $\chW(R)$ is the derived $p$-completion of the quotient of $W(P)$
by the ideal generated by $V^i([x])$, for all $x \in I$ and $i \geq 0$. In
particular, $\chW(R) \twoheadrightarrow W(R)$.

For instance, if $R$ is the quotient $P/(f_1, \dots, f_n)$ where $P$ is a
perfect
$\mathbb{F}_p$-algebra and $f_1, \dots, f_n \in P$, then
$\chW(R)$ is the derived $p$-completion of the quotient of
$W(P)$ by the elements
$\{ p^{i} [f_j^{1/p^i}], i \in \mathbb{Z}_{\geq 0}, 1 \leq j \leq n \} $.
For instance, if $R = \mathbb{F}_p[x^{1/p^\infty}]/(x)$, then $\chW(R)$ is the
``graded decompletion'' of $W(R)$, i.e., $\chW(R)$ is the derived
$p$-completion of the direct sum
of various cyclic groups generated by $[x^i]$, for $i \in \mathbb{Z}[1/p]_{\geq
0}$. See \Cref{chWofsemiperfect,truncatedsemiperfectchW}.
\end{example}

\subsection{An analog of Joyal's characterization for $\chW$}
To further characterize $\chW$,
we recall some preliminaries about
$\delta$-rings and the Witt vectors, cf.~\cite[\S 2]{bhattscholze} for an
account.
Recall that a \emph{$\delta$-ring} is a ring $A$ equipped with a map $\delta: A
\to A$ satisfying various identities, which guarantee that
the map $\phi: A \to A$ defined by $\phi(a) \stackrel{\mathrm{def}}{=} a^p + p
\delta(a)$ is a ring homomorphism (and is equivalent to that when $A$ is
$p$-torsionfree). We
write $\dring$ for the category of $\delta$-rings.
The forgetful functor from $\dring$ to the category of commutative rings
preserves limits and colimits, and admits both adjoints.
The fundamental relation between $\dring$ and the Witt vector construction is as
follows:

\begin{theorem}[Joyal \cite{Joyal85}]
\label{joyal:witt}
For any ring $R$, the ring $W(R)$ has a natural $\delta$-structure such that
$\phi$ is the Witt vector Frobenius $F: W(R) \to W(R)$, and the map
$W(R) \to R$ exhibits $W(R)$ as the cofree $\delta$-ring on $R$. That is, for
any $\delta$-ring $A$, the map $W(R) \to R$ induces a bijection
$\mathrm{Hom}_\delta(A, W(R)) \xrightarrow{\sim} \mathrm{Hom}(A, R)$.
\end{theorem}

For a nilperfect ring $R$, the ring $\chW(R)$ also admits a natural $\delta$-structure, and the map $\chW(R) \to R$ exhibits $\chW(R)$ as the cofree ``$\delta$-nilperfect'' $\delta$-ring in the sense that we now formulate.

\begin{definition}
A $\delta$-ring $A$ is \emph{perfect} if the Frobenius $\phi: A \to A$ is an
isomorphism.
\end{definition}

\begin{remark} \label{perfectdeltarings}
If $R$ is a perfect $\mathbb{F}_p$-algebra, then $W(R)$ is perfect as a
$\delta$-ring, and
conversely any perfect $\delta$-ring which is derived $p$-complete is of this
form.

More generally, $W(R)$ for $R$ a perfect $\mathbb{F}_p$-algebra has the following universal
property: for any derived $p$-complete $\delta$-ring $A$, maps $W(R) \to A$
(either of rings or of $\delta$-rings) are the same as maps of
$\mathbb{F}_p$-algebras $R \to
A^{\flat} \stackrel{\mathrm{def}}{=}
\varprojlim_\phi(A/p)$; see \cite[\S 2.4]{bhattscholze}.
    \end{remark}

\begin{definition}[Cf.~\Cref{deltanilpdef}]
Fix a
$\delta$-ring $A$. We say that
an element $a \in A$ is \emph{$\delta$-nilpotent} if
$\delta^i(a)$ is nilpotent for all $i \geq 0$ and
$\delta^i(a) = 0$ for $i \gg 0$. Equivalently, by
\Cref{deltaofnilpotent} below, it suffices to require that $a$ is nilpotent
and $\delta^i(a) = 0$ for $i \gg 0$.
\end{definition}

An important example of a $\delta$-nilpotent element is the product of
any two $p$-power torsion elements of $A$ (\Cref{prodlemmadeltaring}).

\begin{definition}[\deltanilperfect\ $\delta$-rings]\label{def:Nilperfect_delta-rings}
Let $A$ be a $\delta$-ring. We say that $A$ is \emph{\deltanilperfect} if
$A/\dnil{A}$ is a perfect $\delta$-ring, where $\dnil{A} \subset A$ denotes the ideal of
$\delta$-nilpotent elements.
\end{definition}

With this definition,
any perfect $\delta$-ring is \deltanilperfect. Moreover, it is straightforward to
check that any colimit of \deltanilperfect\
$\delta$-rings remains \deltanilperfect.

We now formulate the analog of \Cref{joyal:witt} for $\chW$.

\begin{theorem}[Cf.~\Cref{chWcofreenilperfect}]
\label{chWcofreenilperfect:intro}
Let $R$ be a nilperfect ring.  Then the map
$\chW(R) \to R$ exhibits $\chW(R)$ as the cofree \deltanilperfect\ $\delta$-ring
on $R$. That is, for any \deltanilperfect\ $\delta$-ring $A$, the natural
map
\[
\Hom_{\dring}(A,\chW(R)) \to \Hom(A,R)
\]
is a bijection.
\end{theorem}
\subsection{$\chW$ via taut square-zero extensions}

The kernel of the map $\chW(R) \to W(R)$ is a square-zero ideal, and $\delta$ acts isomorphically on the kernel. 
Moreover, locally in the flat topology, $\chW(R) \to W(R)$ is surjective. 
Using this observation, we can give another construction of $\chW(R)$ as a $\delta$-ring, at least for some classes of $R$, based on the following definition. 

\begin{definition}[Taut square-zero extensions and derivations]
Let $B \twoheadrightarrow A$ be a surjection of $\delta$-rings whose kernel $I
\subset B$ squares to zero.
We say that the square-zero extension $B \twoheadrightarrow A$ is \emph{taut} if
$\delta$ induces an
isomorphism of abelian groups\footnote{$\delta$ always induces an additive map
on a square-zero
$\delta$-ideal.} $\delta:  I \xrightarrow{\sim} I$ and $I$ is derived $p$-complete. 
One can similarly define the notion of a \emph{taut derivation} of a $\delta$-ring. 
\end{definition}

\begin{remark}
\label{perfectringssplit}
If $A$ is a perfect $\delta$-ring and if $I \hookrightarrow B \twoheadrightarrow A$ is a
square-zero extension of $\delta$-rings with
$I$ derived $p$-complete, then the extension has a unique splitting $A \to B$;
this is a
consequence of the
universal property of the Witt vectors of perfect $\FF_p$-algebras (\Cref{perfectdeltarings}).
More generally, by \Cref{tautsqzerosplits}, any taut square-zero extension of a
\deltanilperfect\ $\delta$-ring by a derived $p$-complete ideal admits a unique
section.
\end{remark}

\begin{theorem}[Cf.~\Cref{chWasuniversaltaut} and \Cref{universaltautsquarezero:general}]
\label{chWunivtaut:intro}
Let $R$ be a $p$-completely nilperfect ring such that $\chW(R) \to W(R)$ is surjective. Then the map
$\chW(R) \to W(R)$ exhibits $\chW(R)$ as the universal taut square-zero
extension of $W(R)$ by a derived $p$-complete ideal. That is, for every taut
square-zero extension $A \twoheadrightarrow W(R)$, there is a unique map of $\delta$-rings $\chW(R) \to A$ over $W(R)$.
\end{theorem}

Why should taut square-zero extensions lead to a decompletion of the Witt vectors?
One observation is the following:
if $\{R_i\}_{i \in I}$ is a filtered diagram of $p$-nilpotent rings such that the Frobenius $F: W(R_i) \to W(R_i)$ is surjective for each $i$, then the natural map
$(\varinjlim_i W(R_i))_{\hat{p}} \to W(\varinjlim_i R_i)$ is a taut square-zero extension (this can be deduced from \Cref{infiniteVdivisibleideals}).
In other words, taut square-zero extensions naturally
arise from the failure of the Witt vector functor to commute with filtered colimits, and taking the universal such extension thus gives a decompletion of the Witt vectors.

\subsection{A left adjoint characterization}
We also formulate a characterization of $\chW$ as a \emph{left} adjoint functor;
this proves useful in giving explicit presentations of $\chW$ in a number of
cases.
For this, it will be convenient to incorporate some additional structure, which
is also carried by $W$. Recall the Witt vector Verschiebung $V: W(R) \to W(R)$.
The Verschiebung satisfies additional compatibilities with the
$\delta$-structure, which can be axiomatized
as follows:
\begin{definition}[Magidson \cite{Magidson}, Drinfeld]
A \dcart ring consists of a $\delta$-ring $A$ and an additive map $V: A
\to A$ with the following properties:
\begin{enumerate}
\item We have the identity $F(V(x)) = px$ for all $x \in A$.
\item We have the projection formula $V( F(x) y ) = xV(y)$ for all $x, y \in
A$.
\item For any $x \in A$, $\delta(V(x)) = x  - p^{p-2} V(x^p)$.
\end{enumerate}
\end{definition}

We prove that the Verschiebung on a \dcart ring is always injective. For $p > 2$, the category of \dcart rings is a full subcategory of the category of $\delta$-rings equipped with an ideal (intended to be the image of Verschiebung), cf.~\Cref{automaticV}; to first approximation, $\delta$ serves as a partial inverse of Verschiebung. When one works with derived $p$-complete \dcart rings for $p > 2$, then $V: A \to A$ has a natural retraction (\Cref{naturalretractionofV}).

The notion of a \dcart ring neatly characterizes the Witt vectors,
via the following result.

\begin{theorem}[Magidson \cite{Magidson}]
\label{VcompleteW}
If $A$ is any \dcart ring and $R$ is any ring, then
\begin{equation}  \Hom_{\ddcart}(A, W(R)) \xrightarrow{\sim} \Hom(A/V, R).
\end{equation}
As a consequence, the functor $R \mapsto W(R)$ establishes an equivalence
between the category of
rings and the category of \dcart rings $A$ which are $V$-complete (i.e.,
$A \xrightarrow{\sim} \varprojlim_n A/V^n A$).
\end{theorem}

In other words,
$W(-)$ is the \emph{right adjoint}
of the functor $A \mapsto A/V$ from \dcart rings to rings, and $W(R)$ is
the \emph{terminal} \dcart ring $A$ with an isomorphism $A/V
\xrightarrow{\sim} R$.
In seeking a decompletion of $W(R)$, it is thus natural to consider
\dcart rings $A$ with $A/V \xrightarrow{\sim} R$ but which are not
$V$-complete.

Indeed, for $p > 2$, $\chW (R)$ acquires the structure of a \dcart ring (in
fact, all
the terms of the square defining $\chW$ do)
and the map $\chW(R) \to W(R)$ is one of \dcart rings; this map induces an
isomorphism on $V$-completion.
We prove that $\chW(R)$ has the dual universal property for derived $p$-complete \dcart rings, 
and is therefore the \emph{initial} derived $p$-complete \dcart ring $A$ with an isomorphism
$A/V \xrightarrow{\sim} R$.

\begin{theorem}[Cf.~\Cref{chWuniversalcartier}]
\label{chWleftadjoint:intro}
Assume $p>2$. Let $R$ be a $p$-complete ring such that $R$ has bounded $p$-power torsion and $(R/p)_{\mathrm{red}}$ is perfect.  Then for any derived $p$-complete \dcart ring
$A$, the natural map
$R \to \chW(R)/V$ induces an
isomorphism
\[
\Hom_{\ddcart}(\chW(R),A) \xrightarrow{\sim} \Hom(R,A/V).
\]
\end{theorem}

\Cref{chWuniversalcartier} is actually a consequence of the theory of taut square-zero extensions of $\delta$-rings. 
Our key observation is that for any derived $p$-complete \dcart ring $A$, the map $A \to \widehat{A}$ to the $V$-completion  is the composite of a taut square-zero extension of $\delta$-rings and the kernel of a taut derivation; the starting point for this fact is that the ideal $V^i(A) \subset A$  squares into $p V^i(A)$ for each $i \geq 1$, from which one readily proves that $\bigcap_{i \geq 0} 
V^i(A) \subset A$ is a taut square-zero ideal on which $V$ gives the inverse to $\delta$.

Now $\chW(R)$ has a ``taut rigidity'' property: all taut square-zero extensions of $\chW(R)$ split uniquely. As a consequence, we deduce that maps of \dcart rings $\chW(R) \to A$  identify with maps of \dcart rings $\chW(R) \to \widehat{A}$, which is equivalent to maps of rings $R \to \widehat{A}/V$ by \Cref{VcompleteW}.

At $p = 2$, there is an analogous result, but with a more complicated algebraic structure called \dcartwo rings, cf.~\Cref{sec:dcartwo}.
Rather than a Verschiebung, a \dcartwo ring $A$ is equipped with a modified Verschiebung $\hV: A \to A$ such that $F\hV$ is no longer multiplication by 2 but rather by an element $\two$ that is required to satisfy further axioms. We develop analogs of the main results for \dcart rings in the setting of \dcartwo rings; for example, we show that \dcartwo rings are a full subcategory of pairs of a $\delta$-ring and an ideal (which is false for \dcart rings at $p = 2$), cf.~\Cref{automatichV}.

We  extend the definition of $\chW$ to all derived $p$-complete rings $R$ with $(R/p)_{\mathrm{red}}$ perfect,
by taking the statement of 
\Cref{chWleftadjoint:intro} as the \emph{definition} of $\chW(R)$ for such $R$.
 
In \Cref{subsec:cyclotomicexamples}, we 
use the algebra of \dcart (resp. \dcartwo) rings to give explicit presentations of $\chW(R)$ for various classes of $R$, e.g., obtaining a mixed characteristic analog of \Cref{chWofsemiperfect:intro}.
 
\begin{example}
Let $R$ be a perfectoid ring. We show in \Cref{thm:chWofperfectoid} that $\chW(R)$ is the ring-theoretic cosaturation of $W(R)$, considered as an algebra over $\ainf(R)$ with respect to the almost structure induced by the map $\ainf(R) \to W(R) \to W( (R/p)_{\mathrm{red}})$. In particular, $\chW(R)$ is always a quotient of $\ainf(R)$ when $R$ is perfectoid.  More generally, the Frobenius is surjective on $\chW(R)$ whenever $R$ is semiperfectoid, in contrast to the Witt vectors, cf.~\cite{dk}. 
\end{example}

\begin{example}
    Consider the ring 
    $R = \left(\zpcycl[x^{1/p^\infty}]\right)^{\wedge}_p/(x)$. 
    In this case, one has the following explicit description of $\chW(R)$ as a sort of ``$q$-analog'' of \Cref{chWofsemiperfect:intro}: it is the $p$-adic completion of the quotient of the $\delta$-ring
\[
   \mathbb{Z}_p[q^{1/p^\infty}, x^{1/p^\infty}]
\]
  by 
  \begin{itemize}
    \item  The elements $(q-1)(q^{1/p^n} - 1)$ for all $n \geq 1$.
    \item
    The elements
    $x, [p]_{q^{1/p}} x^{1/p}, [p]_{q^{1/p^2}} [p]_{q^{1/p}} x^{1/p^2}, \dots$ where $[p]_q$ is the $q$-analog of $p$.
  \end{itemize} 

\end{example}

\begin{remark}[The cohomology of $\chW$]
\label{rmk:cohomologysW}
In later work, we shall explore further cohomological aspects of the $\chW$ construction.
For example, we will show that $\chW$ has no higher cohomology in the flat topology on $\nperftors$. 
For $R$ a smooth $\mathbb{F}_p$-algebra, we will show
that the flat cohomology complex $R\Gamma(\Spec(R),\chW)$ is computed by the two-term complex
$$ W(R) \xrightarrow{d} \left(W \Omega_R^1[1/F]\right)^{\wedge}_p,$$
obtained by truncating the de Rham--Witt complex of $R$ and $p$-completely 
inverting the operator $F$ in degree 1. 
In fact, for any derived $p$-complete ring $R$, the complex $R\Gamma(\mathrm{Spf}(R),\chW)$ is the fiber of the map from $W(R)$ to its taut cotangent complex introduced in \Cref{def:tautcotangentcomplex}. 
\end{remark}

\begin{remark}[Sheared prismatization]
\label{rmk:sheared}
This project began as a prequel to the joint work of the authors with
Artem Kanaev and Mingjia Zhang on \emph{sheared prismatization}
\cite{BKMVZ}, whose goal is to construct an enlargement of the
prismatic formalism devoid of ``nilpotence of $p$-curvature''
constraints. Let us explain this connection. Fix a nilperfect ring
$R$. Recall \cite{bhattlurieAPC} that a point of $\Spf(\mathbb{Z}_p)^\Prism(R)$,
also known as a Cartier--Witt divisor over $R$, is given by a
generalized Cartier divisor $\left(\alpha\colon I \to W(R)\right)$ together with an
isomorphism with $\left(p\colon W(R) \to W(R)\right)$ after base change along
$W(R) \to W(R_{\mathrm{red}}) = W(R)/W(\mathrm{Nil}(R))$. In our initial approach to the results
of \cite{BKMVZ}, we defined a \emph{sheared Cartier--Witt divisor} over
$R$ to be a generalized Cartier divisor $\left(\alpha\colon I \to W(R)\right)$
together with an isomorphism with $\left(p\colon W(R) \to W(R)\right)$ after base
change along $W(R) \to Q(R) \stackrel{\mathrm{def}}{=} W(R)/\hw(R)$.
(The adjective ``sheared''
reflects the stronger constraints on the congruence $\alpha \equiv p$
imposed in the definition.) Upon learning this definition, Drinfeld
suggested instead defining $\chW(-)$ as in \Cref{introshearedW}, and
then defining sheared Cartier--Witt divisors by simply replacing
$W(-)$ with $\chW(-)$ in the definition of Cartier--Witt divisors.
This approach is conceptually cleaner, and is indeed the approach we
shall adopt in the writeup \cite{BKMVZ}; we thank Drinfeld for his
suggestion.
\end{remark}

\subsection*{Notation and conventions}
Throughout, $p$ is a fixed prime number. 

Given a $\delta$-ring $A$, we write 
$w_\delta$ for the unique map of $\delta$-rings, given by \Cref{joyal:witt},
$$ w_\delta: A \to W(A) $$
whose composition with the projection $W(A) \to A$ is the identity.

\subsection*{Acknowledgments}
The definition of $\chW$ and most of the main results of section 3, as well as the key definitions in section 5, were communicated to us by Vladimir Drinfeld. We thank him warmly for his generosity and encouragement, as well as numerous questions and comments on the text.  

As indicated in
Remark~\ref{rmk:sheared}, this project is closely related to
\cite{BKMVZ}, and in fact 
Mingjia Zhang was involved in earlier stages of this project; although she declined to be a coauthor, we thank her heartily for her contributions and for many helpful discussions.

We are also grateful to 
Shachar Carmeli, 
Arthur-César Le Bras, Eike Lau, Jacob Lurie, Kirill Magidson, Joshua Mundinger, 
Juan Esteban Rodr\'iguez Camargo, Nick Rozenblyum, Alexander Petrov, and Peter Scholze  for helpful discussions on the material of this paper.
Pullback squares similar to the one in \Cref{introshearedW} also 
arise in the theory of analytic Witt vectors and analytic prismatization/syntomification
developed in ongoing works of 
Ansch\"utz,  Bosco, Hauck, Le Bras, Rodr\'iguez Camargo, and Scholze. 

Bhatt was partially supported by grants from the Packard Foundation and the Simons Foundation (MPS-SIM-00622511, MPS-PERF-00001529-02). 

Mathew was supported by the Simons Collaboration on Perfection in Algebra, Geometry, and Topology and the NSF
(grants  2507081 and 2152311). Mathew also thanks the Institute for Advanced Study, the Mathematisches Forschungsinstitut Oberwolfach, and the University of Toronto  for their hospitality. 

The work of Vologodsky was supported by an NSERC Discovery Grant
(RGPIN-2026-07347). Part of this work was carried out during Vologodsky's
visits to the University of Chicago and Princeton University, which were supported by the Simons Foundation (the Simons Collaboration on Perfection in Algebra, Geometry, and Topology and MPS-SIM-00622511). 
Vologodsky thanks both
institutions for their hospitality.

\subsection*{Tool and computational resource disclosure}
Some of the results in Sections 5 and 6 were discovered with the help of computer assistance, specifically ChatGPT/Codex/Claude. The main results (in particular, \Cref{chWuniversalcartier})
were originally discovered without computer assistance, but computer assistance 
helped us discover some of the general structural results about \dcartwo rings. Moreover, discussions with ChatGPT helped us simplify, correct, and streamline a number of arguments throughout the paper, and in particular in Sections 5 and 6. 

We used these tools (and GitHub Copilot) for proofreading and editing the manuscript, as well as for directly drafting pieces of proofs in a few cases (with human guidance) in Sections 5 and 6. We have edited, reviewed, and verified all the content generated by these tools to ensure its accuracy.

\section{Preliminaries on $\hw$}

The purpose of this section is to collect various general facts about the Witt
vectors and $\delta$-rings.
In particular,
for any ring $R$, we review the
definition and properties of the ideal $\hw(R) \subset W(R)$.
We also discuss the interaction
between torsion and nilpotence in $\delta$-rings, and in particular define for a
$\delta$-ring $A$ the ideal $\dnil{A}$ of $\delta$-nilpotent elements
(\Cref{deltanilpdef}).

\subsection{Completion of an affine scheme with respect to a grading}
Let $\MM_m = ( \mathbb{A}^1, \times)$ be the multiplicative monoid scheme, so
that an $\MM_m$-action on a module is
equivalent to a nonnegative grading.
We
construct a subfunctor of an affine scheme with $\MM_m$-action.

Fix a base ring $A$.
\begin{definition}[The subfunctor $\hat{M} \subset M$]
\label{Mhatdef}

    Let $M$ be an affine scheme over $A$ equipped with an action of the
    multiplicative monoid $\MM_m$.
    In other words, $M$ is the spectrum of a nonnegatively graded $A$-algebra $T
    = \bigoplus_{i \geq 0} T_i$.

	 We regard $M$ as a functor from $A$-algebras to sets.
Define a subfunctor $\hat{M} \subset M$  as follows: given an $A$-algebra $B$,
     $\hat{M}(B) \subset M(B)$ consists of those elements such that the
	 corresponding $A$-algebra map $T \to B$ annihilates $\bigoplus_{i \geq N}
	 T_i$ for some $N \in \mathbb{Z}_{\geq 0}$. Then  $\hat{M}$
	 is the sequential colimit of the affine schemes	 $\spec( T/ \bigoplus_{i
	 \geq N} T_i)$ under the natural closed inclusions as $N$ increases; in
	 particular, $\hat{M}$ is an ind-scheme.
\end{definition}

\begin{proposition} \label{Mhatcommuteswithfilteredcolimits}
    Suppose that $M = \spec T$ is as in \Cref{Mhatdef} and $T_0$ is a finitely
    presented $A$-algebra and each $T_i$ is a finitely presented $T_0$-module.
    Then the functor  $\hat{M}$ commutes with filtered colimits.
\end{proposition}
\begin{proof}
    This follows because $\hat{M}$ is the filtered colimit of the functors
    $\spec( T / \bigoplus_{i \geq N} T_i)$ and each of these commutes with
    filtered colimits, since each quotient $T / \bigoplus_{i \geq N} T_i$ is a finitely presented $A$-algebra.
\end{proof}
\begin{remark} \label{usefulremarkabouthats}
\begin{enumerate}
    \item
The functor that sends an affine $A$-scheme $M$ with $\MM_m$-action to the
functor $\hat{M}$ commutes with finite limits.
    \item Given an affine $A$-scheme $M$ with $\MM_m$-action, we can rescale the
    $\MM_m$-action by precomposition with the $n$-th power map $t \mapsto t^n$
    (alternatively, rescale the grading by $n$).
    This rescaling does not change the subfunctor $\hat{M} \subset M$.
    \item
    Given an affine $A$-scheme $M$ with $\MM_m$-action and any other affine
    $A$-scheme $M'$, we can equip $M \times M'$ with the action only on the
    first factor.
    Then
    \[ \widehat{ M \times M'} = \hat{M} \times M' . \]
\end{enumerate}
\end{remark}

\begin{corollary}
\label{Mhatissubgroup}
    Suppose $M$ is as in \Cref{Mhatdef} and is equipped with the structure of a
    group scheme compatible with $\MM_m$-action: that is,
the corresponding graded ring $T$ is given the structure of a graded Hopf
algebra. In this case, $\hat{M} \subset M$ is a subgroup functor.\footnote{In particular, $\hat{M}$ is represented by a group ind-scheme.}
\end{corollary}
\begin{proof}
    This follows from the previous remark.
\end{proof}

\subsection{Definition of $\hw$}

In this section, for any ring $R$, we review the construction of  the ideal
$\hw(R) \subset W(R)$
in the ring $W(R)$ of Witt vectors of $R$, and analogously for $\wbig(R)$
(cf.~\cite[V.4.4]{DG70}).

\begin{definition}[The subfunctor $\hw \subset W$]
\label{wh via grading}
The functor $W$ has a natural $\MM_m$-action, such that $r \in R$ acts on $W(R)$
by multiplication by $[r]$.
This action satisfies the condition of \Cref{Mhatdef} and lets us define a
subfunctor $\hw \subset W$. In the same way, we define a subfunctor $\hwbig
\subset \wbig$. Note that the definition also extends to nonunital rings in the evident way.
\end{definition}

\begin{proposition}[$\hw, \hwbig$ via Witt components]
\label{hwviaWitt}
Let $x \in W(R)$ be expressed as $\sum_{i \geq 0} V^i [x_i]$ for $x_i \in R$.
Then $x \in \hw(R)$ if and only if
all of the $x_i$ are nilpotent and $x_i = 0$ for $i \gg 0$.
 Similarly, an element of $\wbig$ belongs to $\hwbig$ if and only if all the
 Witt components are nilpotent and all but finitely many are zero.
\end{proposition}
\begin{proof}
 The $i$th Witt component is (as a function on the Witt scheme) homogeneous of
 degree $p^i$.
Moreover, the algebra of functions on the Witt scheme is precisely the
polynomial algebra on the Witt coordinates.
The result follows.
\end{proof}

\begin{remark} Any element of $\hw(R) \subset W(R)$ is nilpotent. 
    \end{remark}

\begin{remark}[$\hw$ via Joyal coordinates]\label{hwviaJoyal}
The ring of functions on the Witt scheme $W$ has another set of generators,
namely, the Joyal coordinates $j_i, i \geq 0$; the $i$th Joyal coordinate of $x
\in W(R)$ is the zeroth ghost (or Witt) component of $\delta^i(x)$; see
\cite[Def.~3.1.1]{KedlayaPrismaticNotes} and the original source
\cite{Joyal85}.
The ring of functions on $W$ is also the polynomial ring on $j_0, j_1, \dots$
and each $j_i$ is homogeneous of degree $p^i $. Therefore, it follows by similar
reasoning that
$x \in W(R)$ belongs to $\hw(R)$ if and only if all the Joyal coordinates are
nilpotent and all but finitely many vanish.
\end{remark}

\begin{remark}
    The functor $\hw(-)$ commutes with filtered colimits, and is representable
    by an ind-scheme, by \Cref{Mhatcommuteswithfilteredcolimits}.
	 The same holds for $\hwbig(-)$.
\end{remark}

\begin{proposition}
\label{hwbasic}
For any ring $R$,
$\hw(R) \subset W(R)$ is an ideal, and it is stable under $F$ and $V$.
Similarly, $\hwbig(R) \subset \wbig(R)$ is an ideal, and it is stable under all
the Frobenius and Verschiebung operations.
\end{proposition}

\begin{proof}
The statement that $\hw(R) \subset W(R), \hwbig(R) \subset \wbig(R)$ are
subgroups follows from \Cref{Mhatissubgroup}, as does the stability under
Frobenius and  Verschiebung operators (which are $\MM_m$-equivariant up to
rescaling the $\MM_m$-actions).
Since the multiplication $W \times W \to W$ (or $\wbig \times \wbig \to \wbig$)
is a map of $\MM_m$-equivariant schemes where we make $\MM_m$ act only on the
first factor in $W \times W$ (or $\wbig \times \wbig$)
we conclude from \Cref{usefulremarkabouthats} that $\hw \subset W$
(resp.~$\hwbig \subset \wbig$) is an ideal.
\end{proof}

\begin{proposition}
\label{Fnil}
Let $R$ be any ring.
Suppose $x \in \hw(R)$. Then $F^m x = 0 $ for $m \gg 0$.
Similarly, if $y \in \hwbig(R)$, then the $m$-th Frobenius annihilates $y$ for $m
\gg 0$.
\end{proposition}
\begin{proof}
This follows because $F^n: W \to W$ is a map of $\MM_m$-equivariant schemes,
where the $\MM_m$-equivariant structure on the target is defined by
precomposition with the map $p^n: \MM_m \to \MM_m$. In other words, $F^n$
induces a map on the rings of functions on $W$ which multiplies the degree by
$p^n$.
 The argument for $\hwbig \subset \wbig$ is similar.
\end{proof}

As an abelian group,
\begin{equation} \wbig(R) = (1 + t R[[t]])^{\times}.
\label{bigwittaspowerseries} \end{equation}

\begin{proposition}
\label{bigwitthwaspoly}
For any ring $R$, under the identification
\eqref{bigwittaspowerseries}, $\hwbig(R) \subset \wbig(R)$ corresponds to
$ \mathrm{ker}( ( R[t])^{\times} \to R^{\times})$.
\end{proposition}
\begin{proof}
In fact,
this follows because under the identification
\eqref{bigwittaspowerseries},
the coefficient of $t^i$, as a function on $\wbig$, is homogeneous of degree
$i$.
\end{proof}

\begin{proposition}[Product criterion for $\hw$]
Let $R$ be any ring and let $x, y \in W(R)$.
Suppose that:
\begin{enumerate}
\item $F^n x = 0$ for some $n$.
\item All Witt coordinates (or equivalently Joyal coordinates) of $y$ are
nilpotent.\footnote{Equivalently, $y$ belongs to the kernel of $W(R) \to
W(R_{\mathrm{red}})$.}
\end{enumerate}
Then $xy \in \hw(R)$.
\label{productlemmahw}
\end{proposition}
\begin{proof}
By writing $y = \sum_{i \geq 0} V^i( [y_i])$, we can write
$y = a + V^n ( b)$ where $a \in \hw(R)$ and $b \in W(R)$.
Then \[xy = x ( a + V^n (b)) = xa + V^n (  (F^n x)b ) = xa \in \hw(R),\]
since $a \in \hw(R)$ and $\hw(R) \subset W(R)$ is an ideal.
\end{proof}

\begin{example}[Gradings on $\hw$]
\label{gradingsonhw}
Let $R$ be a $\mathbb{Z}[1/p]_{\geq 0}$-graded ring, and let $I = R_+ =
\bigoplus_{i \in \mathbb{Z}[1/p]_{>0}} R_i$.
Suppose that there exists $N \in \mathbb{N}$ such that $R_i = 0$ for all $i \geq
N$; in particular, every element of $I$ is nilpotent.

In this case, we can place a natural $\mathbb{Z}[1/p]_{>0}$ grading on $\hw(I)$
such that
the elements in degree $d$ are those (finite) sums $\sum_{i \geq 0} V^i [x_i]
\in \hw(I)$
such that $x_i \in I$ is
homogeneous of degree $p^i d$.
To see this, it suffices to show that every element of $\hw(I)$ is a finite sum
of homogeneous elements. By \Cref{hwviaWitt}, every element of $\hw(I)$ is a
finite sum of elements of the form $V^i[x]$ with $x \in I$, so it is enough to
treat $[x]$. Write $x$ as a finite sum of homogeneous elements. Since $\hw$
commutes with filtered colimits, we may replace $I$ by the nonunital graded
subring generated by these homogeneous components. After multiplying all degrees
by a power of $p$, we may therefore assume that $I$ is concentrated in positive
integer degrees.

Set $I_{\geq n} = \bigoplus_{m \geq n} I_m$. The filtration
$I = I_{\geq 1} \supset I_{\geq 2} \supset \cdots$ is finite because the grading
is bounded above. It induces a finite filtration on $\hw(I)$ whose associated
graded terms are
\[
\hw(I_{\geq n}/I_{\geq n+1}) = \bigoplus_{\mathbb{N}} I_{\geq n}/I_{\geq n+1}.
\]
Indeed, $I_{\geq n}/I_{\geq n+1}$ is a square-zero nonunital ring. Thus every element of each associated graded term is a finite sum of
homogeneous elements. Inducting along the finite filtration, every element of
$\hw(I)$ is a finite sum of homogeneous terms, as desired.
\end{example}

\subsection{$p$-nilpotent rings}
In this section we prove some results about $\hw \subset W$ that are specific to
 $p$-nilpotent rings.

\begin{example}
Let $R$ be a nonzero $p$-nilpotent ring.
Then $[p] \in \hw(R)$, but $p \notin \hw(R)$.
\end{example}

Throughout this section, we will frequently use the following elementary Witt
vector identity.

\begin{proposition}
Let $R$ be any ring. If $a, b \in R$ and $m \leq n$, then
\begin{equation}
    V^m ( [a])\cdot V^n ( [b]) = p^{m} V^n ( [ a^{p^{n-m} }b]).
    \label{wittidentity}
\end{equation}
\end{proposition}
\begin{proof}
    This follows from the Witt vector identities $x V^i (y) = V^i( (F^i x) y)$,
    $F^i V^i = p^i$, and $F^i ( [z]) = [z^{p^i}]$.
\end{proof}

\begin{proposition}
    \label{lemmatorsioninWittvectors}
Let $R$ be any ring. Let $x \in R$ such that $p^m x = 0$
and $x^n = 0$. Then
there exists $A = A(m, n) \in \mathbb{Z}_{\geq 0}$ such that $p^A [x] = 0$.
\end{proposition}

\begin{proof}
In any ring, for any element $c$, the $i$th Witt component of
\[ F V ([c]) - V F([c])  \]
is a polynomial $\psi_i(c) \in \mathbb{Z}[c]$ in $c$ with no constant term.
Since $F, V$ commute in the Witt vectors of $\mathbb{F}_p$-algebras, we find
that $\psi_i$ is divisible by $p$ for each $i$.

Thus, we find that
\[p[x] = FV [x] = V( [x^p]) +  \sum_{i=0}^\infty V^i ( [ \psi_i(x)]). \]
The right-hand side is a (necessarily finite since $p[x] \in \hw(R)$) sum of iterated applications of the Verschiebung to
terms that satisfy the hypotheses of the lemma for either $(m-1, n)$ or $(m,
n-1)$.
We conclude the lemma by induction on $(m, n)$.
\end{proof}

\begin{corollary}
\label{torsionnilpotent}
Let $R$ be a $p$-nilpotent ring.
Suppose $x \in \hw(R)$. Then $p^m x = 0$ for $m \gg 0$: that is, $\hw(R)
\subset W(R)$ consists of $p$-power torsion elements.
\end{corollary}

Conversely, we can give the following characterization of torsion elements in
$W(R)$, for $R$ $p$-nilpotent.

\begin{proposition}
\label{torsioninW}
    Let $R$ be a $p$-nilpotent ring. The following are equivalent for $x \in
    W(R)$:
    \begin{enumerate}
    \item $x$ is $p$-power torsion.
    \item There exists $n$ such that $F^n(x) = 0$.
    \item There exist $m, n$ such that $p^m F^n (x) \in \hw(R)$.
        \item There exists $N \geq 0$ such that all the Witt components $x_i$ of
        $x$ satisfy $x_i^N = 0$.
    \end{enumerate}
	 Moreover, each of $n, m, N$ can be bounded uniformly in terms of the other two and of the power of
	 $p$ that annihilates $R$.
\end{proposition}
\begin{proof}
Clearly (1) and (2) imply  (3).
(4) implies both (1) and (2)
 thanks to \Cref{lemmatorsioninWittvectors}.
Let us prove that (3) implies (4).
    Without loss of generality, by replacing $R$ by $R/p$,  we can take $R$ to be an $\mathbb{F}_p$-algebra.
    Since $V$ and $F$ commute and $FV = p$, the condition (3) on $x$ becomes the
    assertion that $F^r(x) \in \hw(R)$ for $r \gg 0$. Since $F$ is given by
    raising Witt components to the $p$th power, this condition holds if and only
    if the Witt components of $x$ are uniformly nilpotent, as desired.
\end{proof}

\begin{proposition}[Product criterion in $p$-nilpotent rings]
\label{productsinhwadic}
    Let $R$ be a $p$-nilpotent ring.     Suppose $x\in W(R)$ is $p$-power
	 torsion.
Suppose $y \in \mathrm{ker}(W(R) \to W( R_{\mathrm{red}}))$.
Then $xy \in \hw(R)$.
\label{productsinhw}
	 \end{proposition}
	 \begin{proof}
This follows from \Cref{torsioninW} and \Cref{productlemmahw}.
	 \end{proof}

\begin{proposition}
\label{squareofaminusb}
    Let $R$ be a $p$-nilpotent ring, and let $a, b \in R$ be  elements such that
    $a-b$ is nilpotent.
    Then $([a] - [b])^2 \in \hw(R)$.
\end{proposition}
\begin{proof}
    Our assumptions imply that $a^{p^n} = b^{p^n} $ for $n \gg 0$, whence $F^n (
    [a] -[b]) = 0$. The result now follows from \Cref{productsinhw}.
\end{proof}

\begin{remark}
In \Cref{squareofaminusb}, $[a]- [b]$ need not belong to $\hw(R)$.
For example, if $R = \mathbb{F}_p[\epsilon]/\epsilon^2$, the element $[1+
\epsilon] - 1 \notin \hw(R)$.

In fact, we claim that the Witt coordinates of $[1 + \epsilon] - 1 \in W(
\mathbb{F}_p[\epsilon]/\epsilon^2)$ are
all equal to $\epsilon$. To see this, we may work in
$\mathbb{Z}_p[\epsilon]/\epsilon^2$ and observe that the ghost coordinates are
$(\epsilon, p \epsilon, p^2 \epsilon, \dots )$, from which we may solve for the
Witt coordinates as claimed.
\end{remark}

\subsection{Pre-adic rings}
We formulate the next results more generally for pre-adic rings where $p$ is
topologically nilpotent.
For us, a \emph{pre-adic ring} is a topological ring $R$ such that there exists a finitely generated ideal $I \subset R$ (called an \emph{ideal of definition}) such that $R$ has the $I$-adic topology (cf.~\cite[\href{https://stacks.math.columbia.edu/tag/07E7}{Tag 07E7}]{stacks-project}).

\begin{construction}[$\hw, \hwbig$ for a pre-adic ring]
Given a pre-adic ring $R$, define $\hw(R) \subset W(R)$ to be the ideal consisting of elements whose Witt
components (or equivalently Joyal coordinates) are all topologically nilpotent and converge to zero in the topology
(and similarly for $\hwbig(R) \subset \wbig(R)$).
If $I$ is an ideal of definition and $R$ is $I$-adically complete, then
$\hw(R) = \varprojlim_n \hw(R/I^n)$.
In general,
$\hw(R) = W(R) \times_{W( \widehat{R}_I)} \hw( \widehat{R}_I )$.
\end{construction}

We will need the following criterion for when an element of $W(R)$ belongs to
$\hw(R)$.

\begin{proposition}
\label{criterionforhwunderF}
    Let $R$ be a pre-adic ring where $p$ is topologically nilpotent.
    Let $I$ be an ideal of definition.
    Suppose that for any $n > 0$ and $y \in R$, the following holds:
    if
    $p y \in I^{np}$, then $y \in I^{n+1}$.

    Then given
    $x \in W(R)$, we have $x \in \hw(R)$ if and only if
    \begin{enumerate}
        \item
        The image of $x$  in $W(R/I)$ belongs to $\hw(R/I)$.
        \item $Fx \in \hw(R)$.
    \end{enumerate}
     If $R/I$ is moreover reduced, then item (1) is redundant.
\end{proposition}
\begin{proof}
Clearly items (1) and (2) are necessary for $x \in \hw(R)$, so we need to show
that they are sufficient.
Let $x \in W(R)$ satisfy (1) and (2).
Consider the Joyal coordinates $(x_i)_{i \geq 0}$ of $x$; our assumptions imply that they are topologically nilpotent, so we need to show that
these converge to zero in the $I$-adic topology.

The Joyal coordinates of $Fx$ are $(x_i^p + p x_{i+1})_{i \geq 0}$.
By assumption, the sequence
$(x_i^p + p x_{i+1})_{i \geq 0}$ converges to zero as $i \to \infty$ in the
$I$-adic topology.
We need to show that for any $M$, there exists $N  = N(M)$ such that $x_i \in
I^M$ for $i > N$. We will prove this by induction on $M$.

Item (1) implies $x_i \in I$ for  $i \gg 0$.
This proves the base case $M = 1$. Note that if $R/I$ is reduced (and therefore an $\mathbb{F}_p$-algebra), then the
condition $Fx \in \hw(R)$ already implies that
$x$ even maps to zero in $W(R/I)$.

Now we treat the inductive step.
Suppose that we know that $x_i \in I^{M}$ (and hence $x_i^p  \in I^{Mp}$) for
all $i \gg 0$.
We also know (since $Fx \in \hw(R)$) that $x_i^p + px_{i+1} \in I^{Mp}$ for all
$i \gg 0$.
For $i \gg 0$, this gives $p x_{i+1} \in I^{Mp} $, which implies $x_{i+1} \in
I^{M+1}$ by the assumption on the pair $(R, I)$. This completes the inductive
step and concludes the proof.
\end{proof}

\begin{proposition} \label{ghostcrit}
Let $R$ be a pre-adic ring with the $p$-adic topology, such that $R$ is
$p$-torsionfree.  Let $x \in W(R)$. Then the following are equivalent:
 \begin{enumerate}
 \item
 $x \in \hw(R)$.
 \item
 \begin{enumerate}
 \item The image of $x$ in $W(R/p)$ belongs to $\hw(R/p)$.
     \item
      There exists a sequence $a_n\in \mathbb{Z}$, with $\lim a_n = +\infty$
      such that
			the $n$-th ghost component $ \gh_n(x)$ of $x$ is divisible by $p^{n+a_n}$.
\item If $p =2$, we require additionally that $x$ maps to an element of
$\hw(R/4) \subset W(R/4)$.
 \end{enumerate}

\end{enumerate}
\end{proposition}
\begin{proof}
It is easy to see that (1) implies (2).
Conversely, let $x$ be a Witt vector satisfying the conditions in (2).
We need to show that for any $A$, all but finitely many of the Witt components
are divisible by $p^A$.
Suppose the contrary; then choose $A$ minimal such that $p^A \nmid x_i$ for
infinitely many $i$.
Our assumptions imply that $A > 1$ (and $A>2$ for $p = 2$). Since the hypotheses
are invariant under replacing $x$ by $x + \epsilon$ for any $\epsilon \in
\hw(R)$, we may assume furthermore (also using the inductive hypotheses) that
\emph{all} $x_i$ are divisible by $p^{A-1}$, i.e., we subtract $V^i [x_i]$ for
the (finitely many) $i$ such that $p^{A-1} \nmid x_i$.

Consider the expression
\begin{equation}
 \gh_n(x) = x_0^{p^n} + p x_1^{p^{n-1}} + \dots + p^n x_n. \label{ghost:auxiliary}
 \end{equation}
For $n \gg 0$, $p^{n+A} \mid \gh_n(x)$.
Moreover, since $p^{A-1} \mid x_i$ for all $i$, we find that the term $p^i
  x_i^{p^{n-i}}$ is divisible by $p^{i + (A-1) p^{n-i}}$.
Since $A > 1$, we have for $i < n$ the inequality \[i + (A-1)p^{n-i} \geq i +
(A-1) + (n-i) + 1 = n+A\]
(using the inequality $a p^b \geq a + b + 1$ if $a \geq 1, b \geq 1$ and $p \neq
2$ or $a \geq 2, b \geq 1$ if $p = 2$).

Thus, in the expression \eqref{ghost:auxiliary}, all terms but the last term $p^n x_n$
are divisible by $p^{n+A}$, which implies that $p^{n+A} \mid p^n x_n$, or $p^A
\mid x_n$, as desired.
\end{proof}

\medskip
This criterion also implies the following corollary.

\begin{corollary}
\label{Fandwhatp}
 Let $R$ be a pre-adic ring such that $R$ has the $p$-adic topology and $R$ is
 $p$-torsionfree.
  Suppose that $x \in W(R)$ is such that $Fx \in \hw(R)$. Suppose moreover  that:
  \begin{enumerate}
      \item For $p > 2$, the image of $x$ in $W(R/p)$ belongs to $\hw(R/p)$.
      \item For $p =2$, the image of $x$ in $W(R/4)$ belongs to $\hw(R/4)$.
	  \end{enumerate}
	  Then $x \in \hw(R)$.
\end{corollary}
\begin{proof}
    Set $y = Fx$. By \Cref{ghostcrit}, there exists a sequence $a_n \to +\infty$
    such that the $n$th ghost component $\gh_n(y)$ is divisible by $p^{n+a_n}$.
    Since $\gh_n(y) = \gh_{n+1}(x)$, the ghost components of $x$ satisfy the
    divisibility condition in \Cref{ghostcrit} as well. Together with the
    hypothesis on the image of $x$ in $W(R/p)$ for $p>2$ and in $W(R/4)$ for
    $p=2$, this implies that $x \in \hw(R)$.
\end{proof}

\begin{proposition}
\label{pminusV1}
    For $p > 2$,  $p - V(1) \in \hw(\mathbb{Z}_p) $.
    \end{proposition}
    \begin{proof}
    In fact, this element maps to zero in $W( \mathbb{F}_p)$ and is annihilated
    by Frobenius, so we may apply \Cref{Fandwhatp}.
\end{proof}

The previous proposition fails when $p =2$.
In fact, we have the following result:

\begin{proposition}
    \label{impossibilityinWZ4}
    It is not possible to solve in $W( \mathbb{Z}/4) $ the equation
    $F\bar{x} = V(1) + \epsilon$ with $\bar{x} \in W(\mathbb{Z}/4), \epsilon \in
    \hw(\mathbb{Z}/4)$. In particular, $2 - V(1) \notin \hw( \mathbb{Z}/4)$.
\end{proposition}

\begin{proof}
Choose a lift $x \in W( \mathbb{Z}_2)$ of $\bar{x}$ and write
$\epsilon = F x - V(1)$.  
    Let the Witt coordinates of $x$ be $(x_i)_{i \geq 0}$  and those of
    $\epsilon$ be $(\epsilon_i)_{i \geq 0}$.
    Then by assumption, $v_2(\epsilon_i) \geq 1$ for each $i$. We will derive a
    contradiction by showing that
    $v_2( \epsilon_i) = 1$ for all $i$, so $\epsilon$ cannot project to an element of $\hw( \mathbb{Z}/4)$.

Note first that $x_1$ is a 2-adic unit and $x_i$ is divisible by 2 for $i \neq
1$; this follows from comparison in $W( \mathbb{F}_2)$, where (since $\epsilon$
maps to zero) we find that $x$ must map to $2 = V(1) \in W( \mathbb{F}_2)$.
    Let us take the ghost coordinates of the equation $Fx - V(1)=
 \epsilon$. We find
\begin{gather}
    x_0^2 + 2x_1 = \epsilon_0 \\
    x_0^4 + 2x_1^2 + 4x_2 -2= \epsilon_0^2 + 2 \epsilon_1 \\
    x_0^8 + 2x_1^4 + 4x_2^2 + 8 x_3 - 2 =   \epsilon_0^4 + 2 \epsilon_1^2 + 4
    \epsilon_2 \\
    \dots
\end{gather}
Since $x_1$ is a 2-adic unit and $v_2(x_0) \geq 1$, we find that
$v_2(\epsilon_0) = 1$.

For $i \geq 1$,  we find that the $i$th ghost component  $\gh_i(\epsilon)$ of $\epsilon = F(x) - V(1)$ is also
$\gh_i( F(x) - 2) = \gh_{i+1}(x - 2)
$. This has 2-adic valuation at least $i+2$ since all the Witt components of $x-2$
are divisible by $2$, so $v_2( \gh_i(\epsilon)) \geq i+2$.
On the other hand, 
\[ \gh_i( \epsilon) = \epsilon_0^{2^i} + 2 \epsilon_1^{2^{i-1}} + \dots + 2^{i-1} \epsilon_{i-1}^2
+ 2^i \epsilon_i.\]
If we inductively assume that $v_2( \epsilon_j) = 1$ for $j < i$, the fact that
the displayed quantity has 2-adic valuation $\geq i+2$ forces $v_2( \epsilon_i)
= 1$.
By induction, we conclude that $v_2( \epsilon_i) = 1$ for all $i$.

Finally, the last claim that $2 - V(1) \notin \hw(\mathbb{Z}/4)$ follows because
otherwise we could take $\bar{x} = 2, \epsilon = 2 - V(1)$ to solve the stated
equation.
\end{proof}
\begin{proposition}
\label{uinWZ2}
For $p = 2$,  $2 - V( [-1]) \in \hw( \mathbb{Z}_2)$.
\end{proposition}
\begin{proof}
We claim that the Witt components $x_n$ of $2 - V( [-1])$ satisfy $v_2(x_n) \geq 2$ for $n > 0$.
Moreover,  the ghost components are
$2, 4, 0, 0, 0, \dots$. This will imply 
 $2 - V( [-1]) \in \hw( \mathbb{Z}_2)$ by \Cref{ghostcrit}.

To prove the claim, note that $x_0 = 2$, $x_1 = 0$, and, in general, the Witt components satisfy
\[x_0^{2^n} + 2 x_1^{2^{n-1}} + 4 x_2^{2^{n-2}} + \dots + 2^n x_n = 0.\]
Suppose inductively that for $0 < i < n$, we have $v_2(x_i) \geq 2$.
For $n \geq 2$, all the monomials in the above expression except for the last
then have 2-adic valuation at least $n+2$.
(In fact, $2^n \geq n+2$, and for $0 < i < n$, $i + 2( 2^{n-i}) \geq n+2$.) This
implies  that $v_2(x_n) \geq 2$ for $n > 0$ and completes the proof.
\end{proof}

\subsection{The Frobenius $F: \hw \to \hw$}

Recall that the Frobenius $F: W \to W$ is faithfully flat as a map of schemes
\cite[Prop.~3.4.7]{bhattlurieAPC}.
We start by proving an analog for $\hw$, where it suffices to use the fppf
topology.

\begin{proposition}[Fppf surjectivity of Frobenius]
\label{Fisfppfsurj}
Let $x \in \hw(R)$. Then there exists an fppf cover $R \to R'$ and an element $y
\in \hw(R')$ such that $Fy = x$. That is, $F$ is fppf locally surjective on
$\hw$. \end{proposition}
\begin{proof}
Since \(x\) is a finite sum of terms \(V^n[a]\) with \(a\) nilpotent, it
suffices, after composing finitely many fppf covers, to treat the case
\(x=V^n[a]\).
Fppf locally on $R$, we can write $a = b^{p^n}$, so
$x = [b] V^n(1)$.
Choose $M$ such that $b^{p^M} = 0$.

After replacing \(R\) by an fppf cover, we can find \(c \in R\) and
\(z,y\in W(R)\) such that
$b = c^p$ and $V^n (1) = F(z) + V^{M}(y)$.
The first claim is evident and the second follows because
$F: W_{M+1} \to W_M$ (i.e., the Frobenius on truncated Witt vectors) is an fppf
surjection.

Then we get
\begin{align*} x &  = V^n ( [a]) \\
& =
[b] V^n(1) \\
& = [b]( F(z) + V^M(y)) \\
& = [b] F(z) \\
& =
F( [c]\cdot z)
\end{align*}
Here we used that $b^{p^M} = 0$.
Now $[c]\cdot z \in \hw(R)$ because $c$ is nilpotent
and
$\hw(R) \subset W(R)$ is an ideal by \Cref{hwbasic}.
\end{proof}

Let $W^F$ denote the kernel of $F: W \to W$, and similarly let $\hw^F$ denote
the kernel of $F: \hw \to \hw$.
Recall from  \cite[Lem.~3.2.6]{drinfeld} that over $\mathbb{Z}_{(p)}$,
the natural map
$W^F \to W \to \mathbb{G}_a$ lifts uniquely to an isomorphism
$W^F \xrightarrow{\sim} \gasharp$, where $\gasharp = \spec
\mathbb{Z}_{(p)}\left[ \frac{x^i}{i!}\right]_{i \geq 1}$ is the
divided-power envelope of $0$ in $\mathbb{G}_a$.

\begin{definition}[The group ind-scheme $\gahatsharp$]  We have a subfunctor
$\gahatsharp \subset \gasharp$ defined as follows: $\gasharp $ has a natural
$\MM_m$-action by rescaling, and we use \Cref{Mhatdef} to define the subfunctor
$\gahatsharp$.
Therefore,
\[ \gahatsharp = \varinjlim_N \spec\left(
\mathbb{Z}_{(p)}\left[ \frac{x^i}{i!}\right]_{i \geq 1}
/ \left(
\frac{x^j}{j!}\right)_{j \geq N} \right).  \]
\end{definition}

Explicitly, $\gasharp(R)$ parametrizes the choice of an element $x \in R$ with
divided powers $\{x^{(i)}\}_{i \geq 0}$. An element of $\gasharp(R)$ belongs to
$\gahatsharp(R)$ if and only if all but finitely many of the divided powers
$x^{(i)}$ vanish, i.e., if the divided powers are ``nilpotent.''

\begin{proposition}
\label{hwFgahatsharp}
    The isomorphism $W^F \simeq \gasharp$ over $\mathbb{Z}_{(p)}$ restricts to
    an isomorphism $\hw^F \simeq \gahatsharp$.
	 Thus, by \Cref{Fisfppfsurj}, we have a short exact sequence of fppf sheaves
	 $0 \to \gahatsharp \to \hw \stackrel{F}{\to} \hw \to 0$.
\end{proposition}
\begin{proof}
    This follows from the identification $\gasharp \simeq \mathrm{ker}(F:
	 W \to W)$
	 because the construction $M \mapsto \hat{M}$ commutes with finite limits
	 (\Cref{usefulremarkabouthats}).
\end{proof}

\begin{remark}
    From this point of view, the fact that $p - V(1) \in \hw( \mathbb{Z}_p)$ for
    $p > 2$ but not for $ p = 2$ is equivalent to the statement that the
    canonical divided powers on $(p)\subset \mathbb{Z}_p$ tend $p$-adically to
    zero precisely when $p>2$.
\end{remark}

\subsection{$\delta$-nilpotent elements}

\begin{definition}[$\delta$-nilpotent elements]
\label{deltanilpdef}
Let $A$ be any $\delta$-ring.
We say that an element $a \in A$ is \emph{$\delta$-nilpotent} if
$\delta^i(a)$ is nilpotent for all\footnote{As shown below in \Cref{deltaofnilpotent}, the nilpotence of $a$ implies the nilpotence of $\delta^i(a)$ for all $i \geq 0$.} $i \geq 0$ and
$\delta^i(a) = 0$ for $i \gg
0$.
Let $\dnil{A} \subset A$ be the subset of $\delta$-nilpotent elements.
\end{definition}

\begin{example}
Let $R$ be a $\mathbb{Q}_{\geq 0}$-graded $\delta$-ring such that $R_i = 0 $ for
all $i \gg 0$.
Then any homogeneous element of $R$ of positive degree is
$\delta$-nilpotent, because $\delta$ multiplies the degree by $p$.
\end{example}

\begin{example}
Consider the ring $\mathbb{Z}_p[\epsilon]/(\epsilon^2, p \epsilon)$.
We can make this into a $\delta$-ring via $\delta(a \epsilon) = a \epsilon$ for
$a \in \mathbb{F}_p$.
Then $\epsilon$ is a nilpotent, $p$-torsion element which is not $\delta$-nilpotent.
\end{example}

\begin{proposition}
For any $\delta$-ring $A$,
$\dnil{A}$ is the preimage of
$\hw(A)$ via the $\delta$-map $w_\delta: A \to W(A)$.
\label{crit:dnil}
\end{proposition}
\begin{proof}
Given $a \in A$, the Joyal coordinates of $w_\delta(a)$ are $a, \delta(a),
\delta^2(a), \dots$.  The result follows.
\end{proof}

\begin{corollary}
$\dnil{A} \subset A$ is a $\delta$-ideal, so $A/\dnil{A}$ acquires the structure
of a $\delta$-ring.
\end{corollary}

\begin{proposition}
\label{dnilinW}
Let $R$ be any ring, and consider the $\delta$-ring
$W(R)$. Then $\dnil{ W(R)} = \hw(R)$.
\end{proposition}
\begin{proof}
Suppose $x \in W(R)$ is $\delta$-nilpotent, so that
$\delta^i(x) $ is nilpotent for all $i$ and vanishes for $i \gg 0$.
Taking the zeroth ghost component of each $\delta^i(x)$, we find that the Joyal coordinates of $x$ are nilpotent and vanish for $i \gg 0$, so $x \in \hw(R)$.

Conversely, suppose $x \in \hw(R)$.
This implies that for all sufficiently high degree polynomials $P$ in the Witt components $\{x_i, i \geq 0\}$ (with $x_i$ placed in degree $p^i$), we have $P(x_0, x_1, \dots) = 0$. This implies that all the Witt components of $\delta^i(x)$ are nilpotent and vanish for $i \gg 0$, so $\delta^i(x) \in \hw(R)$ for all $i$ and vanishes for $i \gg 0$.
\end{proof}

We now include some general results
about the interaction between torsion and nilpotence in $\delta$-rings.

\begin{lemma}[{Cf.~\cite[Lem.~2.28]{bhattscholze}}]
\label{torsionmeansnilpotent}
Let $A$ be any $\delta$-ring, and $a \in A$.
Suppose $a$ is $p$-power torsion.
Then:
\begin{enumerate}
\item
$a$ is nilpotent.
\item
If $p^n a = 0$, then $\phi^n(a) = 0$.
\end{enumerate}
\end{lemma}
\begin{proof}
For (2),  we reduce to the case $n = 1$. For any element $x$ of a
$\delta$-ring, one has
$ \delta (px) = \phi(x) - p^{p-1} x^p$.
Taking $x = a$
with $pa = 0$, we find that $\phi(a) = 0$ as desired.

For (1),
it suffices to show that for any perfect field $k$ and any map $f: A \to k$,
$f(a) = 0$.
Now the $\delta$-map
\[A \xrightarrow{w_\delta} W(A) \xrightarrow{W(f)} W(k) \]
necessarily carries $a$ to zero, since $W(k)$ is $p$-torsionfree.
But the composite of this map with the projection $W(k) \to k$ is $f$, whence
the claim. (For an explicit proof, compare \emph{loc.~cit.})
\end{proof}

\begin{corollary}
\label{frobinjtorsionfree}
A $\delta$-ring $A$ such that $\phi: A \to A$ is injective is  $p$-torsionfree.
\end{corollary}

\begin{corollary}
    \label{deltaofnilpotent}
    Let $A$ be any $\delta$-ring, and $a \in A$.
    If $a$ is nilpotent, then $\delta^i(a)$ is nilpotent for each $i$.
\end{corollary}
\begin{proof}
The identity $\phi(a) = a^p + p \delta(a)$ shows that $p\delta(a) = \phi(a) - a^p$ is nilpotent, so some power of $\delta(a)$ is $p$-power torsion and hence nilpotent by \Cref{torsionmeansnilpotent}; therefore, $\delta(a)$ is nilpotent. The result follows by induction.
\end{proof}

\begin{lemma}
\label{deltapreservestorsion}
Let $A$ be any $\delta$-ring.
Suppose $a \in A$ is $p$-power torsion.
Then $\delta^i(a)$ is $p$-power torsion for each $i$.
\end{lemma}
\begin{proof}
The identity
$\phi(a) = a^p + p \delta(a)$ shows that $\delta(a)$, and inductively
$\delta^i(a)$ for each $i$, is $p$-power torsion.
\end{proof}

\begin{lemma} Let $A$ be any $\delta$-ring.
\label{prodlemmadeltaring}
If $x, y \in A[p^\infty]$, then $xy \in \dnil{A}$.
\end{lemma}
\begin{proof}
We use the map $w_\delta: A \to W(A)$.
Note that $w_\delta(x)$ is annihilated by a power of $F$
thanks to \Cref{torsionmeansnilpotent}.
Moreover, the Joyal coordinates of $w_\delta(y)$, i.e., $\delta^i(y)$ for $i \geq 0$, are all nilpotent since $y$ is nilpotent (cf.~\Cref{torsionmeansnilpotent} and \Cref{deltaofnilpotent}).

It follows that $w_\delta(x) w_\delta(y) \in \hw(A)$ by
\Cref{productlemmahw}.
We now conclude that $xy \in \dnil{A}$ by \Cref{crit:dnil}.
\end{proof}

\section{Definition and basic structure of $\chW$}

In this section, we study aspects of the functor $\chW$, defined in
\Cref{introshearedW},
on rings $R$
such that $R_{\mathrm{red}}$ is perfect.
Most of these results were communicated to us by Drinfeld, and appear in
\cite{Dri25,Lau25, HL26}.

\begin{definition}[Nilperfect rings]
Let $\nperf$ denote the category of rings $R$ such that
$R_{\mathrm{red}}$ is a perfect $\mathbb{F}_p$-algebra. We will call such rings
\emph{nilperfect.}
\end{definition}

Note that $\nperf$ includes any semiperfect $\mathbb{F}_p$-algebra, as well as
any
local Artinian ring with perfect residue field of characteristic $p$.

\subsection{$\hw$-torsors and $\hwbig$-torsors}

We begin with a collection of results on $\hw$-torsors
and
$\hwbig$-torsors.

\begin{proposition}[Drinfeld]
\label{wbigtorsors}
Let $R$ be any ring. Consider the following two categories:
\begin{enumerate}
\item The groupoid of $\hwbig$-torsors on $R$ in the fpqc topology.
\item The groupoid of invertible $R[t]$-modules $\mathcal{L}$ together with a
trivialization at $t = 0$.
\end{enumerate}
There is a natural fully faithful inclusion from category (1) into category (2).
\end{proposition}
\begin{proof}
Both categories (1) and (2) satisfy faithfully flat
descent, and the group of automorphisms of
the unit object in (2) (namely, $R[t]$ with the evident trivialization $1$) is
$\hwbig(R) = \mathrm{ker}( (  R[t])^{\times} \to R^{\times})$.
Since every object of (1)  (i.e., every $\hwbig$-torsor) is locally trivial,
we obtain the fully faithful inclusion from (1) into (2).
\end{proof}

Recall \cite[Tag 0EUK]{stacks-project}
that a ring $A$ is \emph{seminormal} if for all $x, y \in A$ such that $x^3 =
y^2$, there exists a unique $a \in A$ such that $x = a^2, y = a^3$.
Note that a perfect $\mathbb{F}_p$-algebra is seminormal; this is because the
map of rings $\mathbb{F}_p[x, y]/(x^3-y^2) \to \mathbb{F}_p[a]$ carrying
$x \mapsto a^2, y \mapsto a^3$ induces an isomorphism on perfections.

\begin{corollary}[Drinfeld]
\label{vanishingofhwtorsors}
Let $R$ be a ring such that $R_{\mathrm{red}}$ is seminormal (e.g., perfect).
Then all fpqc $\hwbig$-torsors on $R$ are trivial. If $R$ is additionally $p$-local, then all
fpqc $\hw$-torsors on $R$ are trivial.
\end{corollary}
\begin{proof}
This follows from \Cref{wbigtorsors} because $\mathrm{Pic}(R[t])
\xrightarrow{\sim} \mathrm{Pic}(R)$ in
this case \cite[Th.~1]{Swan80}. For the second statement, we use that $\hw$ is a
direct
summand of $\hwbig$.
\end{proof}

\begin{corollary}
\label{localtrivialtyWbig}
The fully faithful inclusion of \Cref{wbigtorsors} is an equivalence for a
$p$-nilpotent ring $R$.
Moreover, every fpqc $\hwbig$-torsor (or $\hw$-torsor) can be trivialized
locally in the fppf topology.
\end{corollary}
\begin{proof}
By flat descent of both sides, it suffices to show
that
for any invertible $R[t]$-module $\mathcal{L}$,
there exists an fppf $R$-algebra $R'$ such that the invertible $R'[t]$-module
$\mathcal{L}
\otimes_{R[t]} R'[t]$ is trivial.
In fact, there exists an $R$-algebra $\widetilde{R}$ such that
$\widetilde{R}_{\mathrm{red}}$ is perfect and $\widetilde{R}$ is a filtered
colimit of fppf $R$-algebras (for example, we can adjoin a system of $p$-power roots to each element of $R$).
Moreover, up to localizing, we can arrange $\widetilde{R}$ such that the
invertible
$\widetilde{R}$-module $(\mathcal{L}/t \mathcal{L}) \otimes_R \widetilde{R}$ is
trivial.
Then the invertible $\widetilde{R}[t]$-module $\mathcal{L} \otimes_{R[t]}
\widetilde{ R}[t]$ is trivial because $\widetilde{R}_{\mathrm{red}}$ is
seminormal \cite[Th.~1]{Swan80}.
Writing $\widetilde{R}$
as a
filtered colimit of fppf $R$-algebras, we conclude.
\end{proof}

\begin{question}[Drinfeld]
Does \Cref{localtrivialtyWbig} hold without the condition that $R$ is
$p$-nilpotent?
Equivalently, given a ring $R$ and an invertible $R[t]$-module $\mathcal{L}$,
does there exist a faithfully flat $R$-algebra $R'$ such that $\mathcal{L}
\otimes_{R[t]} R'[t] \in \mathrm{Pic}(R'[t])$ is trivial?

It is known that $R'$ cannot, in general, be taken to be an fppf $R$-algebra.
An example (due to Gabber) is given in \cite[Rem.~2.2.16]{Rosengarten}, which
we also sketch here.
Let $k$ be an algebraically closed field of characteristic zero.
We consider the ring $R = k[[u_1]] \times_{k[\epsilon]/\epsilon^2} k[[u_2]]$,
where
the maps $k[[u_1]] \to k[\epsilon]/\epsilon^2, k[[u_2]] \to
k[\epsilon]/\epsilon^2$
send $u_1, u_2 \mapsto  \epsilon$.
We consider the element of $\mathrm{Pic}( R[t])$ obtained (via Milnor excision,
cf.~\cite[Thm.~I.3.10]{WeibelKbook})
by gluing the trivial
line bundles on $k[[u_1]][t],
k[[u_2]][t]$ along the element $1 + \epsilon t \in (k[\epsilon]/\epsilon^2[t])^{\times}$.
This is a nontrivial (and non-torsion) element of $\mathrm{Pic}( R[t])$, since
$1 + \epsilon t \in ( k[\epsilon]/(\epsilon^2)[t])^{\times}$ (or any nonzero power of it)
cannot be written
as a product of elements in $(k[[u_1]][t])^{\times} = k[[u_1]]^{\times}$ and
$(k[[u_2]][t])^{\times} = k[[u_2]]^{\times}$.
Given any fppf $R$-algebra $R'$, there exists by
\cite[Tag 05WM]{stacks-project} a finite flat $R$-algebra $R''$ and an
$R$-algebra map $R' \to R''$.
But by an argument with multiplicative norms, this class in $\mathrm{Pic}(
R[t])$ also does not vanish after
passage to any finite flat cover of $R$, e.g., $R''$. Therefore, this class
cannot vanish fppf locally.
\end{question}

We now include a result, suggested to us by V.~Drinfeld, 
on the kernel of $\mathrm{Pic}(R[t]) \to \mathrm{Pic}(R)$ for a general ring $R$ (which, by \Cref{wbigtorsors}, may be thought of as a natural generalization of the notion of $\hwbig$-torsor), and reprove a result of \cite{SekiguchiandSuwa}. The rest of this subsection will not be used in the sequel.

Let \(R\) be a commutative ring. Denote by \(\Wbig\) the group scheme of big Witt vectors over \(R\). Thus, for every \(R\)-algebra \(R'\), the group \(\Wbig(R')\) is the multiplicative group of formal power series
\[
  1+tR'[[t]].
\]
Let \(\Wratplus(R')\subseteq \Wbig(R')\) be the submonoid consisting of polynomials
\[
  1+tR'[t].
\]
The functor \(R'\mapsto \Wratplus(R')\) is represented by a monoid ind-scheme \(\Wratplus\) over \(R\). More precisely, \(\Wratplus\) is the free commutative monoid, in the category of pointed ind-affine schemes over \(R\), on the pointed affine line \(\Aone\). Thus, for every commutative monoid \(M\), the monoid of homomorphisms \(\Hom(\Wratplus,M)\) is naturally identified with the monoid of pointed maps \(\Aone\to M\).

Explicitly, \(\Wratplus\) is the colimit of the diagram
\[
  \Spec R \mono \Sym^1(\Aone) \mono \Sym^2(\Aone) \mono \cdots,
\]
where the transition maps are induced by the zero section \(\Spec R\to\Aone\). Under this presentation, the monoid law on \(\Wratplus\) is induced by the maps
\[
  \Sym^n(\Aone)\times \Sym^m(\Aone)\longrightarrow \Sym^{n+m}(\Aone).
\]
Let \(\Wrat\) denote the \'etale sheafification of the presheaf that sends an \(R\)-algebra \(R'\) to the group completion of \(\Wratplus(R')\). This construction goes back to \cite{almkvist}.

Let \(\mathrm{Sh}_{R_\et}(\mathrm{Ab})\) be the category of sheaves of abelian groups on 
the big \'etale site of \emph{flat} affine schemes over \(\Spec R\).
For a commutative group scheme \(G\) over \(R\), we use the same notation for the sheaf represented by \(G\).

Consider the following Picard groupoids:
\begin{align*}
  \mathcal{C}_1(R)
  &:= \left\{\text{commutative group-scheme extensions }
  0\to \mathbb{G}_{m,R}\to G\to \Wbig\to 0\right\}, \\
  \mathcal{C}_2(R)
  &:= \left\{\text{extensions in }\mathrm{Sh}_{R_\et}(\mathrm{Ab}):
  0\to \mathbb{G}_{m,R}\to G'\to \Wrat\to 0\right\}, \\
  \mathcal{C}_3(R)
  &:= \left\{\text{multiplicative line bundles on the monoid ind-scheme }\Wratplus\right\}, \\
  \mathcal{C}_4(R)
  &:= \left\{\text{line bundles on }\Aone\text{ equipped with a trivialization along }
  \Spec R\to\Aone\right\}.
\end{align*}
The assignments \(R\mapsto \mathcal{C}_i(R)\), for \(1\leq i\leq 4\), define stacks, which we denote by \(\mathcal{C}_i\). There are evident restriction functors
\begin{equation}\label{eq:four_groupoids}
  \mathcal{C}_1(R)\longrightarrow \mathcal{C}_2(R)\longrightarrow \mathcal{C}_3(R)\longrightarrow \mathcal{C}_4(R).
\end{equation}

\begin{proposition}\label{pr:extensionsofWbig}
For every ring \(R\), all the functors in \eqref{eq:four_groupoids} are equivalences of Picard groupoids.
\end{proposition}

\begin{proof}
We first construct a quasi-inverse \(\mathcal{C}_4\to \mathcal{C}_3\). Let \(L\) be a line bundle on \(\Aone\). For each \(n\geq 0\), the line bundle \(L^{\boxtimes n}\) on \((\Aone)^n\) has a natural \(S_n\)-equivariant structure and descends uniquely along
\[
  (\Aone)^n\longrightarrow \Sym^n(\Aone)
\]
to a line bundle \(\Norm_n(L)\) on \(\Sym^n(\Aone)\).\footnote{More generally, if \(X\) is an affine \(R\)-scheme that is flat over \(R\), there is a natural functor \cite[\S 4.1]{baranovsky}
\[
  \Norm_n\colon \Pic(X)\longrightarrow \Pic(\Sym^n(X)).
\]
If \(M=\Gamma(X,L)\), then
\[
  \Gamma(\Sym^n(X),\Norm_n(L))\simeq \Gamma_R^n(M):=(M^{\otimes_R n})^{S_n}.
\]
To see that \(\Gamma_R^n(M)\) is invertible over \(\mathcal{O}(\Sym^n(X))\), fix \(y\in \Sym^n(X)\), and let \(T\subset X\) be the image, under one of the projections \(X^n\to X\), of the inverse image of \(y\) in \(X^n\). Since \(T\) is a finite set of points, there is an open neighbourhood \(U\subset X\) containing \(T\) on which \(L\) is trivial. Then \(\Norm_n(L)\) is trivial on an open neighbourhood of \(y\) in \(\Sym^n(X)\).}
A trivialization of \(L\) along the zero section induces a multiplicative structure on the resulting line bundle \(\Norm_\bullet(L)\) over \(\Wratplus\). This construction defines the required quasi-inverse \(\mathcal{C}_4\to \mathcal{C}_3\).

To see that \(\mathcal{C}_2\to \mathcal{C}_3\) is an equivalence, observe that every homomorphism from a presheaf of commutative monoids to a sheaf of strict Picard groupoids extends uniquely to the \'etale sheafification of its group completion. Apply this observation to \(\mathcal{F}=\Wratplus\) and to the sheaf that sends an \(R\)-algebra \(R'\) to the Picard groupoid \(\Pic(R')\).

It remains to prove that \(\mathcal{C}_1\to \mathcal{C}_2\) is an equivalence. We proceed in several steps.

\smallskip
\noindent\emph{Step 1: The functor \(\mathcal{C}_1\to \mathcal{C}_2\) induces an isomorphism on \(\pi_1\).}

Using the equivalence \(\mathcal{C}_2\simeq \mathcal{C}_4\), we identify \(\pi_1(\mathcal{C}_2(R))\) with the group of pointed maps \(\Aone\to \mathbb{G}_{m,R}\), namely
\[
  \hwbig(R)=\ker\bigl(R[t]^\times\xrightarrow{\,\ev_0\,}R^\times\bigr).
\]
On the other hand,
\[
  \pi_1(\mathcal{C}_1(R))=\Hom(\Wbig,\mathbb{G}_{m,R}).
\]
Thus, the claim reduces to the assertion that the map
\[
  \Hom(\Wbig,\mathbb{G}_{m,R})\longrightarrow \hwbig(R),
\]
which sends a homomorphism \(\Wbig\to \mathbb{G}_{m,R}\) to its restriction along the Teichm\"uller map \(\Aone\to\Wbig\), is an isomorphism. This is a theorem of Cartier \cite[Theorem~2]{cartier}; see also \cite[\S~37.5]{hazewinkel} for a detailed exposition.

\smallskip
\noindent\emph{Step 2: The case of reduced seminormal rings.}

If \(R\) is reduced and seminormal, then all the groupoids in \eqref{eq:four_groupoids} are trivial. Indeed, the groupoid of line bundles on an affine space over \(R\) is equivalent to \(\Pic(R)\). Hence every multiplicative line bundle on \(\Wbig\) is pulled back from \(R\), and its multiplicative structure forces it to be trivial. The automorphism groups vanish as well, since \(R[t]^\times=R^\times\). This proves that \(\mathcal{C}_1(R)\) is trivial. The same argument applies to \(\mathcal{C}_4(R)\).

\smallskip
\noindent\emph{Step 3: Reduction to rings essentially of finite type over \(\Z\).}

Recall that the functor \(R\mapsto \Pic(R)\) commutes with filtered colimits. It follows that the functors \(R\mapsto \mathcal{C}_i(R)\) commute with filtered colimits for \(i=1,4\), and therefore, by the equivalences already established, for all \(i\). Hence it is enough to prove the assertion for rings essentially of finite type over \(\Z\).

\smallskip
\noindent\emph{Step 4: Reduction to reduced rings.}

Let \(R\) be a ring, let \(I\subset R\) be a square-zero ideal, and set \(\overline{R}=R/I\). Assume that \(\mathcal{C}_1(\overline{R})\to \mathcal{C}_2(\overline{R})\) is an equivalence. 
We show that \(\mathcal{C}_1(R)\to \mathcal{C}_2(R)\) is also an equivalence. 
Let
\[
  i\colon \Spec \overline{R}\mono \Spec R
\]
be the canonical closed immersion, and write
\[
  \mathcal{I}:=\mathbb{G}_{a,R}\otimes_R I
\]
for the corresponding sheaf in \(\mathrm{Sh}_{R_\et}(\mathrm{Ab})\). The exact sequence
\[
  0\to \mathcal{I}\to \mathbb{G}_{m,R}\to i_*\mathbb{G}_{m,\overline{R}}\to 0
\]
gives a commutative diagram
\begin{center}
\footnotesize
\begin{tikzcd}[column sep=tiny,row sep=large]
\Hom(\Wbig,i_*\mathbb{G}_{m,\overline{R}}) \arrow[r] \arrow[d,"\alpha_1"] &
\Ext^1(\Wbig,\mathcal{I}) \arrow[r] \arrow[d,"\alpha_2"] &
\Ext^1(\Wbig,\mathbb{G}_{m,R}) \arrow[r] \arrow[d,"\alpha_3"] &
\Ext^1(\Wbig,i_*\mathbb{G}_{m,\overline{R}}) \arrow[r] \arrow[d,"\alpha_4"] &
\Ext^2(\Wbig,\mathcal{I}) \arrow[d,"\alpha_5"] \\
\Hom(\Wrat,i_*\mathbb{G}_{m,\overline{R}}) \arrow[r] &
\Ext^1(\Wrat,\mathcal{I}) \arrow[r] &
\Ext^1(\Wrat,\mathbb{G}_{m,R}) \arrow[r] &
\Ext^1(\Wrat,i_*\mathbb{G}_{m,\overline{R}}) \arrow[r] &
\Ext^2(\Wrat,\mathcal{I})
\end{tikzcd}
\end{center}
All Hom and Ext groups in this diagram are computed in \(\mathrm{Sh}_{R_\et}(\mathrm{Ab})\). 
We now use that the natural map $\Ext^1(W_{\mathrm{big},\overline{R}},
\mathbb{G}_{m,\overline{R}}) \to \Ext^1(W_{\mathrm{big}, R},
i_* \mathbb{G}_{m,\overline{R}})$ is an isomorphism, and similarly with \(W_{\mathrm{rat}}\) in place of \(W_{\mathrm{big}}\). This follows from the exactness of $i_*: \mathrm{Sh}_{\bar R_\et}(\mathrm{Ab}) \to \mathrm{Sh}_{R_\et}(\mathrm{Ab})$, the isomorphisms $i^* W_{\mathrm{big}, R}\iso W_{\mathrm{big}, \bar R}$
and $i^* W_{\mathrm{rat}, R}\iso W_{\mathrm{rat}, \bar R}$,
and the adjunction.
As a consequence, 
the terms involving \(i_*\mathbb{G}_{m,\overline{R}}\) are identified with the corresponding Hom and Ext groups over \(\overline{R}\).
By assumption, this means that $\alpha_1$ and \(\alpha_4\) are isomorphisms. 
By the following lemma, $\alpha_2$ and $\alpha_5$ are isomorphisms, so we conclude that 
$\alpha_3$ is an isomorphism, as desired. 

\begin{lemma}[{\cite[Proposition~A.6]{akhil}}]
The map \(\Wrat\to \Wbig\) induces an isomorphism
\[
  \RHom(\Wbig,\mathcal{I})\xrightarrow{\sim}\RHom(\Wrat,\mathcal{I}).
\]
\end{lemma}

\begin{proof}
The result is stated in \emph{loc. cit.} under the additional assumption that \(\overline{R}\) is a field. However, the proof does not use that assumption. For the reader's convenience, we sketch a slightly modified version of the argument.

First, we prove that the map \(\Wrat\to\Wbig\) induces an isomorphism
\begin{equation}\label{eq:cohofBW}
  \rGamma(B\Wbig,\mathcal{I})\xrightarrow{\sim}\rGamma(B\Wrat,\mathcal{I}).
\end{equation}
For every commutative monoid \(M\), the map from \(M\) to its group completion becomes a homotopy equivalence after applying the classifying-space functor. Hence we may replace \(\Wrat\) by \(\Wratplus\).

The next observation is that the \(R\)-coalgebra \(\mathcal{O}(\Wbig)\) is the cofree coaugmented conilpotent cocommutative coalgebra on the free \(R\)-module
\[
  V:=xR[x].
\]
In other words, there is an isomorphism  of coaugmented coalgebras  
\begin{equation}\label{eq:Witt_coalgebra}
  \mathcal{O}(\Wbig)\xrightarrow{\sim}\bigoplus_{i\geq 0}\Gamma^i(V),
\end{equation}
This follows from Cartier duality \cite[Theorem~2]{cartier}, which identifies \(\mathcal{O}(\Wbig)\) with the continuous \(R\)-linear dual of \(\mathcal{O}(\widehat{W}_{\mathrm{big},R})\). In particular, \(\mathcal{O}(\Wbig)\) is a coaugmented graded coalgebra.

The isomorphism \eqref{eq:Witt_coalgebra} has the following properties:
\begin{enumerate}
  \item The composite
  \[
    \mathcal{O}(\Wbig)\longrightarrow \bigoplus_{i\geq 0}\Gamma^i(V)\longrightarrow V,
  \]
  which is only a map of \(R\)-modules, sends \(f\in \mathcal{O}(\Wbig)\) to
  \[
    f([x])-f(0)\in V.
  \]
  This property uniquely characterizes \eqref{eq:Witt_coalgebra}.

  \item For every \(n\geq 0\), the composite
  \[
    \mathcal{O}(\Wbig)\longrightarrow \bigoplus_{i\geq 0}\Gamma^i(V)
    \longrightarrow \bigoplus_{0\leq i\leq n}\Gamma^i(V)
    \xrightarrow{\sim}\mathcal{O}(\Sym^n(\mathbb{A}^1_R))
  \]
  is a map of algebras corresponding to the closed immersion
  \[
    \Sym^n(\mathbb{A}^1_R)\mono \Wbig.
  \]
\item The map
  \[
    \mathcal{O}(\Wbig)\longrightarrow \mathcal{O}(\Wrat^+)
  \]
  identifies the target with the direct product of the graded components of the source.
\end{enumerate}

The grading on \(\mathcal{O}(\Wbig)\) induces a grading on the cobar complex computing \(\rGamma(B\Wbig,\mathcal{I})\). The map
\[
  \rGamma(B\Wbig,\mathcal{I})\longrightarrow \rGamma(B\Wrat^+,\mathcal{I})
\]
identifies the target with the direct product of the graded components of the source. On the other hand, we have\footnote{This follows by observing that $\operatorname{Cotor}^i_{\Gamma^\cdot (V)}(R, I)$, as a functor of $V$, commutes with filtered colimits of flat $R$-modules and that, for a finite free $R$-module $V$, the relevant $\operatorname{Cotor}$ is isomorphic to $\operatorname{Ext}^i_{\Sym^\cdot (V^*)}(R, I)$.}
\[
  R^i\Gamma(B\Wbig,\mathcal{I})\simeq \operatorname{Cotor}^i_{\Gamma^\cdot (V)}(R, I)\simeq \bigwedge^i V\otimes_R I.
\]
In particular, each cohomology group is supported in a single grading degree. This proves \eqref{eq:cohofBW}.

An analogous argument applies to finite products of copies of \(\Wrat^+\) and \(\Wbig\). The lemma now follows from the functorial Breen--Deligne resolution \cite[Appendix to Lecture~4]{scholze2026lecturescondensedmathematics}.
\end{proof}

 As a consequence of this step, observe that if \(A\) is a noetherian ring and the equivalence 
 $\mathcal{C}_1(R)\longrightarrow \mathcal{C}_2(R)$ is known for $R=A_{\mathrm{red}}$, then it is known for \(A\): indeed, the nilradical of \(A\) is nilpotent, and one can successively apply the result above to the square-zero extensions arising from its powers.

\smallskip
\noindent\emph{Step 5: Milnor patching.}

The sheaves \(\mathcal{C}_i\), for \(1\leq i\leq 4\), satisfy Milnor patching \cite[\S~2]{milnor}. Namely, given a Cartesian square
\[
\begin{tikzcd}
A \arrow[r] \arrow[d] & B \arrow[d] \\
A' \arrow[r] & B'
\end{tikzcd}
\]
in which at least one of \(A'\to B'\) or \(B\to B'\) is surjective, the induced square of groupoids is a pullback square:
\[
  \mathcal{C}_i(A)\simeq \mathcal{C}_i(B)\times_{\mathcal{C}_i(B')}\mathcal{C}_i(A'),
  \qquad 1\leq i\leq 4.
\]
For \(\mathcal{C}_2\), this is the usual Milnor patching theorem for invertible modules. For \(\mathcal{C}_1\), it follows from the description in terms of multiplicative line bundles on \(\wbig\) together with patching for projective modules.

\smallskip
\noindent\emph{Step 6: Reduction to the seminormal case.}

It remains to prove that
\[
  \mathcal{C}_1(R)\longrightarrow \mathcal{C}_2(R)
\]
is an equivalence for every reduced ring \(R\) essentially of finite type over \(\Z\).  Denote by \(\operatorname{Min}(R)\) the set of irreducible components of \(\Spec R\), i.e., minimal prime ideals of \(R\).  We argue by induction on the lexicographically ordered pair
\[
  \bigl(\dim R,\#\operatorname{Min}(R)\bigr).
\]
Suppose first that \(R\) has more than one irreducible component. Partition \(\operatorname{Min}(R)\) into two nonempty subsets, and let \(I_1\) and \(I_2\) be the intersections of the minimal primes in the two subsets. Since \(R\) is reduced, \(I_1\cap I_2=0\), and hence
\[
\begin{tikzcd}
R \arrow[r] \arrow[d] & R/I_1 \arrow[d] \\
R/I_2 \arrow[r] & R/(I_1+I_2)
\end{tikzcd}
\]
is a Milnor square. The rings \(R/I_1\) and \(R/I_2\) are reduced and have fewer minimal primes than \(R\). Moreover,
\[
  \dim \bigl(R/(I_1+I_2)\bigr)<\dim R.
\]
Using the remark at the end of Step~4 and the induction hypothesis, the equivalence holds for $R/(I_1+I_2)$. 
Thus, Milnor patching reduces the assertion to the case in which \(R\) is a domain.

Now suppose that \(R\) is a domain, and let \(R^+\) be its normalization. Since \(R\) is essentially of finite type over \(\Z\), the ring \(R^+\) is finite over \(R\). Let
\[
  \mathfrak{c}:=\{a\in R\mid aR^+\subset R\}
\]
be the conductor ideal. It is a nonzero ideal of both \(R\) and \(R^+\), and the conductor square
\[
\begin{tikzcd}
R \arrow[r] \arrow[d] & R^+ \arrow[d] \\
R/\mathfrak{c} \arrow[r] & R^+/\mathfrak{c}
\end{tikzcd}
\]
is Cartesian. It is therefore a Milnor square. The ring \(R^+\) is normal, hence seminormal, so the equivalence is known by Step~2. Since \(\mathfrak{c}\neq 0\), we have \(\dim(R/\mathfrak{c})<\dim R\). Moreover, \(R^+/\mathfrak{c}\) is finite over \(R/\mathfrak{c}\), so it also has dimension strictly smaller than \(\dim R\). Applying the induction hypothesis to the reductions of the two quotient rings, and then Step~4 to recover the quotients themselves, Milnor patching yields the equivalence for \(R\).
\end{proof}

\begin{remark}
The ring \(\wbig(R)\) acts naturally on the group scheme \(\Wbig\). This induces an action of \(\wbig(R)\) on the Picard groupoid \(\mathcal{C}_1(R)\), and hence, by \Cref{pr:extensionsofWbig}, on \(\mathcal{C}_4(R)\); cf.~\cite{daytonweibel}.
\end{remark}
\begin{remark}
   For any ring $R$, every commutative group-scheme extensions 
  \( 0\to \mathbb{G}_{a,R}\to G\to \Wbig\to 0 \) splits. In fact, this follows from \Cref{pr:extensionsofWbig} applied to $R$ and 
    its trivial square-zero extension $R\oplus \epsilon R$.
\end{remark}

\begin{corollary}
\label{cor:extensionsWbigtorsors}
Assume that \(R\) is \(p\)-nilpotent. Then the Picard groupoid \(\mathcal{C}_1(R)\) is equivalent to the groupoid of \(\hwbig\)-torsors over \(R\) for the fpqc topology.
\end{corollary}

\begin{corollary}[{\cite[Theorem~2.8.1]{SekiguchiandSuwa}}]
Let \(R\) be a \(\mathbb{Z}_{(p)}\)-algebra. For an integer \(n>0\), let \(\mathcal{C}_5(R)\) be the Picard groupoid of commutative group-scheme extensions
\[
  0\to \mathbb{G}_{m,R}\to G\to W_{n,R}\to 0
\]
that admit a scheme-theoretic section \(W_{n,R}\to G\). Then there is a natural equivalence
\[
  \cone\bigl(\hw(R)\xrightarrow{F^n}\hw(R)\bigr)\simeq \mathcal{C}_5(R).
\]
Here the cone is viewed as the Picard groupoid associated to the resulting two-term complex.
\end{corollary}

\begin{proof}
For every \(\mathbb{Z}_{(p)}\)-algebra \(R\), the morphism \(\Wbig\to W_R\) has a natural section. It follows from \cref{pr:extensionsofWbig} that, for every extension \(G\) in \(\mathcal{C}_5(R)\), its pullback along the projection
\[
  W_R\longrightarrow W_{n,R},
\]
which factors through \(\Wbig\) as noted above, admits a splitting. The groupoid of such extensions \(G\), equipped with a choice of splitting, is equivalent to \(\hw(R)\) by Cartier duality. The result follows because Cartier duality interchanges \(V\) and \(F\).
\end{proof}

\subsection{The functor $Q$}
\label{sec:Q}

\begin{definition}
Let $R \in \nperf$ be a nilperfect ring.
We let $Q(R) = W(R)/\hw(R)$.
\end{definition}

\begin{remark} 
It will occasionally be helpful to consider $Q$ on more general $p$-nilpotent rings. 
Since $\nperftors$ is a basis for the fpqc topology on $p$-nilpotent rings, we can extend $Q$ to a sheaf on all $p$-nilpotent rings by fpqc descent. For any $p$-nilpotent $R$, $W(R)/\hw(R) \subset Q(R)$. 
\end{remark}

Let $R \in \nperftors$.
Since $\hw(R)$ is generally not derived $p$-complete, neither is $Q(R)$.
Nonetheless, one has the following result.

\begin{proposition}
For any $R \in \nperftors$,     the kernel of the zeroth ghost
coordinate map $Q(R) \to R_{\mathrm{red}}$ is a henselian ideal \cite[Tag
09XD]{stacks-project}.
In particular, $Q(R)$ is $p$-henselian.
\end{proposition}

\begin{proof}
Say that a surjection of rings is \emph{henselian} if the kernel is henselian as
an ideal of the source.
We now use
\cite[\href{https://stacks.math.columbia.edu/tag/0DYD}{Lemma
0DYD}]{stacks-project} (restated slightly): given surjections $A
\twoheadrightarrow B
\twoheadrightarrow C$, then the map $A \twoheadrightarrow C$ is henselian if and
only if $A \twoheadrightarrow B$ and $B \twoheadrightarrow C$ are henselian.
In particular, any ideal contained in a henselian ideal is itself a henselian
ideal.

For any $p$-nilpotent ring $R$,
the maps $W(R) \to R$ and $ R \to R_{\mathrm{red}}$ are henselian
surjections\footnote{In the first case, this follows because
$W(R)$ is complete with respect to the ideal $VW(R)$. We use that $p \in R$ is
nilpotent to guarantee that for any $n$, there exists $N$ such that $
(VW(R))^N \subset V^n W(R)$.} and
therefore so is their composite.
This also implies that $Q(R) \to R_{\mathrm{red}}$ is a henselian surjection,
and since
an ideal contained in a henselian ideal is also henselian, it follows
that $(p) \subset Q(R)$ is henselian.
\end{proof}

\begin{proposition}
\label{Qisfpqcsheaf}
On the category $\nperftors$, $Q$ is a sheaf of rings for the fpqc topology.
\end{proposition}
\begin{proof}
This follows because $W$ and $\hw$ are fpqc sheaves, and by
\Cref{vanishingofhwtorsors}. 
\end{proof}

Next, we discuss some of the structure of $Q$.

\begin{construction}[$F, V, \delta$ on $Q$]
Since $\hw \subset W$ is stable under $F$ and $V$,
these operators
descend to operators $F, V: Q \to Q$.
Since $\hw \subset W$ is a $\delta$-ideal, $Q$ acquires the structure of a sheaf
of $\delta$-rings on $\nperftors$.
\end{construction}

\begin{proposition} \label{VonQ}
If $R \in \nperftors$, then we have a natural short exact sequence
\[ 0 \to Q(R) \stackrel{V^n}{\to} Q(R) \to W_n( R_{\mathrm{red}}) \to 0.  \]
\end{proposition}
\begin{proof}
	    Injectivity of $V : Q(R) \to Q(R)$ follows because an element of
	 $W(R)$ belongs to $\hw(R)$ if and only if its image under $V$ does.
    The cokernel of $V^n: W(R)/\hw(R) \to W(R)/\hw(R)$ is
    $W_n(R_{\mathrm{red}})$. \end{proof}

It turns out
that $F$ and $V$ commute on $Q$ for $p > 2$, and come close to commuting when $p
= 2$; in this respect $Q$ (on $\nperftors$) behaves more like the Witt vectors
of an
$\mathbb{F}_p$-algebra. To formulate and prove these claims, we use the
following definition.

\begin{definition}[The operator $\hV$ on $W$ and $Q$]
For any ring $R$, define an additive operator $\hV:W(R)\to W(R)$ by
\[
\hV(x)=
\begin{cases}
V(x), & p>2,\\
V([-1]x), & p=2
\end{cases}. 
\]
Note also that for an $\mathbb{F}_2$-algebra
we  have $\hV=V$.
Since
$\hw(R)\subset W(R)$ is stable under $V$ and under multiplication by $[-1]$,
$\hV$ carries $\hw(R)$ into itself. 
Thus, for $R\in\nperftors$, it
descends to an additive operator $\hV:Q(R)\to Q(R)$.
\end{definition}

\begin{remark}
The modified Verschiebung $\hV$ appears in
\cite[Intro., p.~2202 and Lem.~1.7]{Laurelations}. By \Cref{pminusV1} for $p>2$ and \Cref{uinWZ2} for $p=2$, if
\[
\epsilon =
\begin{cases}
1, & p>2,\\
[-1], & p=2,
\end{cases}
\]
then $V(\epsilon)=p$ in $Q(\mathbb{Z}_p)=\varprojlim_n Q(\mathbb{Z}/p^n)$.
Equivalently, on $Q$ the preceding definition is the formula
$\hV(x)=V(\epsilon x)$ with $V(\epsilon)=p$. The element $\epsilon=[-1]$ is the
necessary correction at $p=2$, where $V(1)$ itself does not agree with $2$ in
$Q(\mathbb{Z}_2)$. 
\end{remark}

\begin{proposition}
\label{VonQseqprop}
For $R \in \nperftors$,
we have a natural short exact sequence
\begin{equation} \label{VonQseq}  0 \to Q(R) \stackrel{\hV^n}{\to} Q(R) \to
W_n(R_{\mathrm{red}})\to 0.  \end{equation}
\end{proposition}
\begin{proof}
This follows from \Cref{VonQ}, because $\hV$ is $V$ precomposed with an
automorphism.
\end{proof}

\begin{proposition}
\label{FhatV}
    The operators $F, \hV: Q \to Q$ commute and we have $F\hV = \hV F = p$.
	 Moreover, for any $R \in \nperftors$, we have the projection formula $x \hV(y)
	 = \hV( F(x) y)$ for any
	 $x, y \in Q(R)$.
\end{proposition}
\begin{proof}
Let
\[
\epsilon =
\begin{cases}
1, & p>2,\\
[-1], & p=2,
\end{cases}
\]
viewed as an element of \(Q(\mathbb{Z}_p)\). By \Cref{pminusV1,uinWZ2}, we have
\(V(\epsilon)=p\) in \(Q(\mathbb{Z}_p)\), hence in \(Q(R)\) for every
\(R\in\nperftors\). Applying \(F\), we get \(p\epsilon=FV(\epsilon)=F(p)=p\).
Therefore
\[
F\hV(x)=FV(\epsilon x)=p\epsilon x=px.
\]
Moreover, by the projection formula for \(V\),
\[
\hV F(x)=V(\epsilon F(x))=xV(\epsilon)=px.
\]
The same projection formula gives
\[
x\hV(y)=xV(\epsilon y)=V(F(x)\epsilon y)=\hV(F(x)y),
\]
as desired.
\end{proof}

\begin{proposition}\label{Visinvertibleon Qinfinity}
For any $R \in \nperftors$, the kernel of $Q(R) \to Q(R_{\mathrm{red}}) = W(
R_{\mathrm{red}})$ agrees
with
$\bigcap_{i \geq 0} \hV^i Q(R)$, and $\hV$ acts invertibly on this ideal.
\end{proposition}
\begin{proof}
An element of $Q(R)$ represented by $x \in W(R)$ belongs to either
ideal if and only if the Witt components of $x$ are all nilpotent.
The last assertion is clear.
\end{proof}
\begin{proposition}
\label{ppowertorsioninQ}
Let $R$ be any $p$-nilpotent ring.
    An element  $x \in Q(R)$ satisfies $p^n x = 0$ if and only if $F^n(x) = 0$.
The ideal of torsion elements in $Q(R)$ squares to zero.
	 More generally, if $x, y \in Q(R)$ are such that $x$ is $p$-power torsion
	 and
	 $y \in
\bigcap_{i \geq 0} \hV^i (Q(R))$, then $xy = 0$.
\end{proposition}
\begin{proof}
By fpqc descent, we may assume that $R \in \nperftors$. Then this is a consequence of 
   \Cref{FhatV}, \Cref{VonQseqprop}, 
     and \Cref{productsinhw}. 

\end{proof}
\begin{proposition}[$\delta$ on $Q$]
\label{deltaonQ}
Let $R \in \nperftors$.
For any $x \in Q(R)$ which is  $p$-power torsion,  one has $\hV( \delta(x)) =
x$.
As a result, $\delta$ and $\hV$ induce mutually inverse isomorphisms from
$Q(R)[p^\infty]$ to itself.
\end{proposition}
\begin{proof}
If $x \in \mathrm{ker}(Q(R) \to Q(R_{\mathrm{red}}))$, then
$x  = \hV (y)$ for a unique $y \in
\mathrm{ker}(Q(R) \to Q(R_{\mathrm{red}}))$ by \Cref{Visinvertibleon Qinfinity}.
Necessarily $y $ is also $p$-power torsion since $\hV$ is injective.
We claim next that $\hV(y) = V(y)$.
This is clear when $p > 2$,
and for $p = 2$ it follows because
	$[-1] y = y$, since $[-1]-1 \in Q(R)$ is 2-power torsion and products of
	2-power torsion elements vanish by \Cref{ppowertorsioninQ}.

We observe now the formula
$\delta(V(z)) =
z - p^{p-2} V( z^p)$ which holds for all $z \in W(A)$ for any ring $A$  (by
reduction to the
$p$-torsionfree case\footnote{See \S\ref{sec:dcartrings} for more discussion
of this.}), and therefore in $Q(R)$.
Since $x = V(y)$ as well, we find
$\delta(x) = y - p^{p-2} V( y^{p})$.
Since $y$ is $p$-power torsion, we find that $y^2 = 0$  by
\Cref{ppowertorsioninQ}
and
therefore $\delta(x) = y$. The last assertion follows because $\hV:
Q(R)[p^\infty] \to Q(R)[p^\infty]$ is an isomorphism, cf.~\Cref{Visinvertibleon
Qinfinity}.
\end{proof}

\begin{definition}[The functor $\gabar$]
On the category of all $p$-nilpotent rings,
    we write $\gabar$ for the functor carrying  $R$ to $(R/p)_{\mathrm{perf}}$.
We write $\gahat(R)$ for the nilradical of $R$.
	 Note that if $R \in \nperftors$,  then $\gabar(R) =
	 R_{\mathrm{red}}$.
   \end{definition}
	\begin{proposition}
	\label{gabar:basics}
	$\gabar$ is a sheaf of abelian groups for the fpqc topology on the category of $p$-nilpotent
	rings, with no higher cohomology.
We have a short exact sequence of fppf  sheaves
\begin{equation} \label{gabarex} 0 \to \gahat \to \mathbb{G}_a \to \gabar \to
0.  \end{equation}
Equivalently, $\gabar$ is the fppf (or fpqc) sheafification of the presheaf
	$R \mapsto R_{\mathrm{red}}$.
	\end{proposition}
	\begin{proof}
	First, we prove that the functor
	$R \mapsto (R/p)_{\mathrm{perf}}$ is a sheaf for the fpqc topology
	with no higher cohomology.
	In other words, if $R \to R'$ is a faithfully flat map of $p$-nilpotent
	rings, we need to see that in $D(\ZZ)$,
\begin{equation}  (R/p)_{\mathrm{perf}} \xrightarrow{\sim} \varprojlim \left(( R'/p)_{\mathrm{perf}}
\rightrightarrows ((R' \otimes_R R')/p)_{\mathrm{perf}} \triplearrows \dots
\right).
\end{equation}
	However, this follows because $R/p \to R'/p$ is faithfully flat, which
	implies
	by flat descent that the
	natural map induces an equivalence in $D(\ZZ)$,
\[ R/p \xrightarrow{\sim} \varprojlim \left(( R'/p)
\rightrightarrows ((R' \otimes_R R')/p) \triplearrows \dots
\right), \]
and one can take the filtered colimit along Frobenius. This commutes with the
limit along the simplex category $\Delta$ because the \v{C}ech complexes are
discrete and filtered colimits are exact.

Finally, the short exact sequence
\eqref{gabarex}
follows because
the map $\mathbb{G}_a \to \gabar$ is surjective in the fppf topology (by adding
$p$-power roots to elements), and the kernel is exactly $\gahat$.
	\end{proof}

\begin{proposition}
On the category $\nperftors$,
    \begin{enumerate}
        \item
        The natural map $Q/p \to \gabar$ induces an isomorphism
		  after fpqc sheafification.\footnote{It suffices to use the faithfully
		  flat and countably presented topology, so there are no set-theoretic
		  issues.}
        \item The natural map $\gasharp \to W$ induces an isomorphism
		  of fpqc sheaves
between        $\gasharp/\gahatsharp $ and the $p$-torsion $Q[p] \subset Q$.
    \end{enumerate}
\end{proposition}
\begin{proof}
Since $\hV F  = p$ and $F$ is surjective fpqc locally on $Q$ (even on $W$) by
\cite[Prop.~3.4.7]{bhattlurieAPC}, we find that in the category of fpqc sheaves, $Q/p = Q/\hV = Q/V = \gabar$ by
\Cref{VonQ}.
Similarly, we find that $Q[p] $ is the kernel of the surjection of sheaves $F: Q
\to Q$.
We use the diagram of short exact sequences
of sheaves, cf.~\cite[Lem.~3.2.6]{drinfeld} and \Cref{hwFgahatsharp},
\[ \xymatrix{
0 \ar[r] &  \ar[d]  \gahatsharp \ar[r] &\ar[d]   \hw \ar[r]^F &  \ar[d]  \hw
\ar[r] &   0 \\
0 \ar[r] &  \gasharp \ar[r] &  W \ar[r]^F & W  \ar[r] &  0
}\]
to conclude.
\end{proof}

As a consequence, $Q /^{\mathbb{L}} p$ defines a sheaf of (1-truncated) animated rings on $\nperf$, whose $\pi_0$ is $\gabar$ and whose $\pi_1$ is $\gasharp/\gahatsharp$.

\begin{definition}[The functor $\qperf$]
On the category $\nperftors$,
we define the functor $\qperf$ via $\qperf  = \varprojlim_F Q$.

Note that $\qperf$ is an fpqc sheaf of $\delta$-rings  by \Cref{Qisfpqcsheaf}.
Since $\hV$ commutes with $F$ on $Q$ (\Cref{FhatV}), $\hV$ naturally acts on
$\qperf$ commuting
with the projection map $\qperf \to Q$.\footnote{By contrast, when $p = 2$, $V$
does not naturally
act on $\qperf$, unless one restricts to $\mathbb{F}_2$-algebras.}
\end{definition}

\begin{remark}
For any $R \in \nperftors$,
$\qperf(R)$ is $p$-torsionfree. In fact, this follows because $\qperf(R)$ is a
perfect $\delta$-ring,
cf.~\Cref{frobinjtorsionfree}.
\end{remark}

\begin{proposition}
\label{nilinvarianceofqperf}
Let $R \in \nperftors$. Let $I \subset
R$ be an ideal such that there
exists $N \in \mathbb{N}$ such that $x^N = 0$ for all $x \in I$.
Then
$\qperf(R) \xrightarrow{\sim} \qperf(R/I)$.
\end{proposition}
\begin{proof}
In this case, we observe that
a power of
$F$ acts by zero on $\mathrm{ker}(W(R) \to W(R/I))$ and $\mathrm{ker}(
\hw(R) \to  \hw(R/I))$ by \Cref{torsioninW},  whence the claim.
\end{proof}

\begin{proposition}
\label{chWlargeneough}
\label{QperfmodV}
Let $R \in \nperftors$.
Then the maps $\qperf(R)/p \to \qperf(R)/\hV \qperf(R) \to R_{\mathrm{red}}$ are
isomorphisms.
\end{proposition}

\begin{remark}
The following (elementary but slightly involved) argument will  be used in
the sequel to show
that for $R \in \nperf$, $\chW(R)/\hV \chW(R) \xrightarrow{\sim} R$.
\end{remark}

\begin{proof}
By \Cref{nilinvarianceofqperf}, we may assume that $R$ is an
$\mathbb{F}_p$-algebra, so that Frobenius and Verschiebung commute. In this case $\hV=V$: this is clear for $p>2$ by
definition, and for $p=2$ it follows because $[-1]=[1]$ in $W(R)$. We will
therefore write $V$ in place of $\hV$ throughout the rest of the proof.
Since $p =  VF$ and $F: \qperf(R) \to \qperf(R)$ is an isomorphism,
$\qperf(R)/p = \qperf(R)/V$.
Taking $n = 1$ in \Cref{VonQ}, we know that $Q(R)/V =
	R_{\mathrm{red}}$.
	It follows from this, and the diagram of exact sequences
	\[ \xymatrix{
	0 \ar[r] &  Q(R) \ar[r]^{V}  \ar[d]^F  &  Q(R)  \ar[d]^F  \ar[r] &
	R_{\mathrm{red}} \ar[r] \ar[d]^{\phi}  &  0 \\
	0 \ar[r] &  Q(R) \ar[r]^{V} & Q(R) \ar[r] &  R_{\mathrm{red}} \ar[r] &  0
	}\]
		(where $\phi: R_{\mathrm{red}} \xrightarrow{\sim} R_{\mathrm{red}}$ is the
	Frobenius)
	that the map $\qperf(R)/V \to R_{\mathrm{red}}$ is injective.
	It remains to prove surjectivity.

	We use the following approximation observation: if $x \in W(R)$ and $n\geq 1$,
	then there exist $x' \in W(R)$ and $y \in V^n W(R)$ such that
	\[
	F(x')-x-y \in \hw(R).
	\]
	Indeed, modulo $\hw(R)$, and then modulo $V^n Q(R)$, this asks to solve
	$F(\overline{x'})=\overline{x}$ in
	$Q(R)/V^n Q(R) \xrightarrow{\sim} W_n(R_{\mathrm{red}})$, which is possible
	because $R_{\mathrm{red}}$ is perfect. The remaining error may then be
	represented by an element of $V^nW(R)$.

	Fix an element $a \in R$.
	Inductively, we choose a sequence of elements $y_0, y_1, y_2, \dots \in W(R)$
	and $z_0, z_1, z_2, \dots \in W(R)$ such that:
	\begin{enumerate}
	\item
	$z_0 = [a]$.
	\item $F(z_{i+1}) - z_i - y_{i} \in \hw(R) $ for all $i \geq 0$.
	\item $y_{i} \in V^{i+1} W(R)$ for all $i \geq 0$.
	\end{enumerate}
	To construct these sequences, we start with $z_0 = [a]$ and then apply the
	preceding observation with $x=z_i$ and $n=i+1$ to construct $z_{i+1}$ and
	$y_i$.

We then consider the sequence
of elements in $W(R)$:
\begin{gather*}
v_0 = z_0 + y_0 + F( y_1) + F^2 (y_2) + F^3( y_3) + \dots \\
v_1 = z_1 + y_1 + F(y_2) + F^2(y_3) + F^3(y_4) +   \dots \\
v_2 =  z_2 + y_2 +
F(y_3) + F^2(y_4) + F^3(y_5) + \dots \\
\dots \\
v_i = z_i + \sum_{j \geq i} F^{j-i}(y_j).
\end{gather*}

	Note that each of these sums is convergent in $W(R)$ since $y_j \in V^{j+1}W(R)$, 
and $F$ preserves the ideal $V^{j+1}W(R)$ (as $R$ is an $\mathbb{F}_p$-algebra).

We observe that
$F(v_{i+1})- v_i \in \hw(R)$ for each $i \geq 0$.
	Therefore, the sequence $v_0, v_1,
	\dots$ defines  an element of $\qperf(R)$.
	Moreover, $y_0 \in VW(R)$, and for $j \geq 1$,
	$F^j(y_j) \in  VW(R)$, so every correction term in the
	definition of $v_0$ has vanishing zeroth Witt component. Therefore the zeroth
	Witt component of $v_0$ is that of $z_0 = [a]$, namely $a$. Since $a$ was
	arbitrary, this implies that
	$\qperf(R)/V\to R_{\mathrm{red}}$ is surjective, hence an isomorphism.
\end{proof}

\begin{definition}
On the category $\nperftors$,
define the fpqc sheaf $T_F(Q)$ (the ``$F$-Tate module of $Q$'') as
\[ T_F(Q) = \varprojlim \left( \dots \to Q[F^3] \stackrel{F}{\to} Q[F^2]
\stackrel{F}{\to} Q[F] \right) \]
where $Q[F^i] \subset Q$ is the kernel of $F^i$.
\end{definition}
Thus, we have a short exact sequence of fpqc sheaves
on $\nperftors$,
\begin{equation} \label{qperftoQseq}  0 \to T_F(Q) \to \qperf \to Q \to 0.  \end{equation}

\begin{proposition}
\label{qperfsquarezero}
The kernel $T_F(Q)$ of the surjection $\qperf \to Q$ of fpqc sheaves on $\nperf$
is a square-zero ideal in $\qperf$.
\end{proposition}
\begin{proof}
This follows from \Cref{ppowertorsioninQ}, which implies that $F^n: Q \to Q$ is
a square-zero extension of sheaves for each $n$.
\end{proof}

\subsection{Definition of $\chW$ on $\nperf$}

\label{sec:Wforpnilpotent}

\begin{definition}[Cf.~\Cref{introshearedW}]
For $R \in \nperftors$,
we define the \emph{sheared Witt vectors} $\chW(R)$ via the formula
\begin{equation} \label{chWdef} \chW(R) = W(R) \times_{Q(R)}
\qperf(R).
\end{equation}
\end{definition}

\begin{remark}\label{chWisfpqcsheaf}
Thus, $\chW$ is a sheaf of $\delta$-rings for the fpqc topology on
$\nperftors$ with Frobenius lift $F: \chW \to \chW$. Explicitly, an element of
$\chW(R)$ consists of an element $x \in W(R)$ and a
compatible system of lifts through powers of $F$ of the image $\overline{x} \in
Q(R)$.
\end{remark}

For $R \in \nperftors$, we have by definition (and $Q(R) = W(R)/\hw(R)$) a
natural short exact sequence
\begin{equation} \label{secondexactseq}  0 \to \hw(R) \to \chW(R) \to
\qperf(R) \to 0.
\end{equation}

\begin{definition}[$\chW$ for $p$-completely nilperfect rings with bounded $p$-power torsion]
    We say that a derived $p$-complete ring 
    $R$ is \emph{$p$-completely nilperfect} if  $(R/p)_{\mathrm{red}}$ is perfect.

    Let $R$ be 
   a $p$-completely nilperfect ring such that $R$ additionally has bounded $p$-power torsion. 
    We define
    \begin{equation} 
 \chW(R) \stackrel{\mathrm{def}}{=} \varprojlim_{n } \chW(R/p^n),
    \end{equation} 
 which we consider as a $\delta$-ring. 
\end{definition}

\begin{proposition}\label{chWtoWtautsquarezero}
We have a short exact sequence of fpqc sheaves
on
$\nperftors$,
\begin{equation} \label{firstSES} 0 \to T_F(Q) \to \chW \to W \to 0.
\end{equation}
Moreover, $T_F(Q) \subset \chW$  is a square-zero ideal, stable under $\delta$,
and $\delta$ induces an automorphism of $T_F(Q)$.
\end{proposition}

\begin{proof}
By \Cref{qperfsquarezero}, the map
$\qperf \to Q$ is surjective locally in the fpqc topology and its kernel $T_F(Q)$
is a
square-zero ideal in $\qperf$. Moreover, $\delta$ induces an automorphism
of the kernel $T_F(Q)$ (indeed, an inverse to $\hV$) by \Cref{deltaonQ} and
\Cref{ppowertorsioninQ}.
By pulling back along the map $W \to Q$, we obtain the result.
\end{proof}

\begin{proposition}
For any  $R \in \nperftors$,
the ring $\chW(R)$ is derived $p$-complete.
\end{proposition}
\begin{proof}
This follows from \eqref{firstSES} because
$T_F(Q)$ is the inverse limit of the terms $Q[F^n]$ (annihilated by $p^n$ since
$\hV^n F^n = p^n$ on $Q$) and
is therefore derived $p$-complete as a sheaf; for this sheaf-theoretic form of
derived $p$-completeness, see \cite[App.~B.2, especially Lem.~B.2.2]{Dri25}.
Since $W$ has no higher cohomology, we have an exact sequence
\[ 0 \to T_F(Q)(R) \to \chW(R) \to W(R) \to H^1_{\mathrm{fpqc}}(R, T_F(Q)) .  \]
Since $T_F(Q)$ is a derived $p$-complete sheaf,
$H^0_{\mathrm{fpqc}}(R, T_F(Q))$ and $H^1_{\mathrm{fpqc}}(R, T_F(Q))$
are derived $p$-complete abelian groups by \cite[Lem.~B.2.2]{Dri25}. Moreover,
$W(R)$ is classically $p$-complete. The exact sequence now implies that $\chW(R)$ is derived $p$-complete as well, since derived $p$-complete abelian groups are closed under kernels, cokernels, and extensions.
\end{proof}
As a consequence, we can obtain a more convenient version of the exact sequence
\eqref{secondexactseq}.
\begin{proposition}
Let $R \in \nperftors$.
We have a natural short exact sequence
\begin{equation}
\label{thirdexactseq}
0 \to \hw(R)_{\hat{p}} \to \chW(R) \to W( R_{\mathrm{red}}) \to 0.
\end{equation}
\end{proposition}
\begin{proof}
This sequence is obtained by applying derived $p$-completion to the
short exact sequence
\eqref{secondexactseq}, since we know that $\chW(R)$ is derived $p$-complete.
Note that the $p$-adic Tate module of $\hw(R) \subset W(R)$ vanishes, which
implies that \eqref{secondexactseq} remains short exact after applying
derived $p$-completion.
Now $\qperf(R)/p \xrightarrow{\sim} R_{\mathrm{red}}$ thanks to
\Cref{chWlargeneough}, and since $\qperf(R)$ is a perfect $\delta$-ring, it
follows that $\qperf(R)_{\hat{p}} \simeq W( R_{\mathrm{red}})$. The result
follows.
\end{proof}

\begin{corollary}
    \label{chWmodpncommuteswithfilteredcolimits}
    The functor $R \mapsto \chW(R)/^\mathbb{L} p^n$ from $\nperftors$ to animated rings commutes with filtered colimits. 
\end{corollary}
\begin{proof} Without loss of generality, we may take $n = 1$.
Applying $- \otimes_{\mathbb{Z}}^{\mathbb{L}} \mathbb{F}_p$ to
\eqref{thirdexactseq}, we get a natural fiber sequence
\[
\hw(R)/^\mathbb{L}p \to \chW(R)/^\mathbb{L}p \to R_{\mathrm{red}}.
\]
We use that  $W(R_{\mathrm{red}})/^\mathbb{L}p \simeq R_{\mathrm{red}}$ since
$R_{\mathrm{red}}$ is perfect.

Let $R=\varinjlim_i R_i$ be a filtered colimit in $\nperftors$. The left term
commutes with filtered colimits because $\hw$ does, by
\Cref{Mhatcommuteswithfilteredcolimits}, and because
$- \otimes_{\mathbb{Z}}^{\mathbb{L}} \mathbb{F}_p$ commutes with colimits. The
right term clearly commutes with filtered colimits. The result follows. 
\end{proof}

\begin{proposition}
\label{FonchW}
We have a short exact sequence of fpqc sheaves
on $\nperftors$,
\begin{equation}
0 \to \gahatsharp \to \chW \xrightarrow{F} \chW \to 0.
\end{equation}
Moreover, the natural map induces an isomorphism
$\varprojlim_F \chW \xrightarrow{\sim} \varprojlim_F W$.
\end{proposition}
\begin{proof}
The first claim follows from the short exact sequence of sheaves
of \eqref{secondexactseq} and \Cref{hwFgahatsharp}.
The second claim now follows from the definition of $\chW$.
\end{proof}

\begin{proposition}
As an fpqc sheaf on $\nperftors$,
$\chW$ is the quotient of $W^{\mathrm{perf}} = \varprojlim_F W$ by the ideal
consisting of
sequences $x_0, x_1,  x_2 , \dots \in W$ such that:
\begin{enumerate}
\item
$F(x_{i+1}) = x_i$ for all $i$ (so the sequence defines an element of
$W^{\mathrm{perf}}$).
\item $x_0 = 0$.
\item  $x_i \in \hw$ for all $i$.
\end{enumerate}
Equivalently, we have a short exact sequence  of fpqc sheaves on $\nperf$,
\begin{equation}
\label{WperfchWsheafseq}
  0 \to T_F( \hw) \to W^{\mathrm{perf}} \to \chW \to 0,
\end{equation}
where $T_F( \hw)$ denotes the $F$-Tate module of $\hw$.
\end{proposition}
\begin{proof}

In fact, we have an evident map $W^{\mathrm{perf}} \to \chW$ (arising from the
forgetful map $W^{\mathrm{perf}} \to W$ and the quotient map $W^{\mathrm{perf}}
\to \qperf$).
This map induces an isomorphism
$W^{\mathrm{perf}} \to \varprojlim_F \chW$ and thus, by \Cref{FonchW} and
repleteness of the fpqc site (cf.~\cite[Sec.~3.1]{BS15}), a
surjection
$W^{\mathrm{perf}} \to \chW$. By definition, the kernel of $W^{\mathrm{perf}}
\to \chW$ is precisely $T_F(\hw)$, whence the result.
\end{proof}

\begin{example}
Let $R$ be a semiperfect $\mathbb{F}_p$-algebra. Then $\chW(R)$ admits the following presentation. 
Write $R$ as the quotient of the perfect $\mathbb{F}_p$-algebra
    $R^\flat=  \varprojlim_\phi R$ by an ideal $I$. 
    Endow $I= \varprojlim_{n,\phi}  R[\phi^n]$ with the inverse limit topology. 
    Then   $\chW(R)$ is the quotient of $W(R^\flat)$ by the ideal $J$ which consists of all
    $\sum_{n\geq 0} V^n([a_n])\in W(I)$ such that each $a_n$ is topologically nilpotent and 
	    $\lim _{n\to \infty } a_n=0$. This follows readily from \eqref{WperfchWsheafseq}.
\end{example}

For the next result, we will say that an ideal $I \subset R$
equipped with a divided power structure $\left\{\gamma_i: I \to I\right\}_{i
> 0}$ has \emph{uniformly nilpotent} divided
powers if there exists $N \in \mathbb{N}$ such that $\gamma_i(x) = 0$ for all $x
\in I$ and all $i > N$.

\begin{proposition}
Let $R \in \nperf$  and $I \subset R$ an ideal equipped with
uniformly nilpotent divided powers. Then there is a natural short exact sequence
\[ 0 \to  I^{\oplus \mathbb{N}}  \to \chW(R) \to \chW( R/I) \to 0.  \]
\end{proposition}
\begin{proof}
This follows from \eqref{secondexactseq}, together with the
fact that $\qperf(R) \xrightarrow{\sim} \qperf(R/I)$
(\Cref{nilinvarianceofqperf}) and the
isomorphism
$\mathrm{ker}( \hw( R) \to \hw(R/I)) \simeq \bigoplus_{\mathbb{N}} I$
proved in \cite[p.~205]{Zinkformal}.
\end{proof}

\subsection{Frobenius and Verschiebung}
\label{sec:structureofchW}

\begin{construction}[$F$ and $\hV$ on $\chW$]
For $R \in \nperftors$,
the operators $F, \hV$  on $W(R)$ act compatibly on $Q(R)$ and
$\qperf(R)$ and
therefore induce operators $F, \hV: \chW (R)\to \chW (R)$, preserving the exact
sequences
\eqref{firstSES}, \eqref{secondexactseq}, and \eqref{thirdexactseq}.

We have the relation \[\hV( F(x) y) = x \hV(y)\]
for all $x, y \in \chW(R)$ for
any $R$, since we have this in $W, Q, \qperf$.

When $p > 2$, we have $F \hV = p$, since this holds on
each of $W, Q, \qperf$; note that
$\hV  = V$.
However, when $p = 2$, the operator $F \hV$ is multiplication by the element
\begin{equation} \two\stackrel{\mathrm{def}}{=} F\hV(1) \in \chW(\mathbb{Z}_2).
\end{equation}
Here and below, $\qperf(\mathbb{Z}_2)$ denotes the inverse limit
$\varprojlim_n \qperf(\mathbb{Z}/2^n)$; by \Cref{nilinvarianceofqperf}, this is
identified with $\qperf(\mathbb{Z}/2)=\mathbb{Z}_2$.
Explicitly, $\two$ is the element
$(2[-1]  \in W( \mathbb{Z}_2), 2 \in \qperf( \mathbb{Z}_2))$, noting that $2[-1]
= 2 \in Q( \mathbb{Z}_2)$ (cf.~\Cref{uinWZ2}).
\end{construction}

\begin{proposition}
    \label{prop:chWmodV}
For every $R\in\nperf$, we have a natural short exact sequence
\begin{equation} \label{chWmodV}  0 \to \chW(R) \stackrel{\hV}{\to} \chW(R) \to R \to 0
\end{equation}
compatible with the corresponding sequence
$0\to W(R)\xrightarrow{\hV} W(R)\to R\to0$, via the map
$\chW(R) \to W(R)$.
Therefore, the map $\chW(R) \to W(R)$ exhibits the target as the completion of
the source with
respect to $\hV$.
\end{proposition}
\begin{proof}
This follows from the analogous claims for $W(R), Q(R), \qperf(R)$
(cf.~\Cref{VonQseqprop} and \Cref{QperfmodV}), using the pullback definition
$\chW(R)=W(R)\times_{Q(R)}\qperf(R)$.
The compatibility with the corresponding sequence for $W(R)$ shows by induction
on $n$ that
\[
\chW(R)/\hV^n \xrightarrow{\sim} W(R)/\hV^n
\]
for all $n\geq 1$. Since $\hV$ defines the same filtration on $W(R)$ as $V$,
the ring $W(R)$ is $\hV$-adically complete, which proves the final assertion.
\end{proof}

\begin{remark}
Note that $\chW$ is not $\hV$-adically separated: from \eqref{chWdef} and
\eqref{firstSES}, we see that  the intersection of all powers of
$\hV$ on $\chW$ is the ideal $T_F(Q) \subset \chW$, on which $\hV$ acts
invertibly.
\end{remark}

\begin{corollary}
\label{chWlocalnilpkernel}
Let $R \twoheadrightarrow R'$ be a surjection of $p$-nilpotent rings  with
locally nilpotent kernel $I \subset R$, and suppose that $R \in \nperf$.
Then $\chW(R) \twoheadrightarrow \chW(R')$ and
$\mathrm{ker}( \chW(R) \to \chW(R')) \simeq \hw(I)_{\hat{p}}$.
\end{corollary}
\begin{proof}
This follows from \eqref{thirdexactseq}.
\end{proof}

\begin{construction}
Let $R \in \nperf$.
In general, given $x \in R$, the element $[x] \in W(R)$
need not lift to $\chW(R)$.
However, if either of the following conditions holds:
\begin{enumerate}
\item $x$ is nilpotent, or
\item $x$ admits a compatible system of $p^n$-th roots in $R$, i.e., there exist
$x_n \in R$ such that $x_0 = x$ and $x_{n}^p = x_{n-1}$ for all $n \geq 1$,
\end{enumerate}
then one obtains a natural lift of $[x]$ to $\chW(R) = W(R) \times_{Q(R)} \qperf(R)$, given in the first case by $( [x], 0)$ and given in the second case
by $([x],  \{ [x_n] \}_{n \geq 0})$.
If $x$ is nilpotent and admits such a compatible system of $p^n$-th roots,
the two constructions agree.
\label{liftofbracketx}
\end{construction}
\begin{proposition}
\label{surjectioninducessurjection}
Let $R \twoheadrightarrow R'$ be a surjection of $p$-nilpotent rings, and assume
$R  \in \nperf$.\footnote{This implies $R' \in \nperf$ too.} Then $\chW(R)
\twoheadrightarrow \chW(R')$.
\end{proposition}
\begin{proof}

First, we reduce to the case where $R$ is perfect.
Let $I \subset R$ be the nilradical, and $I' \subset R'$ its image in $R'$.
Then, by \Cref{chWlocalnilpkernel}, we have a diagram of short exact sequences
\[ \xymatrix{
0 \ar[r] &  \hw(I)_{\hat{p}} \ar[d] \ar[r] &  \chW(R) \ar[d]  \ar[r] &
W(R_{\mathrm{red}}) \ar[d]  \ar[r] &  0 \\
0 \ar[r] & \hw( I')_{\hat{p}} \ar[r] &  \chW(R') \ar[r] &  \chW(R'/I')
\ar[r] &  0
}\]
Since $I \twoheadrightarrow I'$
is a surjection of locally nilpotent ideals,
we find $\hw(I) \twoheadrightarrow \hw(I')$ and this is preserved under derived
$p$-completion.
Thus, to prove $\chW(R) \twoheadrightarrow \chW(R')$, it suffices to show that
$ \chW( R_{\mathrm{red}}) = W(R_{\mathrm{red}}) \twoheadrightarrow \chW(
R'/I')$.
In other words, we may reduce to proving the conclusion for $R_{\mathrm{red}}
\to R'/I'$ instead of $R \to R'$.

Thus, we may assume that $R$ is perfect.
In fact, it then suffices to show that
the image of $\chW(R) \to \chW( R')$ contains $\hw(R') \subset  \chW(R')$
(since this image is derived $p$-complete).
Since the image is closed under $\hV$, it suffices to show that if $x \in
\mathrm{Nil}(R')$, then $[x] \in \hw(R') \subset \chW(R')$ can be lifted to
$\chW(R)$. We can perform this lift simply by lifting $x$ to $\tilde{x} \in R$
and forming $[\tilde{x}]
\in W(R) = \chW(R)$, cf.~\Cref{liftofbracketx}.
\end{proof}

\subsection{First examples of $\chW$}

\begin{proposition}
\label{chWofArtinian}
Let $R$ be a local Artinian ring with perfect residue field $k$ of characteristic $p$ and maximal
ideal $\mathfrak{m} \subset R$.
In this case,  the map $\chW(R) \to W(R)$ is injective  and its image is the
direct sum
$W(k) \oplus \hw( \mathfrak{m})$ (using the natural embedding $W(k)
\hookrightarrow W(R)$ which is the unique section of $W(R) \twoheadrightarrow
W(k)$).
\end{proposition}

This ring was considered by Zink \cite{ZinkDieudonne}  and Lau
\cite{Laurelations}
and denoted
$\mathbb{W}(R)$.
\begin{proof}
This follows from \eqref{secondexactseq}, using that
$\qperf(R) \xrightarrow{\sim} \qperf(k) = W(k)$, by
\Cref{nilinvarianceofqperf}.
\end{proof}

\begin{example}
    \label{semiperfectexample}
Consider the
semiperfect ring $\mathbb{F}_p[x^{1/p^\infty}]/(x)$.
Using the surjection $\mathbb{F}_p[x^{1/p^\infty}] \twoheadrightarrow
\mathbb{F}_p[x^{1/p^\infty}]/(x)$, we find a surjection
\[ \mathbb{Z}_p[x^{1/p^\infty}]^{\wedge}_p  \twoheadrightarrow \chW(
\mathbb{F}_p[x^{1/p^\infty}]/(x)). \]
In fact, the kernel of $\chW (\mathbb{F}_p[x^{1/p^\infty}]/(x)) \to \chW(
\mathbb{F}_p) = \mathbb{Z}_p$ is
$\hw(\mathbb{F}_p[x^{1/p^\infty}]/(x))_{\hat{p}}$.
By
\Cref{gradingsonhw}, the latter is the derived $p$-completion of its graded
pieces in degrees $i \in \mathbb{Z}[1/p]_{>0}$, each of which is generated
by $[x^i]$.
Note that $p^j [x^i] = 0$ where $j$ is minimal such that $p^j i \geq 1$ and, by
comparison with the Witt vectors, this is optimal.
We thus find that
the map
\begin{equation} \label{checkWofasemiperfectex}
\mathbb{Z}_p \oplus \bigoplus_{ i \in \mathbb{Z}[1/p] \cap (0, 1)}
\mathbb{Z}/p^j\cdot [x^i] \to
\chW( \mathbb{F}_p[x^{1/p^\infty}]/(x)) \end{equation}
(with $j$ chosen as before) exhibits the target as the derived $p$-completion of
the source.
\label{truncatedsemiperfectchW}

Note that the element
$ \lim_{n \to \infty} (1 + [x^{1/p^n}])^{p^n} $
maps to $1 = [1+x]$ in $W( \mathbb{F}_p[x^{1/p^\infty}]/(x))$, but is not equal
to $1$
in $\chW( \mathbb{F}_p[x^{1/p^\infty}]/(x))$; in fact, the difference from $1$
is not homogeneous.
\end{example}

\begin{remark}
In \Cref{truncatedsemiperfectchW}, $\chW( \mathbb{F}_p[x^{1/p^\infty}]/(x))$ is
not classically $p$-complete; in fact, it is not $p$-adically separated.
The derived $p$-completion of the left-hand side of
\eqref{checkWofasemiperfectex} is in fact a standard example of a derived
$p$-complete but not classically
$p$-complete abelian group \cite[Ex.~2.4]{Bhatttorsion}.
In fact, if $R$ is any semiperfect $\mathbb{F}_p$-algebra,
then  $\chW(R) \twoheadrightarrow W(R)$ (since both receive surjective maps from
$W(R^{\flat})$) and the kernel of $\chW(R) \to W(R)$ consists precisely of the
elements which are divisible by all powers of $\hV$, or equivalently (by $\hV F
= F \hV = p$) by all powers of $p$.
In particular, if $\chW(R) \twoheadrightarrow W(R)$ has a nonzero kernel, then
$\chW(R)$ is \emph{never} classically $p$-complete.
\end{remark}

\begin{proposition}
\label{chWofsemiperfect}
Let $R = P/I$ be the quotient of a perfect  $\mathbb{F}_p$-algebra $P$ by an
ideal $I \subset P$.
Then $\chW(R)$ is the derived $p$-completion of the quotient of $W(P)$ by the
ideal generated by $V^j( [x])$ for $x \in I$ and $j \geq 0$.
\end{proposition}
\begin{proof}
By \Cref{surjectioninducessurjection}, $W(P)= \chW(P) \twoheadrightarrow
\chW(R) = W(R) \times_{Q(R)} \qperf(R)$.
It suffices to identify the kernel.
Given $a = \sum_{i \geq 0} V^i ( [a_i])$ for  a sequence $a_0, a_1, \dots \in
P$, then $a$ belongs to the kernel of $\chW(P) \to \chW(R)$ if and only if:
\begin{enumerate}
\item All $a_i \in I$ (i.e., $a$ maps to zero in $W(R)$).
\item For all $n$, all but finitely many terms in the sequence $\phi^{-n}(a_0),
\phi^{-n}(a_1), \dots $ belong to $I$. In other words,  for each $n$,
$\phi^{-n}(a) \in W(P)$
maps to zero in $Q(R)$, or equivalently $a$ maps to zero in $\qperf(R)$.
\end{enumerate}
Consequently, in the expression
$a = \sum_{i \geq 0} V^i [a_i]$, each term
belongs to the ideal $J \subset W(P)$
generated by $V^j( [x])$ for $x \in I$ and $j \geq 0$, and for each $n$, all but
finitely
many terms belong to $p^n J$.
It follows that the kernel of $W(P) \to \chW(R)$ is precisely the $p$-completion
of $J$ (noting also that $\chW(R)$ is derived $p$-complete).
\end{proof}

\subsection{\texorpdfstring{$\delta$-nilperfection and the cofree property of $\chW$}{delta-nilperfection and the cofree property of check W}}

We now turn to the corresponding $\delta$-nilperfection construction. This
formalism gives the right adjoint to the inclusion of \deltanilperfect\
$\delta$-rings and recovers $\chW$ as the cofree \deltanilperfect\
$\delta$-ring.

\begin{definition}[\deltanilperfect\ $\delta$-rings]
\label{nilperfect:def}
We say that a $\delta$-ring $A$ is \emph{\deltanilperfect} if
$A/\dnil{A}$ is perfect as a $\delta$-ring.
We let $\dhl \subset \dring$ be the full subcategory of \deltanilperfect\
$\delta$-rings.
\end{definition}

\begin{remark}
\label{charofdhl}
A $\delta$-ring $A$ belongs to $\dhl$ if and only if there exists a
$\delta$-ideal $I \subset \dnil{A}$ such that $A/I$ is a perfect $\delta$-ring.
This follows because $ \dnil{(A/I)} = 0$ by perfectness, so $I = \dnil{A}$.
\end{remark}

\begin{proposition}
The subcategory
$\dhl \subset \dring$
is
closed under colimits.
\end{proposition}
\begin{proof}
Consider any $J$-indexed diagram $\{ A_j\} $ in $\dhl$.
	Let $A = \varinjlim_{j \in J} A_j$; this colimit can be calculated either in
	$\dring$ or in the category of rings.
	For each $j \in J$, set $I_j=\dnil{A_j}$. Inside $A$, consider the ideal $I$
	generated by each $I_j$; this
	is a $\delta$-ideal contained in $\dnil{A}$ by functoriality. The quotient
	$A/I$ is
$\varinjlim_{j \in J} A_j/ I_j$, which is a perfect $\delta$-ring as
a colimit of such.
By \Cref{charofdhl}, this implies that $A \in \dhl$.
\end{proof}

\begin{definition}[$\delta$-nilperfection of a $\delta$-ring]
\label{def:nilperfection}
Given any $\delta$-ring $A$, 
we write $A^{\mathrm{perf}}= \varprojlim_\phi A$ for the (inverse limit) perfection of $A$.

We define the $\delta$-ring $\nilperf{A}$, called
the
\emph{$\delta$-nilperfection} of $A$, via the pullback square
\begin{equation}
\label{Acheckdef}
\xymatrix{
\nilperf{A} \ar[d] \ar[r] &   (A/\dnil{A})^{\mathrm{perf}} \ar[d]  \\
A \ar[r] &  A/\dnil{A}
}. \end{equation}
As a result, we observe that $\nilperf{A}$ fits into a short exact
sequence \begin{equation} \label{firstseqDelta} 0 \to  \dnil{A} \to \nilperf{A}
\to
(A/\dnil{A})^{\mathrm{perf}} \to 0,  \end{equation}
so $\nilperf{A}$ is an extension of a perfect $\delta$-ring by an ideal of
$\delta$-nilpotent elements. In particular, $\nilperf{A} \in \dhl$.
\end{definition}

\begin{example}[$\chW(R)$ as $\delta$-nilperfection]
\label{chWasnilperfection}
Let $A = W(R)$, where $R$ is a nilperfect ring.
Then $\chW(R)$ is the $\delta$-nilperfection $\nilperf{W(R)}$, by
\Cref{dnilinW} and the definition of $\chW(R)$.
\end{example}

\begin{example}
Suppose $\dnil{A} = 0$. In this case, $\nilperf{A} = A^{\mathrm{perf}}$.
\end{example}

\begin{proposition}
\label{chisrightadjoint}
The functor $A \mapsto \nilperf{A}$ from $ \dring \to \dhl$ is the
right adjoint to the
inclusion $\dhl \subset \dring$.
In particular, if $A$ is \deltanilperfect, then $\nilperf{A} \xrightarrow{\sim} A$.
\end{proposition}
\begin{proof}
Let $A \in \dring, B \in \dhl$.
Then
we need to prove
that the natural map induces an isomorphism
\[ \Hom_{\dring} (B, A) \simeq  \Hom_{\dring} (B, \nilperf{A}). \]
This follows from the definition
\eqref{Acheckdef}
using that
\begin{align*} \Hom_{\dring}(B, A/\dnil{A}) & =
\Hom_{\dring}(B/\dnil{B}, A/\dnil{A}) \\
& =
\Hom_{\dring}(B/\dnil{B}, (A/\dnil{A})^{\mathrm{perf}}) \\
& =
\Hom_{\dring}(B, (A/\dnil{A})^{\mathrm{perf}}) ,
\end{align*}
where we used that the $\delta$-ring $B/\dnil{B}$ is perfect, since $B \in
\dhl$.
\end{proof}

\begin{proposition}
\label{nilperfectioncrit}
Let $f: B \to A$ be  a map of $\delta$-rings.
Then the natural map $\nilperf{B} \to \nilperf{A}$ is an isomorphism if and only if:
\begin{enumerate}
\item $ f$ induces an isomorphism $\dnil{B} \xrightarrow{\sim} \dnil{A}$.
\item
$f$ induces an isomorphism $(B/\dnil{B})^{\mathrm{perf}} \xrightarrow{\sim}
(A/\dnil{A})^{\mathrm{perf}}$.
\end{enumerate}
\end{proposition}
\begin{proof}
In fact, the two conditions imply that $\nilperf{B} \xrightarrow{\sim}
\nilperf{A}$
in light of \eqref{firstseqDelta}. 

Conversely, the definition and the exact sequence \eqref{firstseqDelta} show that for any $\delta$-ring $A$, the nilperfection of $A$ satisfies the conditions, i.e., if we set $C = \nilperf{A} $,  
then $\dnil{C} \xrightarrow{\sim} \dnil{A}$ and $(C/\dnil{C})^{\mathrm{perf}} = (A/\dnil{A})^{\mathrm{perf}}$.
The converse direction then follows. 
\end{proof}

\begin{corollary}
\label{chWcofreenilperfect}
Let $R \in \nperf$. Then the map $\chW(R) \to R$ exhibits $\chW(R)$ as the
cofree \deltanilperfect\ $\delta$-ring on $R$. Equivalently, for any
\deltanilperfect\ $\delta$-ring $A$, the natural map
\[
\Hom_{\dring}(A,\chW(R)) \xrightarrow{\sim} \Hom(A,R)
\]
is a bijection.
\end{corollary}
\begin{proof}
This follows from \Cref{chisrightadjoint} and Joyal's theorem
(\Cref{joyal:witt}), together with the identification
$\chW(R)=\nilperf{W(R)}$.
\end{proof}

\section{Square-zero extensions of $\delta$-rings} 

In this section, we treat some aspects of the theory of square-zero extensions of $\delta$-rings, and we show that in suitable circumstances, the map $\chW(R) \to W(R)$ can be characterized as the 
universal ``taut'' extension. 

\subsection{General square-zero extensions of $\delta$-rings}

Throughout, we fix a base $\delta$-ring $A_0$ (which could be $\mathbb{Z}$). 

\begin{definition}[Square-zero extensions of $\delta$-rings]
\label{squarezerodelta}
A \emph{square-zero extension}
of $\delta$-rings under $A_0$ is a surjective map $B \twoheadrightarrow \overline{B}$ of
$\delta$-$A_0$-algebras such that the kernel squares to zero as an ideal (i.e., it is a
square-zero extension in the ring-theoretic sense).
\end{definition}

\begin{example}
Given any $\delta$-$A_0$-algebra $B$ and any $\delta$-ideal $J \subset B$, the ideal $J^2
\subset B$ is also stable under $\delta$, and the map
$B/J^2 \to B/J$ becomes a square-zero extension by the ideal $J/J^2 \subset
B/J^2$.
\end{example}

\begin{definition}[Frobenius modules]
Let $A$ be a $\delta$-ring. A \emph{Frobenius module} over $A$ is an $A$-module
$M$ equipped with a map of $A$-modules $F: M \to \phi_* (M)$.
(Equivalently, $F$ defines a Frobenius semi-linear map $M \to M$.)
A Frobenius module over $A$ is equivalent to a left module over a noncommutative ring
$A[F] = \bigoplus_{i \geq 0} A\cdot F^i$ where the multiplication $a F^i \cdot b
F^j$ is $a \phi^i(b) F^{i+j}$.
We let $D_{\phi}(A) = D( A[F])$ denote the derived $\infty$-category of
Frobenius modules
over $A$.

\end{definition}

\begin{construction}
Given a square-zero extension $B \twoheadrightarrow \overline{B}$ of $\delta$-$A_0$-algebras, with kernel $J$, the restriction of $\delta$ gives a map $\delta:J\to J$; the
$\delta$-identities imply that $\delta:J\to J$ is an additive map, and
$\phi$-semilinear with respect to
the $\overline{B}$-module structure, i.e., $\delta(xy) = \phi(x) \delta(y)$ for
$x \in \overline{B}, y \in J$.
Thus, $\delta$ gives $J$ the structure of a Frobenius $\overline{B}$-module.
Moreover, $\phi|_J = p \delta|_J$.
\end{construction}

\begin{example}
    \label{tautsquarezeroexample}
Consider the ring $R = \mathbb{Z}[t, \epsilon]/((t-1) \epsilon, \epsilon^2)$.
Since $R$ is $p$-torsionfree, we can make $R$ into a $\delta$-ring via $\phi(
\epsilon) = p \epsilon$ and
$\phi(t) = t^p + p \epsilon$.
We have a short exact sequence
\[ 0 \to \mathbb{Z}[t]/(t -1) . \epsilon \to R \to \mathbb{Z}[t] \to 0,  \]
which exhibits
$R$ as a square-zero extension of $\mathbb{Z}[t]$ by $\mathbb{Z} \epsilon$,
considered as a Frobenius module  over $\mathbb{Z}[t]$  with $t \epsilon =
\epsilon$ and $F(\epsilon) = \epsilon$.
\end{example}

\begin{construction}[Trivial square-zero extensions]
\label{squarezerotrivdelta}
Let $A$ be a $\delta$-ring and $M$ a Frobenius module over $A$ with $F: M \to
\phi_* M$.
We can form a $\delta$-ring $A \oplus M$, where the
$\delta$-structure is given by
\begin{equation} \label{squarezeroformula}  \delta( a, m) = ( \delta(a), F(m) - a^{p-1} m).
\end{equation}
In particular,
$\phi(a, m) = ( \phi(a), p F(m))$. By reduction to the $p$-torsionfree case, one
checks that
\eqref{squarezeroformula} defines a $\delta$-structure on $A \oplus M$.
Note that $A \oplus M$ is an augmented $\delta$-ring over $A$, and is
naturally an abelian
group object in this
category.
\end{construction}

\begin{proposition} 
Let $A$ be a $\delta$-ring. 
The category of abelian group objects of $\dring_{A// A}$ 
is identified with the category of 
Frobenius modules over $A$, via 
\Cref{squarezerotrivdelta}. 
\end{proposition}
\begin{proof} 
It is well-known that the category of abelian group objects of
$\mathrm{Ring}_{A//A}$ is equivalent to the category of $A$-modules, via the
square-zero construction $N \mapsto A \oplus N$. One then checks that adding a
$\delta$-structure to $A \oplus N$ (compatible with the previous data) is
equivalent to adding a Frobenius module structure to $N$. 
\end{proof} 
As a consequence, the following discussion is a special case of the formalism of
the cohomology of algebras over a monad, cf.~\cite{Beck}, \cite{BarrBeck}, or
\cite[II.5]{Quillen67}.

In the usual theory for rings, a \emph{derivation} $R \to N$ for $R$ a ring and
$N$ an $R$-module is equivalent to a section of the projection map $R \oplus N
\to R$.
One can make an analogous definition for $\delta$-rings.

\begin{definition}[$\delta$-derivations]
\label{deltaderivations}
Let $A$ be a $\delta$-ring under $A_0$ and let $M$ be a Frobenius $A$-module, with
Frobenius-semilinear operator denoted $F_M:M\to M$. By a
\emph{$\delta$-derivation} from $A$ to $M$ (relative to the base $A_0$), we mean a section of the projection
$A\oplus M\to A$ in the category of $\delta$-$A_0$-algebras. 
One checks that specifying a $\delta$-derivation is equivalent to specifying an ordinary $A_0$-derivation $d: A \to M$ such that
\begin{equation}
\label{deltaderivationformula}
d(\delta(a))=F_M(d(a))-a^{p-1}d(a)
\end{equation}
for every $a\in A$.
\end{definition}
\begin{remark} 
Note that a $\delta$-derivation $d: A \to M$ satisfies, for any $a \in A$, 
\begin{equation} \label{deltaderivationF} d( \phi(a)) = p F_M(d(a)).\end{equation}
Conversely, if $M$ is $p$-torsionfree, then any ordinary derivation $d: A \to M$ (relative to the base $A_0$) that satisfies 
\eqref{deltaderivationF} is a $\delta$-derivation (relative to $A_0$).

\end{remark} 
\begin{construction}[$\Omega^1$ of  a $\delta$-ring]
\label{dividedFondeltaring}
Let $A$ be a $\delta$-$A_0$-algebra which is free as a $\delta$-$A_0$-algebra (and in particular
a polynomial ring). Suppose first that $A_0$ is $p$-torsionfree. 
Then the Frobenius $\phi: \Omega^1_{A/A_0} \to
\phi_*(\Omega^1_{A/A_0})$
is divisible by $p$, yielding (since $\Omega^1_{A/A_0}$ is torsionfree)
a map of $A$-modules $F: \Omega^1_{A/A_0} \to \phi_* ( \Omega^1_{A/A_0})$
satisfying the formula
\begin{equation} 
\label{frobeniusonomega1}
F ( dx) = \frac{d\phi(x)}{p} =  d \delta(x) + x^{p-1} dx,
\end{equation} 
so that $p F = \phi^*$. 
\end{construction}

\begin{example} 
\label{cotangentoffreedelta}
Suppose $A$ is the free $\delta$-$A_0$-algebra on a generator $x$, so $A = A_0[x, 
\delta(x), \delta^2(x), \dots ]$. 
In this case, 
$\Omega^1_{A/A_0}$ is the free $A$-module on $dx, d \delta(x), d \delta^2(x), \dots$. 
From \eqref{frobeniusonomega1}, we find that $\Omega^1_{A/A_0}$ is the free Frobenius
module on the class $dx \in \Omega^1_{A/A_0}$. 
\end{example}

\begin{construction}[The cotangent complex as a Frobenius module]
\label{cotangentasFrobmodule}
We left Kan extend  \Cref{dividedFondeltaring} from  pairs $(A_0 \to A)$ with $A_0$ $p$-torsionfree and $A/A_0$ free to all maps of $\delta$-rings. 
Since the free $\delta$-rings are polynomial rings, this left Kan extension recovers the usual cotangent complex $L_{A/A_0} \in D(A)$.
Thus, for any $\delta$-ring $A$ under any base $A_0$, we obtain a natural
map 
$F: L_{A/A_0} \to \phi_* L_{A/A_0}$ in $D(A)$, so $L_{A/A_0} \in D(A)$ refines to an object of the derived $\infty$-category
$D_\phi(A)$ of Frobenius modules over $A$. 
\end{construction}

\begin{example} 
\label{universaldeltaderivation}
In particular, on $\pi_0$, we obtain a map $F: \Omega^1_{A/A_0} \to
\Omega^1_{A/A_0}$ such that $pF = \phi^*$; this is given by the same formula
in \eqref{frobeniusonomega1}, and thus makes $\Omega^1_{A/A_0}$ a Frobenius module for any $A$.

Let $A$ be a $\delta$-ring. 
The universal derivation
$d:A\to\Omega^1_{A/A_0}$ is a $\delta$-derivation for the Frobenius module structure above, and it is straightforward to check that it is universal among $\delta$-derivations.
\end{example}

We now formulate the classification of square-zero extensions via the cotangent
complex as a Frobenius module, cf.~\cite[Th.~6]{Beck} or
\cite[Prop.~2.4]{Quillencohomology}. 
For the convenience of the reader, we briefly reproduce an argument here. 

\begin{theorem} 
\label{classification:squarezero}
Let $\overline{B}$ be a $\delta$-$A_0$-algebra, and let $J$ be a Frobenius module for
$\overline{B}$. 
Then the groupoid of square-zero extensions (in $\delta$-$A_0$-algebras) of
$\overline{B}$ by $J$ is  naturally equivalent to 
the underlying groupoid of $\mathrm{RHom}_{D_\phi(\overline{B})}(L_{\overline{B}/A_0}, J[1])$. 
\end{theorem} 
\begin{proof} 
Consider 
the category $(\dring)_{A_0/\overline{B}}$ of
$\delta$-$A_0$-algebras equipped
with a map to $\overline{B}$. 
To any $B' \in (\dring)_{A_0/\overline{B}}$, we associate the groupoid $\mathrm{SqZero}( B',
J)$ of 
square-zero extensions of $\delta$-$A_0$-algebras of $B'$ by $J$
(seen as a Frobenius $B'$-module via $B' \to \overline{B}$). 
Via pullback of extensions, this defines a functor 
$(\dring)_{A_0/\overline{B}} \to \mathrm{Groupoids}$. 
When $B'$ is free as a $\delta$-$A_0$-algebra, any square-zero extension splits,
and is consequently isomorphic to $B' \oplus J$; the automorphisms of $B' \oplus
J$ are identified with $\Hom_{B', \phi}( \Omega^1_{B'/A_0}, J)$, and so
we have the desired identification.

To reduce to an arbitrary $B' \in (\dring)_{A_0/\overline{B}}$ (in particular, $B' = \overline{B}$),
we define a Grothendieck topology on the opposite of the category
$(\dring)_{A_0/\overline{B}}$ of
$\delta$-$A_0$-algebras equipped
with a map to $\overline{B}$
such that the covering families are generated by surjections $B_1
\twoheadrightarrow B_2$. 
We observe that the construction $B' \mapsto \mathrm{SqZero}(B', J)$ 
(as a functor $(\dring)_{A_0/\overline{B}} \to \mathrm{Groupoids}$)
is a sheaf
of groupoids for the above topology. That is, if $B_1 \twoheadrightarrow B_2$,
then 
\[ \mathrm{SqZero}(B_2, J) \simeq \varprojlim \left(  \mathrm{SqZero}( B_1, J) 
\rightrightarrows \mathrm{SqZero}( B_1 \times_{B_2} B_1, J) \triplearrows \dots
\right),
\]
because the analogous statement holds for the category $(\dring)_{A_0/B_2}$
(indeed, it holds for sets), and being a square-zero
extension can be tested after pulling back. 
Thus, to produce the desired identification, it suffices to do so on a basis;
however, we have already done so on free $\delta$-$A_0$-algebras. 
\end{proof}

\subsection{Taut square-zero extensions}

\begin{definition}[Perfect Frobenius modules]
Let $A$ be a $\delta$-ring, and let
$M \in D_\phi(A)$. We say that $M$ is \emph{perfect} if $F: M \xrightarrow{\sim}
\phi_* (M)$.
Any $N \in D_\phi(A)$ has a \emph{perfection}, given by $\varinjlim_{n, F}
\phi_*^n(N)$; this is the left adjoint of the inclusion of perfect Frobenius
modules into $D_\phi(A)$.
\end{definition}

\begin{remark}
Perfect Frobenius modules over $A$ are equivalent to perfect Frobenius
modules over the \emph{colimit perfection} $A_{\mathrm{perf}}
\stackrel{\mathrm{def}}{=} \varinjlim_\phi A$.
\end{remark}

\begin{definition}[Taut square-zero extensions and derivations]
\label{tautderivations}
We say that a square-zero extension $0 \to J \to B  \to \overline{B} \to 0$
of $\delta$-rings is \emph{taut} if
$J$ is perfect as a Frobenius module over $\overline{B}$ and is derived $p$-complete. Similarly, 
a $\delta$-derivation into a Frobenius module $M$ is  \emph{taut} if $M$ is perfect and derived $p$-complete as a Frobenius module.
\end{definition}

\begin{example}
The square-zero extension of \Cref{tautsquarezeroexample} becomes taut after $p$-completion.
The universal taut derivation of a $\delta$-ring $A$ is obtained by derived $p$-completing the tautological map into $\Omega^1_A[1/F]$, where $\Omega^1_A$ is equipped with the Frobenius module structure as in \Cref{universaldeltaderivation}. \label{univtautderivation}
\end{example}

\begin{definition}[Taut \'etale maps]
Let $f: A_0 \to A_1$ be a map of $\delta$-rings. We say that $f$ is \emph{taut \'etale} if 
any diagram of $\delta$-rings of the form 
\begin{equation} 
\xymatrix{
    A_0 \ar[d] \ar[r] &  C \ar[d]  \\
    A_1 \ar[r] \ar@{.>}[ur] &  \overline{C} } 
\end{equation} 
where $C \to \overline{C}$ is a taut square-zero extension, admits a unique dashed lift.
\end{definition}

In order to work with taut \'etale maps, it is convenient to use the following slightly stronger condition. 
\begin{definition} 
We say that a map of $\delta$-rings $A_0 \to A_1$ is a \emph{taut equivalence} if: 
\begin{enumerate}
    \item  $(A_0/p)_{\mathrm{perf}} \xrightarrow{\sim} (A_1/p)_{\mathrm{perf}}$.
    \item 
    Pullback along $A_0\to A_1$ induces an equivalence from the category of taut square-zero extensions of $A_1$ with kernel annihilated by $p$ to the corresponding category for $A_0$.
\end{enumerate}
\end{definition}

\begin{definition}[The taut cotangent complex]\label{def:tautcotangentcomplex}
Given a $\delta$-ring $A$ over a base $\delta$-ring $A_0$, we let $\lt_{A/A_0}$ denote 
\[
\lt_{A/A_0} \stackrel{\mathrm{def}}{=}
\left(\varinjlim_{n,F}\phi_*^nL_{A/A_0}\right)^\wedge_p\in D_\phi(A),
\]
and call it the \emph{taut cotangent complex}. \end{definition}

By \Cref{classification:squarezero}, taut square-zero extensions of $A$ in $\delta$-$A_0$-algebras  are controlled by the 
1-truncation of $\lt_{A/A_0}$.

\begin{remark}[Taut equivalences are taut \'etale]
Note that the second condition in the statement that the map $A_0 \to A_1$ of $\delta$-rings is a taut equivalence is equivalent, given the first condition, to the following requirement: for every perfect Frobenius module $M$ over $A_1/p$, the induced map
\[
\operatorname{Map}(\lt_{A_1},M[1])\to \operatorname{Map}(\lt_{A_0},M[1])
\]
is an equivalence.
In particular, the space (i.e., groupoid) of maps from $\lt_{A_1/A_0}$ into $M[1]$ is contractible, which means (by derived $p$-completeness) that the 
1-truncation of $\lt_{A_1/A_0}$ vanishes.

As a consequence, if $A_0 \to A_1$ is a taut equivalence, then it is taut \'etale.
In fact, the obstruction to finding the lift vanishes
because $\lt_{A_1/A_0}$ vanishes in homological degrees $\leq 1$. 
However, the condition of being a taut equivalence is a priori stronger, because it also implies that any taut square-zero extension of $A_0$ with kernel annihilated by $p$ should in fact be pulled back from $A_1$. 
\end{remark}

\begin{proposition} 
    Let $A_0$ be a $\delta$-ring. 
Let $A$ be a $\delta$-$A_0$-algebra. Suppose that 
$\Omega^1_{A/A_0}[1/F]/p = 0$. 
Then $A$ admits an initial taut extension in the category of $\delta$-$A_0$-algebras. 
\end{proposition} 
\begin{proof} 
This follows from the classification of square-zero extensions in \Cref{classification:squarezero}.  Explicitly, 
the universal taut extension is given by $H_1(\lt_{A/A_0})$.
\end{proof} 

\begin{definition}\label{delta-taut-rigidity}
We say that a $\delta$-ring $A_0$ is \emph{taut rigid} if every taut extension of $A_0$ admits a unique section.
This holds if and only if $\lt_{A_0}$ vanishes in homological degrees $\leq 1$.

\end{definition} 

\begin{remark} 
The $\delta$-ring $A_0$ is taut rigid if and only if 
every taut square-zero extension of $A_0$ by an ideal annihilated by $p$ admits a unique section.
In fact, this follows from derived $p$-complete Nakayama's lemma. 
Consequently, taut rigidity is preserved under taut equivalences. 
\end{remark}

\begin{proposition}\label{tautequivdiagonalcriterion}
Let $A_0 \to A_1$ be a flat map of $\delta$-rings such that $(A_0/p)_{\mathrm{perf}} \xrightarrow{\sim} (A_1/p)_{\mathrm{perf}}$. 
Suppose that $\lt_{A_0}$ is concentrated in degree zero. 
Suppose
 $A_1 \otimes_{A_0} A_1 \to A_1$ is a taut equivalence. Then $A_0 \to A_1$ is a taut equivalence. \end{proposition} 

\begin{proof} 
    We freely use here that $A_0, A_1, A_1 \otimes_{A_0} A_1$ all have the same 
    $p$-completed perfection, so we can omit various base changes when working with perfect Frobenius modules. 
Because $\lt_{A_0}$ is concentrated in degree zero, it follows that $\lt_{A_0} \to \lt_{A_1}$ is an equivalence after 1-truncation if and only if $\lt_{A_1/A_0}$ vanishes after 1-truncation.

     By assumption, $\lt_{A_1 \otimes_{A_0} A_1 }$ and $\lt_{A_1}$ 
    have the same maps into perfect $p$-torsion Frobenius modules placed in degree zero or one. 
    Upon taking the mapping cone of the map from $\lt_{A_0}$ and using the transitivity triangle for the cotangent complex, we find that 
$\lt_{ (A_1 \otimes_{A_0} A_1)/A_0}$ and $\lt_{A_1/A_0}$  have the same maps into any perfect $p$-torsion Frobenius module placed in degree zero or one.
    We have that
    $$ \lt_{(A_1 \otimes_{A_0} A_1)/A_0} \simeq \lt_{A_1/A_0}  \oplus     \lt_{A_1/A_0}, $$
    so maps from $\lt_{A_1/A_0}$ and from $\lt_{A_1/A_0} \oplus \lt_{A_1/A_0}$ into any perfect $p$-torsion Frobenius module placed in degree zero or one are equivalent.
This (together with Nakayama's lemma) implies that $\lt_{A_1/A_0}$ vanishes in degrees $\leq 1$. 
 \end{proof}

\subsection{Lifting $\delta$-nilpotent elements}

In this subsection we study various Hensel-type lifting properties for taut square-zero extensions of $\delta$-rings. 
We show that $\delta$-nilpotent elements lift uniquely through taut square-zero extensions, and that the same holds for  an appropriate notion of ``topologically $\delta$-nilpotent'' if one imposes suitable bounded torsion conditions.

\begin{lemma}
\label{squarezerolemma}
Consider a $\delta$-ring $C$ and a square-zero $\delta$-ideal $J \subset C$.
Fix $c \in C$ such that 
$cJ = 0$, and
fix $j \in J$. 
Then $\delta(c + j) = 
\delta(c) + \delta(j)$. 
If $\delta: J \simeq \phi_* J$ is an isomorphism (e.g., $J$ is taut), then it suffices that $\phi^n(c)$ annihilates $J$ for some $n$. 
\end{lemma}
\begin{proof} 
The identity for $\delta(a+b)$ gives
\[
\delta(c+j)=\delta(c)+\delta(j)-c^{p-1}j,
\]
since $J^2=0$. Thus the first assertion follows from $cJ=0$.
If $\phi^n(c)$ annihilates $J$ and $J$ is a perfect Frobenius module, then  
$cJ = 0$. 
\end{proof}

\begin{proposition}
\label{tautextensionsdeltanilp}
Any taut square-zero extension of $\delta$-rings
\[
0 \to I \to B \to A \to 0
\]
induces an isomorphism
\[
\dnil{B} \xrightarrow{\sim} \dnil{A}.
\]
\end{proposition}
\begin{proof}
Since $\delta: I \simeq \phi_* I$ is an isomorphism, it follows that no nonzero element of $I$ can be $\delta$-nilpotent. This proves injectivity.

For surjectivity, let $\epsilon \in \dnil{A}$ and choose any lift
$\widetilde{\epsilon}_0 \in B$.
For $n \gg 0$, the element $\delta^n(\widetilde{\epsilon}_0)$ belongs to $I$.
Since $\delta$ is bijective on $I$, set
\[
i=-\delta^{-n}(\delta^n(\widetilde{\epsilon}_0)) \in I.
\]
By \Cref{squarezerolemma},
\[
\delta^n(\widetilde{\epsilon}_0+i)
=\delta^n(\widetilde{\epsilon}_0)+\delta^n(i)=0.
\]
Moreover, each $\delta$-iterate of $\widetilde{\epsilon}_0+i$ is nilpotent:
its image in $A$ is nilpotent, and any lift of a nilpotent element through a
square-zero extension is nilpotent. Thus $\widetilde{\epsilon}_0+i$ is
$\delta$-nilpotent and lifts $\epsilon$.
\end{proof}

\begin{corollary} 
    \label{tautdeltanilpdescent}
The map $A \to A/\dnil{A}$ is a taut equivalence.
\end{corollary} 
\begin{proof} 
Given a taut square-zero extension $B \twoheadrightarrow A$, we need to construct a unique descent to  a taut square-zero extension of
$A/\dnil{A}$. For this, we form 
$B/\dnil{B} \to A/\dnil{A}$, which is a descent by \Cref{tautextensionsdeltanilp}. 
Note that this is the only possible descent: any descent gives a splitting of $ B \twoheadrightarrow A$ over $\dnil{A}$, and the uniqueness of the splitting follows from the uniqueness in \Cref{tautextensionsdeltanilp}.
\end{proof}

\begin{corollary}
\label{tautsqzerosplits}
Suppose that $A$ is a \deltanilperfect\ $\delta$-ring.
Any taut square-zero extension  of $\delta$-rings
\( 0 \to I \to B \to A \to 0  \)
with $I$ derived $p$-complete
admits a unique section (in $\delta$-rings): that is, $A$ is taut rigid. 
\end{corollary}
\begin{proof}
This follows from 
\Cref{tautdeltanilpdescent} and taut rigidity of perfect $\delta$-rings (which follows since the mod $p$ cotangent complex vanishes for those).
\end{proof}

Next, we consider a variant of $\delta$-nilpotence 
for $\delta$-rings equipped with the $p$-adic topology and prove analogs of the above results. 

\begin{definition} 
Let $A$ be a $\delta$-ring with bounded $p$-power torsion, which we equip with the $p$-adic topology. 
Let $a \in A$. 
We say that $a$ is \emph{topologically $\delta$-nilpotent}  if 
 the map $$ A \xrightarrow{w_\delta} 
W(A) $$
carries $a$ into $\hw(A)$; equivalently (since $A$ has the $p$-adic topology), the image of $w_\delta(a)$ in $W(A/p^n)$ belongs to $\hw(A/p^n)$ for each $n > 0$). 

Explicitly, $a$ is topologically $\delta$-nilpotent if and only if $a, \delta(a), \delta^2(a), \dots$ are all topologically nilpotent, and 
$\delta^n(a) \to 0$ in the $p$-adic topology as $n \to \infty$. 
We let $\dniltop{A} \subset A$ denote the set of topologically $\delta$-nilpotent elements. \end{definition} 

It is easy to see that $\dniltop{A}$ is a $\delta$-ideal in $A$.
Observe also that if $a \in A$ is topologically nilpotent (i.e., $a^n \in pA$ for some $n$) and $\delta(a) \in \dniltop{A}$, then $a \in \dniltop{A}$ as well.

\begin{lemma}
\label{pdividespthree}
For every prime \(p\), the element \([p^3]\in \hw(\mathbb Z_p)\) is divisible by
\(p\) in \(\hw(\mathbb Z_p)\). Consequently, for any ring \(R\) with bounded $p$-power torsion (with the $p$-adic topology), one has
for $n \geq 1$,
\[
\hw(p^{2n+1}  R)\subset p^n\,\hw(pR).
\]
\end{lemma}
\begin{proof}
By \cite[Lem.~4.7.3]{drinfeld}, one has
\(
p\mid [p^2]
\) in $W(\mathbb Z_p)$. 
Multiplying by $[p]$ and using that $\hw(p\mathbb Z_p)$ is an ideal of $W(\mathbb Z_p)$, we get that $[p^3]/p \in \hw(p\mathbb Z_p)$.
For the last statement, we may assume without loss of generality that $R$ is $p$-complete. 
 
Given $x \in \hw(p^{2n+1}R)$, we can write $x = \sum_{i \geq 0} V^i([p^{2n+1} x_i])$ for some $x_i \in R$ converging to zero in the $p$-adic topology. 
Then $x = p^n\sum_{i \geq 0} V^i(([p^{2n}]/p^n)\cdot [px_i]) \in p^n\,\hw(pR)$.
This gives the desired claim $\hw(p^{2n+1}R)\subset p^n\,\hw(pR).$
\end{proof}

\begin{lemma}\label{dividingtopdnilbyp}
Let $A$ be a $p$-torsionfree $\delta$-ring. 
Suppose $a \in \dniltop{A}$ and suppose that $a, \delta(a), \delta^2(a), \dots$ are all divisible by $p^{2n+1}$ for some $n \geq 1$. 
Then $a/p^n \in \dniltop{A}$ as well.
\end{lemma} 

\begin{proof} 
This follows from  \Cref{pdividespthree}. 
\end{proof}

\begin{remark} 
Let $R$ be a $p$-torsionfree ring. Then an element of $W(R)$ is topologically $\delta$-nilpotent if and only if it belongs to $\hw(R)$ (where $R$ is given the $p$-adic topology).
The difficult direction is to show that if $x \in \hw(R)$, then $x$ is topologically $\delta$-nilpotent, which follows from \Cref{pdividespthree}. 
\end{remark}

\begin{proposition}\label{topdnilliftstautextension}
Consider a taut square-zero extension
of $\delta$-rings
\begin{equation} 
0 \to I \to B \to A \to 0
\end{equation} 
such that $pI = 0$, and such that $A$ is $p$-torsionfree. 
Then $\dniltop{B} \xrightarrow{\sim} \dniltop{A}$. That is, every topologically $\delta$-nilpotent element of $A$ lifts uniquely to a topologically $\delta$-nilpotent element of $B$.
\end{proposition} 
\begin{proof} 
    Uniqueness of the lift is clear since no element of $I$ can be topologically $\delta$-nilpotent by tautness, so the main task is to prove existence, which is a slightly more elaborate version of the argument in \Cref{tautextensionsdeltanilp}.

    \emph{Step 1: Lifts of multiples of $p$.} 
    Let $a \in A$ be topologically $\delta$-nilpotent, and choose any lift $b \in B$.
We observe first that $pb$ is topologically $\delta$-nilpotent in $B$. 
In fact, it suffices to show that for any $n$, the natural map 
$B \xrightarrow{w_\delta} W(B) \to Q(B/p^n)$ annihilates $pb$. 
But the kernel of $$ Q(B/p^n) \to Q(A/p^n)$$
 is annihilated by $p$, since $B/p^n \to A/p^n$ is a square-zero extension of rings with kernel annihilated by $p$.
It follows that
any element of $p\dniltop{A}$ admits a lift to $\dniltop{B}$.

\emph{Step 2: Lifting through $\delta$.} 
Now we show that if $a \in \dniltop{A}$ and $\delta(a)$ admits a lift $r \in \dniltop{B}$, then $a$ admits a lift to $\dniltop{B}$.
In fact, choose any lift $b_0 \in B$ of $a$.
This means that $\delta(b_0) = r + \eta$ for some $\eta \in I$. 
We try to find $\epsilon$ such that $\delta(b_0 + \epsilon) = r$. 

To do this, we use 
that $b_0$ is topologically nilpotent, so $\phi^N(b_0) \in pB$ for some $N \gg 0$.
Since $p I = 0$, we can apply \Cref{squarezerolemma}
to obtain 
\[
\delta(b_0+\epsilon)= \delta(b_0) + \delta(\epsilon) = r+\eta+\delta(\epsilon).
\]
Since $\delta: I \to I$ is an isomorphism, we can find a unique $\epsilon \in I$ such that $\delta(\epsilon) = -\eta$. 
It follows that $b \stackrel{\mathrm{def}}{=} b_0 + \epsilon$ is a lift of $a$ such that $\delta(b) = r$, which forces $b \in \dniltop{B}$.

\emph{Step 3: Conclusion.}
Given $a \in \dniltop{A}$, 
we observe that $\delta^n(a) \in p \dniltop{A}$ for $n \gg 0$.
This follows because \(w_\delta(\delta^n(a)) \in \hw(p^3A)\) for \(n\gg 0\), so \Cref{dividingtopdnilbyp} applies.
This means that $\delta^n(a)$ admits a lift to $\dniltop{B}$ by Step 1. But by applying Step 2 repeatedly, we find that $a$ admits a lift as well. 
\end{proof}

\begin{corollary}\label{topdnilquotienttautequiv}
Let $A$ be a $p$-torsionfree $\delta$-ring.
    Then the map 
    $ A \to A/\dniltop{A}$ is a taut equivalence.
\end{corollary} 
\begin{proof} 
This is proved in the same way as \Cref{tautdeltanilpdescent}, using
\Cref{topdnilliftstautextension} in place of \Cref{tautextensionsdeltanilp}.
\end{proof}

In order to apply \Cref{topdnilquotienttautequiv}, 
given a $\delta$-ideal (more generally, we can work with  a nonunital $\delta$-ring $I$), 
we need a criterion for all elements of $I$ to be topologically $\delta$-nilpotent.
\begin{proposition}\label{topdnilnonunitalcriterion}
Let $I$ be a nonunital, $p$-torsionfree $\delta$-ring. Suppose that: 
\begin{enumerate}
    \item  Every element $x \in I$ is annihilated by a power of $F$.  
    \item Every $x \in I$ admits divided powers in $I$ that converge to zero in the $p$-adic topology. 
\end{enumerate}
Then every element of $I$ is topologically $\delta$-nilpotent.
\end{proposition} 
\begin{proof} 
Let $x \in I$ be given. 
We need to show that 
$w_\delta(x) \in \hw(I) \subset W(I)$. Without loss of generality, we may assume
that $I$ is $p$-local. 

Since $I$ has divided powers, we have $p \mid F(x)$.
Therefore, we may 
consider the element 
$$w_\delta(x) - V( w_\delta(F(x)/p)) \in W(I).$$
This element is annihilated by $F$, whence it belongs to $\gasharp(I) = \ker(F: W(I) \to W(I))$ and thus records the divided powers on $x$. Since the divided powers on $I$ are $p$-adically nilpotent, 
it follows that $w_\delta(x) - V( w_\delta(F(x)/p)) \in \hw(I) \subset W(I)$ (cf.~\Cref{hwFgahatsharp}). 
Modulo $\hw(I)$, we therefore have $w_\delta(x) \equiv V( w_\delta(F(x)/p))$.
Repeating this and using the local nilpotence of $F$ on $I$, we get $w_\delta(x) \in \hw(I)$, as desired.
\end{proof}

\begin{lemma}\label{pFtopdnil}
Let $A$ be a $\delta$-ring with bounded $p$-power torsion. 
For $x \in A$ and $n \geq 0$, one has $p^n x \in \dniltop{A}$ if and only if $F^n(x) \in \dniltop{A}$.
\end{lemma} 
\begin{proof} 
This follows from \Cref{ppowertorsioninQ}, applied to $A/p^m$ for each $m > 0$.
\end{proof} 

\begin{proposition} 
    \label{SEStopdnil}
Let $0 \to I' \to I \to I'' \to 0$ be a short exact sequence of nonunital $\delta$-rings, each of which is $p$-torsionfree.
Suppose that: 
\begin{enumerate}
    \item  Each element of $I', I''$ is topologically $\delta$-nilpotent. 
    \item 
	$F$ is locally nilpotent on $I'$ and $I''$ (hence on $I$). 
\end{enumerate}
Then each element of $I$ is topologically $\delta$-nilpotent.
\end{proposition} 
\begin{proof} 
Let $x \in I$, and suppose that $F^N(x) = 0$. 
By assumption, we can choose $i \gg 0$ such that 
$$\delta^i(x)
	= p^N z + x',$$ for some $x' \in I', z \in I$. 

Applying $F^N$ to both sides, we find that 
	$p^N F^N(z) \in I'$, whence $F^N(z) \in I'$ by the $p$-torsionfreeness of $I''$.
It follows that $F^N(z) \in \dniltop{I'} \subset \dniltop{I}$. 
However, this means that $p^N z \in \dniltop{I}$ by \Cref{pFtopdnil}, and therefore $\delta^i(x) \in \dniltop{I}$.
 It is also straightforward to see that $\delta^j(x)$ is topologically nilpotent for all $j$, since any element annihilated by a power of $F$ is topologically nilpotent. The result follows. 
\end{proof}

\subsection{Application to $\chW$}

\begin{proposition} 
\label{tautrigidityofchWasdcart}
Let $R$ be a $p$-complete ring with bounded $p$-power torsion, such that $(R/p)_\mathrm{red}$ is perfect. 
Then the $\delta$-ring $\chW(R)$ is taut rigid.
\end{proposition} 
\begin{proof} 
We have a short exact sequence
\[
0 \to \hw(R) \to \chW(R) \to \qperf(R) \to 0,
\]
from \eqref{secondexactseq} and passage to the limit.  
Consider the torsion in $\hw(R)$. 
Given a torsion element $x \in \hw(R)$, all the Witt components of $x$ must be torsion and therefore nilpotent (since they are topologically nilpotent and $R$ has bounded $p$-power torsion).
Moreover, there can only be finitely many nonzero components, again by the bounded torsion hypothesis. 
This means that $x$ is $\delta$-nilpotent. 
It follows from \Cref{tautdeltanilpdescent} that the map 
$\chW(R) \to \chW(R)/\chW(R)[p^\infty]$ is a taut equivalence.
Next, the image of $\hw(R)$ in the $p$-torsionfree quotient  $\chW(R)/\chW(R)[p^\infty]$ is topologically $\delta$-nilpotent, so by \Cref{topdnilquotienttautequiv} the map 
$\chW(R) \to \qperf(R)$ is a taut equivalence. 
Since $\qperf(R)$ is taut rigid by perfectness, the result follows. 

\end{proof}

\begin{corollary}
\label{chWasuniversaltaut}
Let $R$ be a $p$-completely nilperfect ring with bounded $p$-power torsion. 
Suppose $\chW(R) \to W(R)$ is surjective. 
Then the map $\chW(R) \to W(R)$ is the initial taut square-zero extension of $W(R)$. \end{corollary}
\begin{proof} 
This follows because $\chW(R)$ is taut rigid, and the map $\chW(R) \to W(R)$ is a taut square-zero extension by \Cref{chWtoWtautsquarezero} and passage to the limit. 
\end{proof} 

\section{\dcart\ rings and \dcartwo rings}

The purpose of this section is to develop some further algebra around $\chW$. 
The main tool will be the theory of \dcart\ rings and its 2-primary analog \dcartwo rings, which are $\delta$-rings equipped with a Verschiebung-type operator ($V$, resp.~$\hV$) satisfying certain identities. 
Our main result (\Cref{chWuniversalcartier}) is that $\chW(R)$ satisfies a natural universal property as a left adjoint
in the category of derived $p$-complete \dcart (resp.\ \dcartwo) rings.

\subsection{\dcart\ rings}
\label{sec:dcartrings}

\begin{definition}[{Magidson, \cite[Def.~3.4.4]{Magidson}}]
\label{dcartrings}
A \emph{\dcart ring} is a $\delta$-ring $A$ with $\delta$-ring Frobenius $F: A
\to A$
which is  equipped with an
additive operator $V:
A \to A$ that satisfies the following conditions:
\begin{enumerate} 
\item $F(V(a)) = pa$ for all $a \in A$. 
\item  $a V(b) = V( F(a) b)$ for all $a, b \in A$. 
\item For all $a \in A$,
\begin{equation}
    \delta ( V(a)) = a - p^{p-2} V(a^p) \label{deltaVidentity}
\end{equation}
\end{enumerate}
We let $\ddcart$ denote the category of \dcart rings.
\end{definition}

\begin{remark}
\label{dcart:redundant}
Condition (3) of \Cref{dcartrings} (which is due to Drinfeld) is implied by the
previous conditions after multiplication by $p$ or applying $F$,
and is thus redundant if $A$ is $p$-torsionfree. 
\end{remark}

\begin{remark}
For any $A \in \ddcart$ and $a, b \in A$ and $i \leq j \in \mathbb{Z}_{\geq 0}$,
we have $V^i(a) V^j(b) =
p^i  V^j( F^{j-i} (a) \cdot b)$. This follows from applying the projection formula
repeatedly.
\end{remark}

\begin{remark}
The category $\ddcart$ has all limits and colimits, and the forgetful functor to
rings (or sets) preserves limits and sifted colimits.
\end{remark}

\begin{proposition}
\label{torsionfreedcartring}
The functor $A \mapsto (A, V(A))$ establishes an equivalence between the category of $p$-torsionfree \dcart rings $A$  and the category of pairs $(A, I)$ where $A$ is a $p$-torsionfree $\delta$-ring and $I 
\subset A$ is an ideal such that 
$F$ induces an isomorphism $F|_I: I \simeq pA$. 

\end{proposition}
\begin{proof}
We define the inverse functor: if $(A, I)$ is a pair as above, then we can define $V: A \to A$ to be the inverse of the isomorphism
$F/p: I
 \simeq A$. It is straightforward to check that these two functors are inverse to one another. 
\end{proof}

\begin{example}
\label{ZpasVdelta}
The $\delta$-ring $\mathbb{Z}$ (with $F = \mathrm{id}$) is a \dcart ring with
$V(x) = px$ for all $x
\in \mathbb{Z}$.
\end{example}

\begin{definition}[Dieudonn\'e $\delta$-rings]
\label{def:dieudonne-delta-rings}
A \emph{Dieudonn\'e $\delta$-ring} is a \dcart ring $A$ with the additional
property that $VF = p$ on $A$ (equivalently, $V(1) = p$).
\end{definition}

\begin{remark}
    A \dcart ring $A$ is a Dieudonn\'e $\delta$-ring if and only if $V(1) = p$ in
$A$.
The ring $\mathbb{Z}$ (as in \Cref{ZpasVdelta}) is the initial Dieudonn\'e
$\delta$-ring.
\end{remark}

\begin{remark}
    \label{ptorsionfreeDieudonnedeltaring}
Let $A$ be a $p$-torsionfree ring. To give $A$ the structure of a Dieudonn\'e
$\delta$-ring is equivalent to giving an \emph{injective} ring map $F: A \to A$
such that:
\begin{enumerate}
\item  $F(a) \equiv a^p \pmod{pA}$ for all $a \in A$.
\item
$F(A) \supset pA$.
\end{enumerate}
\end{remark}

\begin{remark}[Perfect $\delta$-rings]
\label{perfectdeltaringdcart}
Let $A$ be a perfect $\delta$-ring, i.e., assume that the Frobenius lift
$F:A\to A$ is an isomorphism. Then $A$ has a unique \dcart structure, given by
\[
V(a)=pF^{-1}(a).
\]
Since $A$ is $p$-torsionfree, the last identity follows. 
Moreover, $A$ is a Dieudonn\'e $\delta$-ring. 
\end{remark}

\begin{example}[Witt vectors as a \dcart ring]
For any ring $R$, the Witt vectors $W(R)$ with its usual $\delta$-structure and
Verschiebung is a \dcart ring.
In fact, one only needs to verify the identity for $\delta V$, and this reduces
to the case where $R$ is $p$-torsionfree, where we can use
\Cref{dcart:redundant}.
When $R$ is an $\mathbb{F}_p$-algebra, $W(R)$ is additionally a
\dieu-ring.
\end{example}

\begin{example}[$\chW$ as a \dcart ring]
    \label{chWdcart}
Assume $p>2$. For $R \in \nperf$,
$\chW(R)$ with $\delta$ and $\hV = V$ has the structure of a
\dcart ring, and the map
$\chW(R) \to W(R)$ is a morphism of
\dcart rings.

The usual Witt vectors $W(R)$ form a \dcart ring.
The ideal $\hw(R) \subset W(R)$ is a $\delta$-ideal which is stable under $V$, which implies that the quotient $Q(R) = W(R)/\hw(R)$ inherits the structure of a \dcart ring such that the quotient map $W(R) \to Q(R)$ is a morphism of \dcart rings.
Moreover, \Cref{FhatV} shows that
the Frobenius $F:Q(R)\to Q(R)$, which is a $\delta$-map, commutes with $\hV$, so it is a map of
\dcart rings. Passing to the
perfection defining $\qperf(R)$ gives $\qperf(R)$ a compatible \dcart structure. 
The pullback definition of $\chW(R)$ then gives the \dcart structure on
$\chW(R)$, and the map $\chW(R)\to W(R)$ is a map of \dcart rings by construction.
\end{example}

\begin{remark}[Witt vectors of a $\delta$-ring]
\label{joyalVcoords}
If $A$ is a $\delta$-ring with associated $\delta$-map $w_\delta: A \to W(A)$, then the map
 \[
\prod_{i\geq0} A \longrightarrow W(A),\qquad
(a_i)_{i\geq0}\longmapsto \sum_{i\geq0}V^i(w_\delta(a_i))
\]
is an isomorphism of abelian groups.  Equivalently, every $x\in W(A)$ admits a
unique expansion
\(
x=\sum_{i\geq0}V^i(w_\delta(a_i))\) for a sequence
$ a_i\in A, i \geq 0$.
In particular, this applies to $A=\mathbb{Z}_p$ with its initial $\delta$-ring
structure. 
This follows because $W(A)$ is 
complete and torsionfree with respect to $V$, $W(A)/V \simeq A$, and 
the zeroth Witt coordinate of $w_\delta(a)$ is $a$. 
\end{remark}

\begin{construction}[{Free \dcart rings, cf.~\cite[Cons.~3.4.5]{Magidson}}]
    \label{cons:freedCartring}
Let
$A$ be a $\delta$-ring; we denote the $\delta$-ring Frobenius by $F: A \to A$.
We can form a \dcart ring $A[V]$, which as an $A$-module is given by
\[ A[V] \stackrel{\mathrm{def}}{=} \bigoplus_{i \geq 0}  F_*^i A . \]
We define the operator $V$ on $A[V]$ by shifting from the $i$th summand to the
$(i+1)$st summand, so
as abelian groups
$A[V] = \bigoplus_{i \geq 0} V^i (A)$.
The structure as a commutative ring is determined by the projection
formula. Explicitly, for $0\leq i\leq j$,
\[
V^i(a)V^j(b)=p^iV^j(F^{j-i}(a)b),
\]
and the case $i>j$ follows by commutativity. Finally, $F$ restricts to
the $\delta$-ring Frobenius on the zeroth summand and satisfies
$F(V^i(a)) = p V^{i-1}(a)$, and $\delta$ is determined by $\delta$ on $A$, and the
$\delta$-ring identities
together with the identity for $\delta V$.\footnote{Instead of checking all
of these
identities individually, one can also use \Cref{torsionfreedcartring} in the
$p$-torsionfree case, and then resolve $A$ by $p$-torsionfree $\delta$-rings.}

One checks \cite[Prop.~3.4.7]{Magidson} that $A \mapsto A[V]$ is the left
adjoint to the forgetful functor from \dcart rings to
$\delta$-rings.\footnote{In \emph{loc.~cit.}, this assumes that $A$ is
$p$-torsionfree, but this extends to the general case by left Kan extension.} In
particular, the free \dcart ring on $n$ generators $x_1, \dots, x_n$ is obtained
as $\mathbb{Z}\left\{x_1, \dots, x_n\right\}[V]$, where
$\mathbb{Z}\left\{x_1,\dots, x_n\right\}$ is the free $\delta$-ring on $n$
generators.
\end{construction}

\begin{example}
The initial \dcart ring is given by $\mathbb{Z}[V]$.
\end{example}
\begin{proposition}
\label{chWZpdescription}
 Consider 
$\chW(\mathbb{Z}_p)\stackrel{\mathrm{def}}{=}
\varprojlim_n \chW(\mathbb{Z}/p^n)$. 
If $p > 2$, then 
 the natural map from the $p$-completion of the initial \dcart ring
\begin{equation}
    \label{ZVtochWZp:eq}
\mathbb{Z}[V]^\wedge_p\longrightarrow \chW(\mathbb{Z}_p)
\end{equation}
is an isomorphism. 
For all $p$ (including $p=2$\footnote{However, when $p = 2$, $\chW(\mathbb{Z}_2)$ is not stable under $V$.}), every element of $\chW(\mathbb{Z}_p) \subset W(\mathbb{Z}_p)$ can be
written uniquely as a $p$-adically convergent Verschiebung expansion
\[
\sum_{i\geq 0}V^i(a_i),\qquad a_i\in\mathbb{Z}_p,\quad a_i\to 0.
\]
\end{proposition}
\begin{proof}
For $p>2$, $\chW(\mathbb{Z}_p)=\varprojlim_n\chW(\mathbb{Z}/p^n)$ is a
$p$-complete \dcart ring.  Hence we obtain a unique map of \dcart rings 
as in \eqref{ZVtochWZp:eq}.
Moreover, by \Cref{chWofArtinian}, each map
$\chW(\mathbb{Z}/p^n)\to W(\mathbb{Z}/p^n)$ is injective; passing to the
inverse limit, we get an injection
$\chW(\mathbb{Z}_p)\hookrightarrow W(\mathbb{Z}_p)$.  
The composite map 
$\mathbb{Z}[V]^\wedge_p\to\chW(\mathbb{Z}_p)\hookrightarrow W(\mathbb{Z}_p)$ is injective 
(cf.~\Cref{joyalVcoords})
and  the image consists of the sums
$\sum_iV^i(a_i)$ with $a_i\in\mathbb{Z}_p$ and $a_i\to0$. 
This implies that $\mathbb{Z}[V]^\wedge_p\to\chW(\mathbb{Z}_p)$ is injective. 

It remains to prove surjectivity.  By \Cref{chWofArtinian}, after passing to
the inverse limit we have
\(
\chW(\mathbb{Z}_p)=W(\mathbb{F}_p)\oplus \hw(\mathbb{Z}_p)
\subset W(\mathbb{Z}_p).
\)
The summand $W(\mathbb{F}_p)=\mathbb{Z}_p$ is already in the image of $\mathbb{Z}[V]^\wedge_p$, so it
suffices to show that any $\epsilon\in\hw(\mathbb{Z}_p)$ belongs to this image as well.  By
\Cref{joyalVcoords}, write 
\[
\epsilon=\sum_{i\geq0}V^i(a_i),\qquad a_i\in\mathbb{Z}_p.
\]
If
$\gh_n(\epsilon)$ denotes the $n$th ghost component of $\epsilon$, then
\[
\gh_n(\epsilon)=\sum_{i=0}^n p^ia_i,\qquad
a_n=\frac{\gh_n(\epsilon)-\gh_{n-1}(\epsilon)}{p^n}\quad(n\geq1).
\]
By \Cref{ghostcrit}, there is a sequence $c_n\to+\infty$ such that
$\gh_n(\epsilon)\in p^{n+c_n}\mathbb{Z}_p$.  The displayed formula then gives
$a_n\to 0$, and
hence $\epsilon$ belongs to the image of
$\mathbb{Z}[V]^\wedge_p$.
\end{proof}

For future reference, we record also the following normal form for free Dieudonn\'e \(\delta\)-algebras. In this case, we obtain the normal form when the $\delta$-ring Frobenius is split injective, which is the case for free \(\delta\)-rings.

\begin{proposition}[Normal form for Dieudonn\'e $\delta$-algebras]
    \label{prop:normalformDieudonne}
Let \(B\) be a  \(\delta\)-ring and suppose that the $\delta$-ring Frobenius \(
\varphi:B\longrightarrow \varphi_*B
\)
admits a \(B\)-linear retraction. Choose a complement
\(
\varphi_*B=\varphi(B)\oplus C.
\)
If \(\mathbb D_B\) denotes the free Dieudonn\'e \(\delta\)-algebra under \(B\),
then the map
\[
B\oplus\bigoplus_{n\geq0}\varphi_*^nC
   \longrightarrow \mathbb D_B,
\qquad
(b,(c_n))\longmapsto b+\sum_{n\geq0}V^{n+1}(c_n),
\]
is an isomorphism of \(B\)-modules.
\end{proposition}

\begin{proof}
Using the relation $VF = p$, 
it is straightforward to see that the displayed map is surjective, i.e., that its image is a Dieudonn\'e \(\delta\)-subalgebra of \(\mathbb D_B\) containing \(B\).
Thus, it remains only to prove injectivity. 
The perfection $B_{\mathrm{perf}}$ of $B$ is a perfect $\delta$-ring, and hence a Dieudonn\'e $\delta$-ring by \Cref{perfectdeltaringdcart}.
 One checks that the resulting map
\(\mathbb D_B\to B_{\mathrm{perf}}\) 
restricts to an injection on 
$B \oplus\bigoplus_{n\geq0}\varphi_*^nC$, whence the result.
\end{proof}

\begin{proposition}
\label{BtoBVtautequivalence}
Let $p > 2$, and let $B$ be any $\delta$-ring.
The map 
$L_B[F^{-1}] \to L_{B[V]}[F^{-1}]$ is an isomorphism on 1-truncations. 
In particular, the natural map $B \to B[V]$  is a taut equivalence. 
\end{proposition}

\begin{proof}
By taking resolutions, we may assume that $B$ is a free $\delta$-ring over $\mathbb{Z}_{(p)}$.
In this case, because $L_B$ is concentrated in degree zero, it suffices to show that the 1-truncation of $L_{B[V]_{\leq n}/B}[F^{-1}]$ vanishes. 

We prove the result more generally for the filtered pieces $B[V]_{\leq n} = \bigoplus_{i=0}^n F_*^i B$ of $B[V]$, which are $\delta$-subrings.
Passing to the colimit gives the result for $B[V]$. 

The $n$th power of Frobenius on $B[V]_{\leq n}$ factors through $B$. 
In particular, the map $B \to B[V]_{\leq n}$ induces an isomorphism on perfections. 
This in particular implies that the transitivity triangle for cotangent complexes gives a cofiber sequence 
$L_{B}[F^{-1}] \to L_{B[V]_{\leq n}}[F^{-1}]  \to L_{B[V]_{\leq n}/B}[F^{-1}].$
Moreover, it implies that $(pF)^n$ annihilates $L_{B[V]_{ \leq n}/B}$, since $pF$ is the map induced by functoriality of the Frobenius. It follows that $L_{B[V]_{ \leq n}/B}[F^{-1}]$ is annihilated by $p^n$ and therefore agrees with its $p$-completion $\lt_{B[V]_{ \leq n}/B}$. 
We will show that the 1-truncation of $\lt_{B[V]_{ \leq n}/B}$ vanishes, which will imply the result.

Since $B$ is assumed to be a free 
$\delta$-ring, 
$B[V]_{\leq n}$ is flat over $B$, and 
$\lt_B$ is concentrated in degree zero. 
By \Cref{tautequivdiagonalcriterion}, it suffices to show that 
$B[V]_{\leq n} \otimes_B B[V]_{\leq n} \to B[V]_{\leq n}$ is a taut equivalence.
For this, it suffices by \Cref{topdnilquotienttautequiv} to show that the kernel $J$ of
\[
B[V]_{\leq n}\otimes_B B[V]_{\leq n}\to B[V]_{\leq n}
\]
consists of topologically $\delta$-nilpotent elements.

We apply \Cref{topdnilnonunitalcriterion} to the nonunital $\delta$-ring $J$.
The kernel $J$ is generated by the classes
\[
\lambda_{j,x}=V^j(x)\otimes 1-1\otimes V^j(x),
\qquad x\in B,\ 1\leq j\leq n.
\]
It follows that $J$ has divided powers, since $J$ is a $\delta$-ideal and
$F(J)\subset pJ$, cf.~\cite[\S 2.5]{bhattscholze}. Moreover, $\lambda_{j,x}$ is annihilated by $F^j$, so to
conclude it suffices to show that its divided powers are topologically ($p$-adically) nilpotent. In fact,
\[
\lambda_{j,x}^p
=
V^j(x)^p\otimes 1-1\otimes V^j(x)^p
+\sum_{i=1}^{p-1}\binom{p}{i}(-1)^{p-i}V^j(x)^i\otimes V^j(x)^{p-i}.
\]
This is divisible by $p^2$, whence $\gamma_p(\lambda_{j,x})\in pJ$.
Thus the divided powers of the generators tend $p$-adically to zero, and the
result follows from the criterion.
The general case follows by resolving $B$ by free $\delta$-rings.

\end{proof}

\begin{proposition}[Injectivity of $V$]
\label{Visinjective}
Let $A \in \ddcart$.
Then $V: A \to A$ is injective.
\end{proposition}
\begin{proof}
Suppose $V(a) = 0$.
Then we find that $F(V(a)) = pa = 0$.
Moreover, we find
\[0 = \delta(V(a)) = a - p^{p-2} V(a^p)  .\]
For $p >2$, this already yields $a = 0$: indeed $pa=0$ implies
$pa^p=0$, so
\[
pV(a^p)=V(pa^p)=0,
\]
and therefore $p^{p-2}V(a^p)=0$.
For $p = 2$, we obtain
$a = V(a^2)$. Since $2a=0$, also $2a^4=0$. Squaring the equality and using
$V(x)^2=V(2x^2)$ at $p=2$, we get
\[
a^2=V(a^2)^2=V(2a^4)=0.
\]
Thus $a=V(a^2)=0$.
\end{proof}

\begin{lemma}
\label{deltaonVcongruence}
Let $A$ be a \dcart ring, and let $i \geq 1$.
Suppose $a_1, a_2 \in A$ are congruent modulo $V^i(A)$.
Then $\delta(a_1), \delta(a_2)$ are congruent modulo $V^{i-1}(A)$.
Moreover, the induced map
\[ \delta: V^i(A)/V^{i+1}(A) \to V^{i-1}(A)/V^{i}(A)  \]
is the inverse to the isomorphism
$V: V^{i-1}(A)/V^{i}(A) \xrightarrow{\sim}
V^i(A)/V^{i+1}(A)$.
\end{lemma}
\begin{proof}
In fact,
suppose $a_2 = a_1 + b$.
Then $\delta(a_2) = \delta(a_1) + \delta(b)  -  \sum_{k=1}^{p-1} \frac{1}{p}
\binom{p}{k} a_1^{k} b^{p-k}$,
so it suffices to show that
$\delta(b) \in V^{i-1}(A)$.

This is clear when $i = 1$.
For $i > 1$, write $b = V(b')$ with $b' \in V^{i-1}(A)$. Then
$\delta(b) = b' - p^{p-2}V( b'^p) $, which is congruent to $b'$ modulo $V^i(A)$.
\end{proof}

\begin{theorem}
\label{automaticV}
Let $A, B \in \ddcart$. Consider a map of $\delta$-rings $f: A \to B$ such that
$f$ carries the ideal $V(A) \subset A$ into the ideal $V(B) \subset B$.
If $p > 2$, $f$ is a map in $\ddcart$ (i.e., $f$ respects $V$).
If $p = 2$ and $f$ additionally sends $V(1)$ to $V(1)$ or if $B$ is $V$-adically separated (i.e., $\bigcap_{i \geq 0} V^i(B) = 0$), then $f$ is a map in
$\ddcart$.
\end{theorem}

This result admits a simple, abstract proof using \Cref{BtoBVtautequivalence}. 
We give an elementary, computational proof here.  The more abstract proof will be given in the case $p = 2$ below (\Cref{automatichV}), where an elementary proof is more difficult.  The analogous proof also works for $p > 2$.

\begin{proof}
By assumption and \Cref{Visinjective}, there is a map (necessarily additive) $g:
A \to B$ such that
for each $a \in A$,
\begin{equation} \label{fVVg}
f( V(a)) = V( g(a)).
\end{equation}
Applying $F$ to both sides and using that $F$ commutes with $f$, we find $pf =
pg$.
We apply $\delta$ to both sides, noting that $f$ commutes with $\delta$; thus,
we find
\[ f( a - p^{p-2} V(a^p)) = g(a) - p^{p-2} V(g(a)^p).  \]
Moreover, by \eqref{fVVg}, we can write the left-hand side as
$f(a) - p^{p-2} V(g(a^p))$. So we get the equation in $g$,
\begin{equation}  g(a) = f(a) + p^{p-2}  V(g(a)^p - g(a^p)).  \end{equation}
Writing $g = f+ h$ and using $ph = 0$ and $f(a)^p = f(a^p)$, we find
the identity
\begin{equation} \label{hfunctionaleq} h(a) = p^{p-2}V  (  h(a)^p - h(a^p) ).
\end{equation}
When $p > 2$,
this already gives $h = 0$ since we have $ph = 0$. 

If $p = 2$ and $B$ is $V$-adically separated, then  \eqref{hfunctionaleq} also implies that $h =0$. 
Suppose now $p = 2$\footnote{This is also a consequence of \Cref{dcartisdcartwoplusepsilon} and \Cref{automatichV} below.} and we have $f(V(1)) =
V(1)$.
In \eqref{fVVg}, we replace $a$ by $Fa$.
Then
\[
V(g(F(a)))=f(V(F(a)))=f(V(1)a)=V(1)f(a)
=V(F(f(a)))=V(f(F(a))),
\]
which gives
 $g \circ F = f \circ F$.
This means that $h$ vanishes on the image of $F$.
Since $2h = 0$ and $h$ is additive, this means that $h$ also vanishes on $p$th powers, whence we get
$h(a) =  V(h(a)^2)$.
Substituting this identity into itself, we get
	$h(a) = V( V(h(a)^2)^2) =0 $, using again $2h =0$.
	\end{proof}

\begin{example}[Preserving the Verschiebung ideal is not enough at $p=2$]
\label{VidealnotVmap}
Consider the $\delta$-Cartier ring $Q( \mathbb{Z}/4) = W( \mathbb{Z}/4)/\hw(\mathbb{Z}/4)$. 
The $\delta$-Cartier structure on $W( \mathbb{Z}/4)$ descends to $Q( \mathbb{Z}/4)$. 
Consider the map 
$F: Q( \mathbb{Z}/4) \to Q( \mathbb{Z}/4)$ given by the $\delta$-ring Frobenius. 
Then $F$ is a $\delta$-ring map that preserves the ideal $V(Q( \mathbb{Z}/4)) = \hV( Q( \mathbb{Z}/4))$ since $F, \hV$ commute (\Cref{FhatV}).
However, $F$ does not commute with $V$ since 
$F(V(1)) = 2 \neq V(F(1)) = V(1) \in Q( \mathbb{Z}/4)$ by \Cref{impossibilityinWZ4}, so $F$ is not a map of   $\delta$-Cartier rings. 
\end{example}

We now record the following result, which is a special case of the results of \cite{Magidson} on derived $\delta$-Cartier rings; for convenience, we include an elementary proof for classical rings.

\begin{definition}[$V$-completeness]
Let \(A\) be a \dcart ring.  We say
that \(A\) is \emph{\(V\)-complete} if
\[
A \xrightarrow{\sim}\varprojlim_n A/V^n(A).
\]
For any \dcart\ ring \(B\), its \emph{\(V\)-completion} is
\(\varprojlim_n B/V^n(B)\). 
\end{definition}

\begin{theorem}[{The $V$-complete Cartier theorem, cf.~\cite[Theorem~1.2.10]{Magidson}}]
\label{Vcompletecartierwitt}
The functor $A\mapsto A/V$ from \dcart rings to rings commutes with colimits.
Its right adjoint carries a ring $R$ to the \dcart ring $W(R)$. Moreover,
$R\mapsto W(R)$ induces an equivalence between commutative rings and
$V$-complete \dcart rings.
\end{theorem}
\begin{proof}
Let $A$ be a $V$-complete \dcart ring. By Joyal's theorem
(\Cref{joyal:witt}), there is a unique map of $\delta$-rings
\[
A\longrightarrow W(A/V)
\]
whose composite with $W(A/V)\to A/V$ is the tautological map. In particular it
carries $V(A)$ into $VW(A/V)$, so it is a map of \dcart rings by
\Cref{automaticV} (using $V$-separatedness at $p = 2$). Since $V$ is injective on both sides and the map is an
isomorphism modulo $V$, it is an isomorphism after passing to all quotients
modulo $V^n$; $V$-completeness then gives $A\simeq W(A/V)$.

For rings $R$ and $R'$, a \dcart map $W(R)\to W(R')$ is the same as a $\delta$-map
$W(R)\to W(R')$ carrying $VW(R)$ into $VW(R')$, by \Cref{automaticV}. By Joyal's theorem, this is
equivalent to a ring map $R\to R'$. The remaining claims follow.
\end{proof}

\subsection{\dcartwo rings}\label{sec:dcartwo}

In this subsection, we develop a 2-primary variant of the theory of \dcart rings, which we call \dcartwo rings; a basic example is the sheared Witt vectors. The definition of \dcartwo rings was communicated to us by Drinfeld. We prove analogs of all the results for \dcart rings. Although the proofs are more involved, the theory is better behaved than the theory of \dcart rings at $p=2$, and in particular, the analog of \Cref{automaticV} holds without any additional assumptions.

\subsubsection{Definition and basic properties}
\begin{definition}[\dcartwo rings, Drinfeld]
Suppose $p=2$. A
\emph{\dcartwo ring} consists of a
$\delta$-ring $A$ (with $\delta$-ring Frobenius denoted $F: A \to A$)
equipped with an additive map $\hV: A \to A$ such that the following
properties hold.
 
\begin{enumerate}
\item Let $\two = F(\hV(1)) \in A$.
Then
$\two^2 = 4$ and $\delta(\two) = -1$.\footnote{This implies
$F(\two) = 2$. Strictly speaking, $\delta(\two) = -1$ is a consequence of the identity for $\delta(\hV(a))$ for $a = 1$, but we include it as a separate axiom for clarity.}
\item $F(\hV(a)) = \two a$ for all $a \in A$.
\item $a \hV(b) = \hV(F(a) . b)$ for all $a, b \in A$.
\item For all $a \in A$, \begin{equation}
    \delta(\hV(a)) = -a + \hV(\two \delta(a)) \label{deltaVtildeidentity}
\end{equation}
\end{enumerate}
We write $\ddcartwo$ for the category of \dcartwo rings.
\end{definition}

We begin by recording a few basic consequences of the identities. 
\begin{lemma} 
    \label{hVof2tilde}
Let $A$ be a \dcartwo ring.  
Then $\hV(\two) = 2 + \two$. 
More generally, for any $a \in A$, one has 
\begin{equation}
    \label{hVtwoF}
    \hV( \two F(a)) = (2 + \two)a.  \end{equation}

Conversely, if $A$ is a 2-torsionfree $\delta$-ring with an additive map $\hV: A \to A$ such that the identities in the definition of a \dcartwo ring hold except for the last one, and if $\hV(\two) = 2 + \two$, then the last identity also holds.
\end{lemma}
\begin{proof}
To see this, we write
(where we use the \dcartwo identities)
\begin{align*} 
    2 + \two &= 2 + F( \hV(1)) \\
    & = 2 + \hV(1)^2 + 2 \delta( \hV(1)) \\
    & = 2 + \hV( F( \hV(1))) + 
2(-1) \\
& = \hV( \two). 
\end{align*}
The more general identity \eqref{hVtwoF} follows from the projection formula. 

Conversely, if $A$ is a 2-torsionfree $\delta$-ring with an additive map $\hV: A \to A$ such that the identities in the definition of a \dcartwo ring hold except for the last one, and if $\hV(\two) = 2 + \two$, then we can check the last identity as follows:
\begin{align*}
2 \delta(\hV(a)) & = 
F( \hV(a))
- \hV(a)^2  \\ 
& = \two a - \hV( a F \hV( a ))\\
& = \two a - \hV( \two a^2).
\end{align*}
Moreover, we have
\begin{align*}
\hV( \two a^2) 
& = \hV( \two F(a) - 2 \two \delta(a)) \\ 
& = (2 + \two)a - 2 \hV( \two \delta(a)), 
\end{align*}
where we have used \eqref{hVtwoF} in the last step. Substituting the second identity into the first gives the desired identity for $\delta(\hV(a))$.

\end{proof}

\begin{lemma}
\label{hVbasicidentities}
Let $A$ be a \dcartwo ring. Then:
\begin{enumerate}
\item The image $\hV(A)$ is an ideal of $A$.
\item One has
\[
(2+\two)A\subset \hV(A),\qquad (2-\two)\hV(A)=0. 
\]
\end{enumerate}
\end{lemma}
\begin{proof}
The first assertion follows from the projection formula
$a\hV(b)=\hV(F(a)b)$. The inclusion $(2+\two)A\subset \hV(A)$ follows from
\Cref{hVof2tilde}. If $x=\hV(y)$, then
\[
\two x=\two\hV(y)=\hV(F(\two)y)=\hV(2y)=2\hV(y)=2x,
\]
so $(2-\two)x=0$.
\end{proof}

\begin{lemma}
\label{deltaonhVcongruence}
Let $A$ be a \dcartwo ring.
Let $i\geq 1$.
Suppose $a_1,a_2\in A$ are congruent modulo $\hV^i(A)$.
Then $\delta(a_1),\delta(a_2)$ are congruent modulo $\hV^{i-1}(A)$.
\end{lemma}
\begin{proof}
In fact, suppose $a_2=a_1+b$ with $b\in \hV^i(A)$.
Then $\delta(a_2)=\delta(a_1)+\delta(b)-a_1b$, so it suffices to show that
$\delta(b)\in \hV^{i-1}(A)$.
This is clear when $i=1$.
For $i>1$, if $b=\hV(b')$, then
\[
\delta(b)=-b'+\hV(\two\delta(b')).
\]
Inductively, $\delta(b')\in \hV^{i-2}(A)$, whence the result for $b$.
\end{proof}

\begin{remark} 
Let $A$ be a \dcartwo ring such that $\hV(1) = 2$. 
Then $\hV$ defines on $A$ the structure of a Dieudonn\'e $\delta$-ring as in \Cref{def:dieudonne-delta-rings}, and so 
one obtains an embedding of the category of Dieudonn\'e $\delta$-rings into the category of 
\dcartwo rings whose essential image is exactly those 
\dcartwo rings where $\hV(1) = 2$. 
\end{remark}

\subsubsection{Comparison with \dcart rings; $\hV$-completeness}
Although the identities for \dcartwo rings are much more complicated than those of \dcart rings, the basic construction of \dcartwo rings is actually via the following forgetful functor from \dcart rings.
We will later show (\Cref{comparisonofdcartanddcartwoonVcompleteor2invertible}) that this forgetful functor is an equivalence on certain subcategories, enabling us to reduce many results about \dcartwo rings to (easier) results about \dcart rings.
\begin{construction}
    [\dcartwo rings from \dcart rings at $p=2$]
    \label{dcartwofromdcart}
Suppose $p=2$ and let $A$ be a \dcart ring. 
We construct on $A$ the structure of a \dcartwo ring.

Put\footnote{Note that in $W(\mathbb{Z})$, $\epsilon$ maps to $[-1]$. The identities for $\epsilon, \two$ can be checked in the initial $\delta$-Cartier ring, which injects into $W(\mathbb{Z})$.}
\[
\epsilon=V(1)-1,\qquad \two=2\epsilon,
\]
and define $\hV: A \to A$ by $\hV(a) = V(\epsilon a)$, so that $F\hV(a) = \two a$ for all $a \in A$.
The projection formula for $V$ gives the projection formula for $\hV$. It remains to check the identities for $\two$ and for $\delta(\hV(a))$.

To this end, we use the following identities: 
\begin{enumerate}
\item $2\epsilon=\two$ (by definition). 
\item $\delta(\epsilon)=0$.
Indeed, the identity $\delta(V(1))=1-V(1)$ gives
\[
\delta(\epsilon)=\delta(V(1)-1)=\delta(V(1))+\delta(-1)+V(1)=(1-V(1))-1+V(1)=0.
\]
\item 
$\epsilon^2 = F(\epsilon) = 1$. Indeed, $\delta(\epsilon)=0$ gives
$F(\epsilon)=\epsilon^2$, while $F(\epsilon)=F(V(1)-1)=F(V(1))-1=1$.
\item $\hV(\epsilon)=1+\epsilon$.
Indeed, $\hV( \epsilon) = V(\epsilon^2) = V(1) = 1 + \epsilon$.
\end{enumerate} 

The above identities imply that $\two^2 = 4$ and 
$\delta(\two) = -1$, which verifies the identities for $\two$.

Finally, for $a \in A$, 
\begin{equation} \label{aux:deltaVhV}
\delta(\hV(a))=\delta(V(\epsilon a))
=\epsilon a-V((\epsilon a)^2)=\epsilon a-V(a^2) = \epsilon a - \hV( \epsilon a^2).
\end{equation}
On the other hand,
\begin{equation} \label{aux:atimesoneplusepsilon}
a(1+\epsilon)=aV(1)=V(F(a))=V(a^2)+V(2\delta(a)) = \hV( \epsilon a^2 ) + \hV( \two \delta(a)).
\end{equation}
Combining 
\eqref{aux:deltaVhV} and \eqref{aux:atimesoneplusepsilon} gives
the desired identity for $\delta(\hV(a))$. 
\end{construction}

\begin{example}[$\chW$ as a \dcartwo ring]
\label{chWdcartwo}
For any $2$-completely nilperfect ring $R$ with bounded $2$-power torsion,
$\chW(R)$ with $\delta$ and $\hV$ has the structure of a
\dcartwo ring, and the map
$\chW(R) \to W(R)$ is a morphism of
\dcartwo rings.

Indeed, the usual Witt vectors form a \dcart ring, and \Cref{dcartwofromdcart}
gives it the structure of a \dcartwo ring.
In this case
$V(1)=1+[-1]$, so $\epsilon=V(1)-1=[-1]$ and we get
$\hV(x)=V([-1]x)$. Thus $W(R)$ is a
\dcartwo ring, with
\(
\two=F(\hV(1))=2[-1].
\)
The quotient map $W(R)\to Q(R)$ is compatible with $\delta$, $F$, and $\hV$ by
the construction of these operations on $Q$. Moreover, \Cref{FhatV} shows that
the Frobenius $F:Q(R)\to Q(R)$, which is a $\delta$-map, commutes with $\hV$, so it is a map of
\dcartwo rings. The remainder of the construction is just as in \Cref{chWdcart}. Note that $F: Q(R) \to Q(R)$ usually does not commute with the usual Verschiebung, and as a result the \dcartwo ring structure on $\chW(R)$ generally does not arise from a \dcart ring structure on $\chW(R)$ via \Cref{dcartwofromdcart}.
\end{example}

\begin{proposition}
    \label{dcartisdcartwoplusepsilon}
The functor 
    from \Cref{dcartwofromdcart} 
    establishes an equivalence of categories between: 
    \begin{enumerate}
        \item the category of \dcart rings, and
        \item the category of \dcartwo rings $A$ equipped with a choice of $\epsilon\in A$ such that $2\epsilon=\two$, $\delta(\epsilon)=0$, and $\hV(\epsilon)=1+\epsilon$.
    \end{enumerate}
\end{proposition}
\begin{proof}
    We construct the inverse of the functor of \Cref{dcartwofromdcart}.
 Let $A$ be a \dcartwo ring,
and suppose that $\epsilon\in A$ satisfies
the three specified identities. 
Applying $F$ to the last identity gives
\[
2\epsilon^2=\two\epsilon=F(\hV(\epsilon))
=F(1+\epsilon)=1+F(\epsilon)=1+\epsilon^2,
\]
where the last equality uses $\delta(\epsilon)=0$. Thus $\epsilon^2=1$ and $\two \epsilon = 2$.
Now define
\[
V(a)=\hV(\epsilon a).
\]
Then
\[
F(V(a))=F(\hV(\epsilon a))=\two\epsilon a=2a,
\]
and the projection formula for $\hV$ gives the projection formula for $V$. 
It remains to check the identity for $\delta V$. Since
$\delta(\epsilon)=0$ and $\epsilon^2=1$, one has
$\delta(\epsilon a)=\delta(a)$ by the identity for $\delta(xy)$. Hence
\[
\delta(V(a))
=\delta(\hV(\epsilon a))
=-\epsilon a+\hV(\two\delta(a)).
\]
On the other hand, applying the projection formula to
$\hV(\epsilon)=1+\epsilon$ gives
\[
a(1+\epsilon)=a\hV(\epsilon)=\hV(F(a)\epsilon)
=\hV(\epsilon a^2+\two\delta(a))
=V(a^2)+\hV(\two\delta(a)).
\]
Substituting this in the preceding formula gives
\[
\delta(V(a))=a-V(a^2),
\]
so $V$ defines a \dcart structure on $A$.

It is straightforward to check that both functors are inverse to each other. 
\end{proof}

\begin{definition}[$\hV$-completeness]
    We say that a \dcartwo ring $A$ is \emph{$\hV$-complete} if the natural map $A \to \varprojlim_{i} A/\hV^i(A)$ is an isomorphism, i.e., $A$ is complete with respect to the linear topology induced by the system of ideals $\hV^i(A) \subset A$.\footnote{We will show in \Cref{hVisinjective} below that $\hV: A \to A$ is injective.}
    Given a \dcartwo ring $B$, we write $\widehat{B}$ for the $\hV$-completion of $B$, i.e., $\widehat{B} = \varprojlim_i B/\hV^i(B)$ is the completion of $B$ with respect to the above topology.  
    Using \Cref{deltaonhVcongruence}, one checks that $\widehat{B}$ inherits the structure of a \dcartwo ring such that the natural map $B \to \widehat{B}$ is a morphism of \dcartwo rings.
\end{definition}

\begin{proposition}
\label{comparisonofdcartanddcartwoonVcompleteor2invertible}
    The functor 
    from \dcart rings to \dcartwo rings given by \Cref{dcartwofromdcart} restricts to an equivalence between: 
\begin{itemize}
    \item The full subcategories of $V$-complete \dcart rings and $\hV$-complete \dcartwo rings.
    \item     The full subcategories of \dcart rings and \dcartwo rings where $2$ is invertible. 
\end{itemize}
\end{proposition}

\begin{proof}
    In both cases, we check that for such a $\dcartwo$ ring, there is a  unique choice of $\epsilon$ satisfying the identities of \Cref{dcartisdcartwoplusepsilon}; this defines the inverse functor.

    Let $A$ be a $\hV$-complete \dcartwo ring. 
    Define $\epsilon \in A$ by the formula
    \[ \epsilon = - \sum_{i = 0}^\infty \hV^i(1).  \]
    Then $\epsilon$ is well-defined since $A$ is $\hV$-complete, and it satisfies the equation 
    $\epsilon + 1 = \hV( \epsilon)$; 
    moreover, $\hV$-completeness implies that $\epsilon$ is the unique element of $A$ satisfying this equation. 
    From this, we find by applying $\delta$ 
    $$ \delta(\epsilon) - \epsilon = -\epsilon + \hV(\two \delta(\epsilon)), \quad \text{so} \quad \delta(\epsilon) = \hV ( \two \delta(\epsilon)). $$
    Applying this inductively and using $\hV$-separatedness of $A$, we find $\delta(\epsilon) \in \bigcap_{i \geq 0} \hV^i(A) = 0$.\footnote{This argument shows that a $\hV$-separated \dcartwo ring can admit at most one compatible structure as a \dcart ring.} 

    Using the identity $\hV(\two) = 2 + \two$ from \Cref{hVof2tilde}, we get
    \begin{align*} 2 \epsilon - \two =    
        \sum_{i = 1}^\infty \hV^i(-2) - 2 - \two   
        = \hV( 2 \epsilon - \two). \end{align*}
By $\hV$-separatedness, this implies $2 \epsilon = \two$.
By \Cref{dcartisdcartwoplusepsilon}, the map $V: A \to A$ defined by $V(a) = \hV( \epsilon a)$ defines the structure of a $\delta$-Cartier ring on $A$. 
It is straightforward to check that the two functors are inverse to each other on the full subcategories of $V$-complete \dcart rings and $\hV$-complete \dcartwo rings.

Next we treat the comparison between 
\dcart rings and \dcartwo rings where $2$ is invertible.
For a \dcartwo ring $A$ with $1/2 \in A$, put
$\epsilon=\two/2$; clearly this is forced by the identities of \Cref{dcartisdcartwoplusepsilon}, so we need only check that the remaining identities are satisfied. 
Since $\two^2 = 4$ and $F(\two) = 2$, we obtain 
 $\epsilon^2=1$ and $F(\epsilon)=1$, so
$\delta(\epsilon)=0$. Also, dividing $\hV(\two)=2+\two$ by 2 
yields $\hV(\epsilon)=1+\epsilon$. 
We now conclude thanks to \Cref{dcartisdcartwoplusepsilon} that the functor of \Cref{dcartwofromdcart} is an equivalence between the full subcategories of \dcart rings and \dcartwo rings where $2$ is invertible.
\end{proof}

We now obtain the analog of \Cref{Vcompletecartierwitt} for \dcartwo rings. 
\begin{theorem} 
The functor $A\mapsto A/\hV$ from \dcartwo rings to rings commutes with colimits.
Its right adjoint carries a ring $R$ to the \dcartwo ring $W(R)$. Moreover,
$R\mapsto W(R)$ induces an equivalence between commutative rings and
$\hV$-complete \dcartwo rings.
    \label{mapsbetweendcartwoandVcomplete}
Explicitly,   let $A$ be a \dcartwo ring. 
    Let $B$ be a $\hV$-complete \dcartwo ring.  
    Then $\Hom_{\ddcartwo}(A, B) \simeq \Hom(A/\hV, B/\hV)$. Moreover, this also agrees with the collection of $\delta$-maps $A \to B$ that carry $\hV(A) $ into $\hV(B)$. 
\end{theorem}
\begin{proof}
    This follows from 
    \Cref{comparisonofdcartanddcartwoonVcompleteor2invertible} and 
    \Cref{Vcompletecartierwitt}. 
    For the last claim, 
    we use Joyal's theorem to conclude that $\delta$-maps $A \to B$ are identified with ring maps $A \to B/V = B/\hV$, and the condition that $\hV(A)$ is carried into $\hV(B)$ yields the subset of ring maps $A/\hV \to B/V$. 
\end{proof}

\subsubsection{2-torsionfree \dcartwo rings} 
We now establish a 2-primary analog of 
\Cref{torsionfreedcartring}.

\begin{proposition}[2-torsionfree \dcartwo rings]
    \label{ptorsionfree2Cart}
The functor $A \mapsto (A, \hV(A))$ establishes an equivalence between the category of 2-torsionfree \dcartwo rings and the category of pairs $(A,  I)$ where $A$ is a 2-torsionfree $\delta$-ring and $I \subset A$ is an ideal satisfying: 
\begin{enumerate}
    \item There exists $\two \in A$ such that $\two^2 = 4$, $\delta(\two) = -1$, and $\two + 2 \in I$.
    \item The $\delta$-ring Frobenius $F: A \to A$ induces an isomorphism of abelian groups $F|_I: I \xrightarrow{\sim} \two A$.
\end{enumerate}
\label{dcartwofromideal}
\end{proposition}

\begin{proof}
First, let us check that the functor is well-defined, i.e., that if $A$ is a 2-torsionfree \dcartwo ring, then the pair $(A, \hV(A))$ satisfies the conditions. 
This follows easily from the definition of a \dcartwo ring and the identities of \Cref{hVof2tilde}. 

Conversely, 
fix a pair $(A, I)$ satisfying the above conditions. 
We observe that $\two \in A$ is unique subject to the above conditions: in fact, $2 + \two$ is the unique element of $I$ that maps via $F$ to $4$. 
We can define $\hV: A \to A$ by $\hV(a) = F|_I^{-1}(\two a)$.
One checks that $\hV(\two) = 2 + \two$ (since both sides belong to $I$ and map via $F$ to $4$), 
and the \dcartwo identities follow from the properties of $F$ and $I$, cf.~\Cref{hVof2tilde}. This defines the inverse functor from the category of pairs $(A, I)$ satisfying the above conditions to the category of 2-torsionfree \dcartwo rings.
\end{proof}

\subsubsection{Structure of free objects}

\begin{construction}[Free \dcartwo rings]
    \label{cons:freedCarttwo}
    Let $R$ be a $\delta$-ring. By the adjoint functor theorem, we can form the free \dcartwo ring on $R$, which we denote $R[\hV]$.
    
\end{construction}
However, $R[\hV]$ has a more complicated description than the free \dcart ring $R[V]$ of \Cref{cons:freedCartring}. 
The description presented below (\Cref{prop:RhVdecomposition} and \Cref{cokernelRhVtoRV}) was originally suggested by ChatGPT.
\begin{proposition}
    \label{prop:RhVdecomposition}
    Suppose $R$ is a $\delta$-ring such that $F: R \to R$ is split injective  as a map of abelian groups (in particular, $R$ is 2-torsionfree, cf.~\Cref{torsionmeansnilpotent}), and choose an abelian group decomposition 
    $R \simeq F(R) \oplus C$ for $C \subset R$.\footnote{Note that this is always possible if $R$ is a free $\delta$-ring.}

  Then $R[\hV]$ admits a direct sum decomposition of the form
    \begin{equation} \label{RhVdecomposition}
        R[\hV] = R \otimes_{\mathbb{Z}} \mathbb{Z}[\two] \oplus \bigoplus_{i > 0} \hV^i( R \oplus \two. C).
    \end{equation}
    That is, every $x \in R[\hV]$ has a unique expression of the form
    \begin{equation} \label{rhvexpr} x = a_0 + \two b_0 + \sum_{i > 0} \hV^i( a_i + \two  b_i), \end{equation}
    where $a_i, b_i \in R$ (with all but finitely many equal to zero) and $b_i \in C$ for $i > 0$. 
\end{proposition}
\begin{proof}
    In fact, it is easy to see that the set of expressions \eqref{rhvexpr} 
    where all $a_i, b_i \in R$ for $i \geq 0$ (with no assumption yet that $b_i \in C$ for $i > 0$) is stable under $F, \hV, \delta$, and multiplication. Since $R[\hV]$ is generated by $R$ under the \dcartwo operations, it follows that
    any $x \in R[\hV]$ admits a representation as in \eqref{rhvexpr} with all $a_i, b_i \in R$. 
    We need to prove that, given $x$, it is possible to choose $a_i, b_i \in R$ such that $b_i \in C$ for 
    $i  > 0$. 
    This follows from repeatedly simplifying via $\hV^i( \two F(b)) = \hV^{i-1}( (2+ \two)b)$ (cf.~\Cref{hVof2tilde}). 
    
    We now need to show uniqueness, i.e., that 
    if 
\begin{equation} \label{redundantsum} a_0 + \two b_0 + \sum_{i > 0} \hV^i( a_i + \two  b_i) = 0 \end{equation} with the $a_i, b_i \in R$ and $b_i \in C$ for $i > 0$, 
then $a_i = b_i = 0$ for all $i$. 

Note that (thanks to \Cref{dcartwofromdcart}) we have a natural map of \dcartwo-rings $R[\hV] \to R[V]$.
Furthermore, $R[\hV][1/2]= R[V][1/2]$, since the theories of \dcart rings and \dcartwo rings agree after
inverting 2 (\Cref{comparisonofdcartanddcartwoonVcompleteor2invertible}). 

For an element of $R[V]$, and hence in particular for an element of $R[\hV]$, we can consider its ``leading coefficient,''
i.e., the leading coefficient of $\sum_{i\geq 0} V^i(r_i)$ is $V^{j}(r_j)$ for $j$ maximal such that $r_j \neq 0$. 
We check: 
\begin{itemize}
\item 
    For $i > 0$, 
the leading coefficient of $\hV^i(a_i)$ is $V^{i+1}(F(a_i))$. 
\item For $i > 0$, the leading coefficient of $\hV^i( \two b_i) = V^i (2b_i)$ is simply $V^i(2b_i)$. 
\item The leading coefficient of 
$\two b_0 = (-2 + V(2))b_0$ is $V(2 F(b_0))$. 
\end{itemize}

All these can be deduced from the identity
\begin{gather} \hV^i(x) = V^{i-1}(\hV(x)) = -V^{i}(x) + V^{i+1}(F(x)), \quad i \geq 1. \label{somemorehVidentities}
\end{gather}
Comparing leading coefficients then easily implies that the vanishing of \eqref{redundantsum} implies that $a_i = b_i = 0$ for all $i$. 

\end{proof}

\begin{corollary}
    \label{cokernelRhVtoRV}
For any $\delta$-ring $R$ satisfying the conditions of \Cref{prop:RhVdecomposition}, the map $R[\hV] \to R[V]$
is injective.
Moreover, given a sum 
$x = \sum_{i \geq 0 } V^i (c_i)
\in R[V]$, the following are equivalent: 
\begin{enumerate}
\item $x \in R[\hV] \subset R[V]$. 
\item 
$\sum_{i > 0} c_i^{1/2^i} = 0 \in (R/2)_{\mathrm{perf}} $. 
\end{enumerate}
\end{corollary}
\begin{proof} 
By \Cref{prop:RhVdecomposition}, $R[\hV]$ is 2-torsionfree. We also know that the map $R[\hV] \to R[V]$ becomes an isomorphism after inverting 2, whence it must be injective. 
For any $i > 0$ and $r \in R$, one has $V^i( 2r) \in R[\hV] \subset R[V]$. 
In fact, this follows from $V^i(2r) = \hV^{i-1} \hV(\two r)$ for any $i \geq 1$.

For any $y = \sum_{i \geq 0} V^i(d_i) \in R[V]$, set $\nu(y) = \sum_{i > 0} d_i^{1/2^i} \in (R/2)_{\mathrm{perf}}$. 
Clearly $\nu: R[V] \to (R/2)_{\mathrm{perf}}$ is a well-defined, surjective map of abelian groups. Furthermore, $\nu$ vanishes on $R[\hV]$ by the preceding paragraph, in light of \eqref{somemorehVidentities} and the identity $\two = 2(V(1)-1)$ in $R[V]$. 
Conversely, 
given $x \in R[V]$ such that $\nu(x) = 0$, 
the combination of the identity \eqref{somemorehVidentities} and 
$\hV^i( \two b) = V^i (2b)$ for $i > 0$ implies that $x$ can be rewritten as an element of $R[\hV]$.
\end{proof} 

\begin{remark} 
In the proof of \Cref{cokernelRhVtoRV}, the map 
$\nu: R[V] \to (R/2)_{\mathrm{perf}}$ is a taut derivation, where we make $(R/2)_{\mathrm{perf}}$ into a Frobenius module over $R[V]$ 
via the natural map $R[V] \to R/2$ which annihilates $V( R[V])$. 
That is, $R[\hV]$ is the kernel of a taut derivation on $R[V]$ which annihilates $R \subset R[V]$. 
One checks in fact that $\Omega^1_{R[V]/R}[F^{-1}] \simeq (R/2)_{\mathrm{perf}}$, via the map $\nu$. 

By contrast, 
when $p > 2$, a taut derivation on $R[V]$ into a derived $p$-complete module that annihilates $R$ automatically vanishes. This is one reason that the prime 2 is special in the theory. 

\end{remark}

\begin{remark}[{The \dcartwo ring $R[\hV]$ in general}]
Let $R$ be any $\delta$-ring. 
Then $R[\hV]$ can be described as the 
kernel of the taut derivation $R[V] \to (R/2)_{\mathrm{perf}}$ of \Cref{cokernelRhVtoRV}.
To see this, we need to show that this functor is left Kan extended from free $\delta$-rings. 
This follows because the functor $R \mapsto (R/2)_{\mathrm{perf}}$ is left Kan extended from free $\delta$-rings as a functor to the derived category; indeed, $(R/2)_{\mathrm{perf}}$ is also the perfection of $R\quot 2$. 
\end{remark} 

\begin{example}
    \label{freehVofperfectdelta}
By \Cref{prop:RhVdecomposition}, the initial \dcartwo ring 
$\mathbb{Z}[\hV]$ has as $\mathbb{Z}$-basis
$1, \two, \hV(1), \hV^2(1), \dots$. 

More generally, let $R$ be a perfect $\delta$-ring. 
Then any element of $R[\hV]$ can be uniquely expressed as a finite sum of the form 
$$  \two b_0 + \sum_{i \geq 0} \hV^i(a_i) $$
for $b_0 \in R$ and $a_i \in R, i \geq 0$ with all but finitely many $a_i$ equal to zero.
\end{example}

\begin{proposition}
\label{chWZ2description}
The natural map $\mathbb{Z}[\hV] \to \chW( \mathbb{Z}_2)$ is an isomorphism after 2-completion.
\end{proposition}
\begin{proof}
In fact, we know from \Cref{chWZpdescription}
that we have an  inclusion $\chW( \mathbb{Z}_2) $ into the 2-completion of $\mathbb{Z}[ V]$. 
This inclusion is proper, because $V(1) \notin \chW( \mathbb{Z}_2)$ (cf.~\Cref{chWZpdescription} and \Cref{impossibilityinWZ4}).
Therefore, $\chW( \mathbb{Z}_2)$ contains the 2-completion of $\mathbb{Z}[\hV]$, since  $\mathbb{Z}[\hV] \to \mathbb{Z}[V]$ is injective. Since the  cokernel of this map is $\mathbb{Z}/2$ by \Cref{cokernelRhVtoRV},  the result follows. 
\end{proof}

\begin{remark}
    As a result, we observe the following general principle. To verify an identity in an arbitrary \dcartwo ring, it suffices to verify it in the case of a \dcartwo ring that arises from a \dcart ring via \Cref{dcartwofromdcart} (and even one that is 2-torsionfree). This is because all identities can be checked in the free \dcartwo rings; now we use that free \dcartwo rings admit injective maps into free \dcart rings by \Cref{cokernelRhVtoRV}. In fact, this reasoning shows that it suffices to verify any such identities in $W(R)$, for $R$ a 2-torsionfree ring. 
\label{generalprincipleforcheckingidentities}
\end{remark}

\begin{lemma}\label{hVtwodeltahVidentity}
Let \(A\) be a \dcartwo ring.  For every \(s\in A\), one has
\[
\two\delta(\hV(s))=2s-\hV(\two s^2).
\]
\end{lemma}
\begin{proof}
    This can be checked by a direct computation using the identities of a \dcartwo ring, but it also follows from the general principle of \Cref{generalprincipleforcheckingidentities} and an easy calculation in a \dcart ring. 
\end{proof}

Next, we record a result stating that the
last \dcartwo identity takes the same form as the usual \dcart identity when $x \in \hV(A)$.
This could also have been proved via direct calculation.

\begin{lemma} 
    \label{hVversusVdcartidentity}
Let $A$ be a \dcartwo ring, and let $x \in \hV(A)$. Then 
\[ \delta(\hV x) = x - \hV( x^2). \]
\end{lemma} 
\begin{proof} 
In fact, 
if $A$ arises from a \dcart ring, then $\hV = V$ on $\hV(A)$. The result follows from \Cref{generalprincipleforcheckingidentities}. 
\end{proof} 

\begin{proposition}\label{BtoBhVtautequivalence}
Suppose $p =2$, and let $B$ be any $\delta$-ring.
The map
$L_B[F^{-1}] \to L_{B[\hV]}[F^{-1}]$
is an isomorphism on 1-truncations.
In particular, the natural map $B \to B[\hV]$ is a taut equivalence.
\end{proposition} 
\begin{proof} 
    This reduces to the case 
    where  $B$ is a free $\delta$-ring over $\mathbb{Z}_{(2)}$. We filter $B[\hV]$ by the subrings $B[\hV]_{\leq n}$ generated by $\hV^i(B)$ for $i \leq n$.
    As in the proof of \Cref{BtoBVtautequivalence}, it suffices to show that
    $\lt_{B[\hV]_{\leq n}/B}$ vanishes in degrees $\leq 1$ for each $n$, and for this it suffices to show that the kernel of 
    $B[\hV] \otimes_B B[\hV] \to B[\hV]$ consists of topologically $\delta$-nilpotent elements. 
    However, this requires one additional step beyond the proof of \Cref{BtoBVtautequivalence}. 

We consider the \dcartwo-ideal in $B[\hV]$ generated by 
$2 - \hV(1)$. 
Let us observe that every element in this \dcartwo ideal is topologically $\delta$-nilpotent. 
In fact, to check topological $\delta$-nilpotence in $B[\hV]$, we may work in the $\hV$-completion of $B[\hV]$, which is $W(B)$ by \Cref{mapsbetweendcartwoandVcomplete}, since the explicit description in \Cref{prop:RhVdecomposition,cokernelRhVtoRV} shows that the torsion in $W(B)/B[\hV]$ is killed by $2$. 

We know that in $W(\mathbb{Z}_2)$, the element $2 - \hV(1)$ is topologically $\delta$-nilpotent by \Cref{uinWZ2}. Moreover, the topologically $\delta$-nilpotent elements of $W(B)$ are $\hw(B) \subset W(B)$, which is stable under the \dcartwo operations.\footnote{Alternatively, one can argue directly that in any \dcartwo ring, the topologically $\delta$-nilpotent elements form a \dcartwo ideal.}
It follows easily that the \dcartwo ideal $J \subset B[\hV]$ generated by $2 - \hV(1)$  has the property that every element is topologically $\delta$-nilpotent. 

The quotient of $B[\hV]$ by the ideal $J$ is the free Dieudonn\'e $\delta$-ring $B'$ on the $\delta$-ring $B$. 
By the description of $B'$ in \Cref{prop:normalformDieudonne}, we see that $B'$ is a free $B$-module.

Now we have an evident surjection  of nonunital $\delta$-rings
$$ \ker( B[\hV] \otimes_B B[\hV] \to B[\hV]) \to \ker( B' \otimes_B B' \to B'), $$
and we have seen that the kernel 
has the property that every element is topologically $\delta$-nilpotent.
By \Cref{SEStopdnil}, 
it suffices to show that every element of the $2$-torsionfree nonunital $\delta$-ring $\ker(B' \otimes_B B' \to B')$ is topologically $\delta$-nilpotent.

By \Cref{prop:normalformDieudonne}, it suffices to show that any element of the form
$\lambda_{i, x} = V^i x \otimes 1 - 1 \otimes V^i x \in B' \otimes_B B'$ is topologically $\delta$-nilpotent, where $x \in B$ and $i > 0$.
These elements are all annihilated by a power of Frobenius, and become divisible by $2$ after applying one Frobenius. 
By \Cref{topdnilnonunitalcriterion}, it thus suffices to show that their divided powers are 2-adically nilpotent. 

In fact, the divided square of 
$\lambda_{i, x}$ is 
\begin{align} 
   \gamma_2( \lambda_{i, x}) & =  
    2^{i-1}(V^i(x^2) \otimes 1 + 1 \otimes V^i(x^2)) -  V^i(x) \otimes V^i(x) \\
& = 2^{i-1}( V^i(x^2) \otimes 1 - 1 \otimes V^i(x^2))
- (1 \otimes V^i(x) )\cdot (V^i(x) \otimes 1 -  1 \otimes V^i(x))
\end{align}
This is the sum of an element (the first term) which is divisible by $2$ 
in the ideal (when $i = 1$, we use that $V(x^2)$ is divisible by 2 since we are in a Dieudonn\'e $\delta$-ring) and the product of an element in the ideal by an element (namely, $1 \otimes V^i(x)$) which admits divided powers in the ambient ring $B' \otimes_B B'$. 
Therefore, the divided powers of $\gamma_2( \lambda_{i, x})$ are 2-adically nilpotent, whence the same holds for $\lambda_{i, x}$. 
\end{proof} 
\subsubsection{Main results}
\begin{proposition}
\label{hVisinjective}
Let $A$ be a \dcartwo ring. Then $\hV:A\to A$ is injective.
\end{proposition}
\begin{proof}
Consider the $\hV$-completion $\hat{A}$ of $A$.
By \Cref{comparisonofdcartanddcartwoonVcompleteor2invertible}, the \dcartwo ring $\hat{A}$ is induced (via the functor of \Cref{dcartwofromdcart}) by a $V$-complete \dcart ring.  In particular, $\hV$ is injective on $\hat{A}$ since $V$ is injective on any \dcart ring (\Cref{Visinjective}). 

Suppose $a \in A$ is such that $\hV(a)=0$. It follows from the previous paragraph that $a \in \bigcap _{i \geq 0} \hV^i(A)$.
Applying $F$, we find $\two a=0$. 
Since $a \in \bigcap _{i \geq 0} \hV^i(A)$, we get $2a = \two a = 0$ (using \Cref{hVbasicidentities}). 

Applying $\delta$, we obtain the identity
\begin{equation}
a=\hV(\two\delta(a)) = \hV( 2 \delta(a)),
\label{V2delta2}    
\end{equation}
since $\delta(a) \in \bigcap _{i \geq 0} \hV^i(A)$ by \Cref{deltaonhVcongruence}.

Finally, $2a = 0$ gives $F(a) = 0$ thanks to \Cref{torsionmeansnilpotent}. This implies that 
$a . \hV(A) = 0$ by the projection formula, whence $a^2 = 0$ since $a \in \hV(A)$.
Together, we get $2 \delta(a) = 0$ and, from \eqref{V2delta2}, $a = 0$, as desired.
\end{proof}

\begin{lemma}[Infinitely $\hV$-divisible torsion]
\label{deltaisooninfiniteVdivtorsion}
Let $C$ be a \dcartwo ring and consider
\[
I_C=\bigcap_{i\geq 0}\hV^i(C[2^\infty]).
\]
Then:
\begin{enumerate}
\item $I_C \subset C$ is an ideal stable under $F, \hV, \delta$. 
\item $I_C^2=0$.
\item $\delta:I_C\xrightarrow{\sim}I_C$ is an isomorphism of abelian groups.
Moreover, $\delta$ and $\hV$ induce inverses to each other on $I_C $.
\end{enumerate}
\end{lemma}
\begin{proof}
    It is easy to see that $I_C \subset C$ is an ideal. 
Moreover, $C[2^\infty]$ is stable under $F$, $\hV$, and $\delta$ (cf.~\Cref{deltapreservestorsion}). 
This implies that  $F, \hV, \delta$ preserve $I_C$ by \Cref{deltaonhVcongruence}. Moreover, 
the map $\hV: I_C \to I_C$ is an isomorphism, thanks to \Cref{hVisinjective}. 

We show that $I_C^2=0$.
In fact, let $z_1,z_2\in I_C$ and suppose $2^Nz_1=2^Nz_2=0$.
Write $z_1=\hV^{N}(w_1)$ and $z_2=\hV^{N}(w_2)$ where $w_1, w_2 \in I_C$ also satisfy $2^N w_1 = 2^N w_2 = 0$ (by \Cref{hVisinjective}).  
Then, since multiplication by 2 and $\two$ coincide on $I_C$, we find 
\[
\begin{aligned}
z_1z_2
&=\hV^{N}(w_1)\hV^{N}(w_2)  \\
&=\hV^{N}\bigl(w_1F^{N}\hV^{N}(w_2)\bigr)
=2^N\hV^{N}( w_1w_2)=0,
\end{aligned}
\]
proving $I_C^2=0$.

Moreover, $\delta$ is additive on $I_C$, since $I_C^2=0$. 
For $a \in I_C$,
we have $a\in \hV(C)$ and $a^2=0$, so \Cref{hVversusVdcartidentity}
gives $\delta(\hV(a))=a$.
\end{proof}

\begin{theorem}
\label{automatichV}
Let $A,B$ be \dcartwo rings. Let $f:A\to B$ be a map of $\delta$-rings such that
$f(\hV(A))\subset \hV(B)$, where $\hV(A)$ and $\hV(B)$ denote the image ideals
of $\hV$. Then $f$ is a map of \dcartwo rings.
Equivalently, the functor
\[
\ddcartwo\longrightarrow
\{(A,I): A\text{ is a }\delta\text{-ring and }I\subset A\text{ is an ideal}\},
\qquad
A\longmapsto (A,\hV(A)),
\]
is fully faithful, where morphisms on the right are $\delta$-ring maps
preserving the distinguished ideals.
\label{dcartwo-full-faithfulness-pairs}
\end{theorem}
\begin{proof}
The result asserts that if $f: A \to B$ is a map of $\delta$-rings such that $f(\hV(A)) \subset \hV(B)$, then $f$ commutes with $\hV$.

Suppose first that 
$I_B \stackrel{\mathrm{def}}{=} \bigcap_{n \geq 1} \hV^n(B[2^\infty])$  vanishes. 
Then $B$ injects into 
the product of the $\hV$-completion of $B$ and $B[1/2]$. 
Thus, it suffices to prove the result after replacing $B$ by either of these two rings. 
In this case, the claim follows  from \Cref{ptorsionfree2Cart} and \Cref{mapsbetweendcartwoandVcomplete}. 

We now prove the result in general.
The theorem statement amounts to the commutativity of the diagram
\[
\begin{tikzcd}
A[\hV] \ar[r] \ar[d] & B[\hV] \ar[d] \\
A \ar[r] & B
\end{tikzcd}
\]
where \(A[\hV]\) and \(B[\hV]\) denote the free \dcartwo rings on the underlying
\(\delta\)-rings of \(A\) and \(B\), the maps \(A[\hV]\to A\) and \(B[\hV]\to B\)
come from the \dcartwo structures on \(A,B\), and the map \(A[\hV]\to B[\hV]\)
is obtained by applying \((-) [\hV]\) to \(f\). We know this diagram commutes
after passing to \(B/I_B\), and after precomposing with the section
\(A\to A[\hV]\). However, \(B\to B/I_B\) is a square-zero extension of
\(\delta\)-rings and \(\delta:I_B\simeq I_B\) by
\Cref{deltaisooninfiniteVdivtorsion}. 
Since \(L_{ A[\hV]/A}[1/F]\) vanishes in degrees $\leq 1$ by
\Cref{BtoBhVtautequivalence}, the result follows.
\end{proof}

\subsection{Universal property of $\chW$}

In this subsection and the sequel we use the following notational convention.  If \(p>2\), we consider
the category \(\ddcart\) of \dcart rings and write \(\hV\) for the usual
operator \(V\).  If \(p=2\), we only consider the category \(\ddcartwo\) of
\dcartwo rings and write $\hV$ as usual. 

We will use the following analog of \Cref{deltaisooninfiniteVdivtorsion}. 
\begin{proposition}[Infinitely \(\hV\)-divisible ideals]
\label{infiniteVdivisibleideals}
Let \(A\) be a derived \(p\)-complete \dcart ring if \(p>2\) (resp.  a  derived
\(2\)-complete \dcartwo ring if \(p=2\)). Put
\[
I=\hV^\infty A\stackrel{\mathrm{def}}{=}\bigcap_{i\geq 0}\hV^i(A).
\]
Then \(I^2=0\), and \(\delta:I\xrightarrow{\sim}I\).  
On $I$, $\delta$ and $\hV$ are inverses to each other.
\end{proposition}
\begin{proof}
We recall that if \(M\) is derived \(p\)-complete and
\(m_0,m_1,\dots\in M\) are such that \(m_n=pm_{n+1}\) for all \(n\geq 0\),
then \(m_0=0\).  Indeed, such a system defines a map
\(\mathbb{Z}[1/p]\to M\), and
\(\operatorname{Hom}(\mathbb{Z}[1/p],M)=0\) for derived \(p\)-complete \(M\).

The group \(I\) is derived \(p\)-complete and  \(\hV:I\to I\) is an
automorphism, cf.~\Cref{Visinjective} if \(p>2\) and
\Cref{hVisinjective} if \(p=2\). 

Let \(x,y\in I\). Set
\[
x_n=\hV^{-n}(x),\qquad y_n=\hV^{-n}(y),\qquad m_n=\hV^n(x_ny_n).
\]
Then \(m_0=xy\). The product formula
\(\hV(a)\hV(b)=p\hV(ab)\) for $a, b \in I$ (valid at $p = 2$ because multiplication by 2 and $\two$ agree on $I$) gives \(m_n=pm_{n+1}\).   Thus \(xy=0\), proving \(I^2=0\).

It remains to identify \(\delta\) on \(I\).  If \(p>2\) and \(z=\hV(w)\) with
\(w\in I\), then
\[
\delta(z)=\delta(\hV(w))=w-p^{p-2}\hV(w^p)=w,
\]
since \(I^2=0\).  Thus \(\delta\) is inverse to \(\hV\) on \(I\).
At $ p= 2$, the same calculation as 
in \Cref{deltaisooninfiniteVdivtorsion} shows that $\delta$ and $\hV$ are inverses to each other on $I$.
\end{proof}

\begin{lemma} 
    Let $A$ be a  \dcart ring (resp. a  \dcartwo ring if $p = 2$). 
Let $M$ be a derived $p$-complete $A$-module equipped with an isomorphism \begin{equation} \label{phiiso}\phi: M \simeq F_* M \end{equation} of $A$-modules. 
Then the 
$A$-module structure
on $M$ extends uniquely to an $\widehat{A}_{\hV}$-module structure
such that \eqref{phiiso} is also $\widehat{A}_{\hV}$-linear. 
\label{VcompletionFrobsemilinear}
\end{lemma}
\begin{proof}
    Let $B$ be a $\delta$-ring. 
    Then we observe that the category of perfect Frobenius $B$-modules   is equivalent to the analogous category with $B_{\mathrm{perf}} = \varinjlim_F B$ in place of $B$.
    The result now follows because $M$ is derived $p$-complete and 
   $A$ and $\widehat{A}_{\hV}$ have the same $p$-completed perfection by \Cref{completionofperfection}. 
\end{proof}

\begin{lemma} 
    \label{completionofperfection}
Let $A$ be a \dcart ring (resp. a \dcartwo ring if $p = 2$). 
\begin{enumerate}
    \item  
The $p$-completion of the colimit perfection $A_{\mathrm{perf}}$ of $A$ is isomorphic to $W( (A/(\hV, p))_{\mathrm{perf}})$. 

    \item
If $A$ is additionally derived $p$-complete, then 
$A^{\mathrm{perf}}$ is  isomorphic to $W( (A/(\hV, p))^{\mathrm{perf}})$.
\end{enumerate}
\end{lemma} 
\begin{proof} 
Given any $\delta$-ring $B$, the $p$-completion of the colimit perfection of $B$ is isomorphic 
to $W( (B/p)_{\mathrm{perf}})$,
and, if $B$ is derived $p$-complete, the limit perfection of $B$ (which is already $p$-complete) is isomorphic to 
$W( (B/p)^{\mathrm{perf}})$. 
The result then follows since the image of the ideal $\hV(A)$ in $A/p$ is nilpotent. 
\end{proof}

\begin{remark} 
We will use in the sequel without comment the following basic fact: if $M$ is an abelian group and $f: M \to M$ is an \emph{injective} endomorphism,
then the map from $M$ to the completion $\widehat{M}_f = \varprojlim M/f^n(M)$ has the property that $f$ acts isomorphically on both the kernel and cokernel; this follows easily from the fact that $\widehat{M}_f$ is also the derived $f$-completion of $M$.
\end{remark}

\begin{proposition}[The taut derivation of the completion quotient]
\label{Vcompletiontautfactorization}
Let \(A\) be a derived \(p\)-complete \dcart ring if \(p>2\), and a derived
\(2\)-complete \dcartwo ring if \(p=2\).   Let \(C_A=\widehat A_{\hV}/\mathrm{im}(A \to \widehat A_{\hV})\). 
We consider $C_A $ as a Frobenius module via the map  \(\hV^{-1}: C_A \to C_A\).\footnote{Note that $\hV$ is an isomorphism on $C_A$ thanks to injectivity of $\hV : A \to A$.} 

Then the
Frobenius $A$-module structure on $C_A$ uniquely extends to a Frobenius $\widehat A_{\hV}$-module structure, and with respect to this structure,
the quotient
map
\[
d_A:\widehat A_{\hV}\longrightarrow C_A
\]
is a taut derivation of the $\delta$-ring $\widehat A_{\hV}$.
\end{proposition}
\begin{proof}
Our strategy is as follows: the construction of the module structure and the verification of the relevant identities 
reduces (by left Kan extension) to the case where $A$ is the derived $p$-completion of a  free \dcart ring (resp.~the $2$-completion of a free \dcartwo ring), 
since we may write
any derived $p$-complete \dcart ring (resp.~\dcartwo ring) as a reflexive coequalizer
of a pair of $p$-completions of free
\dcart rings (resp.~\dcartwo rings).
Moreover, $\hV$-completion commutes with reflexive coequalizers, since
$\hV$ is injective on any \dcart ring (resp.~\dcartwo ring) and $\hV$-completion thus agrees with derived
$\hV$-completion. 
So we will assume throughout that $A$ is the $p$-completion of a free \dcart ring (resp.~\dcartwo ring) on a (possibly infinite) set of generators.
Write $\widehat A=\widehat A_{\hV}$,  $C=C_A$, and $d = d_A: \widehat A \to C$ the quotient map. 
Every $z\in\widehat A$ has a decomposition
$z=z_{<n}+z_{\geq n}$ with $z_{<n}\in A$ and
$z_{\geq n}\in \hV^n\widehat A$. We also use
\begin{equation}
\hV^i\widehat A\cdot \hV^j\widehat A\subset p^{\min(i,j)}\widehat A, \quad i, j \geq 1. \label{productsofVpowers}
\end{equation}
This follows from the projection formula. Note that when $p = 2$, $\widehat{A}$ comes from a \dcart ring by \Cref{comparisonofdcartanddcartwoonVcompleteor2invertible}, so the same conclusion holds.

The key observation is that 
\(C\) has bounded $p$-power torsion, and hence is classically $p$-adically separated and complete since it is derived $p$-complete.
To see this, suppose first $p > 2$. 
Let \(B\) be a free \(\delta\)-ring and let
\( 
A
\) 
be the $p$-completion of the free \dcart ring on $B$. 
Then
\[
A=\left\{\sum_{i\geq 0}\hV^i(b_i)\mid b_i\to 0\text{ $p$-adically}\right\}
\subset
\widehat A_{\hV}=\prod_{i\geq 0}\hV^i(B^\wedge_p).
\]
Thus \(C_A=\widehat A_{\hV}/A\) is even \(p\)-torsionfree. 
When $p = 2$, 
the free \dcartwo ring injects into the free \dcart ring with cokernel 2-torsion by \Cref{cokernelRhVtoRV}, so the conclusion holds by comparing $\widehat A_{\hV}/A$ with the analog for the corresponding \dcart ring.

Next, we construct the $\hat{A}$-module structure on $C$. 
In fact, this structure is a consequence of \Cref{VcompletionFrobsemilinear} applied to the $A$-module $C$ equipped with the isomorphism $\hV^{-1}: C \simeq F_* C$ of $A$-modules, but we will also spell it out in detail. 
For
\(r,z\in\widehat A\), define \(r\cdot d(z)\in C\) by the congruences
\begin{equation}
\label{tailregularizedaction}
r\cdot d(z)\stackrel{\mathrm{def}}{\equiv} d(r_{< n}z_{\geq n}) \equiv d( r z_{\geq n})\pmod {p^NC},\qquad n\geq N \geq 2,
\end{equation}
where \(z=z_{<n}+z_{\geq n}, r = r_{<n} + r_{\geq n}\) with \(z_{<n}, r_{<n}\in A\) and
\(z_{\geq n}, r_{\geq n} \in \hV^n\widehat A\); here the second congruence follows from \eqref{productsofVpowers}, and shows independence of the choice of decomposition of $r$.
Note that for a different decomposition 
$z = z'_{<n} + z'_{\geq n}$, the difference $z_{<n}-z'_{<n}=z'_{\geq n}-z_{\geq n}$ belongs to $A$,
which implies that
$d( r_{<n} ( z_{\geq n} - z'_{\geq n})) = 0$. 

This implies that the right-hand side of \eqref{tailregularizedaction} is well-defined modulo $p^N C$.
Since $C$ is $p$-adically separated and complete, this defines a well-defined element $r\cdot d(z)\in C$.
Note also that if 
$r \in \hV^n \widehat A$ for some $n\geq N$, then $r\cdot d(z)\equiv 0\pmod {p^NC}$.

We next check that
\(d:\widehat A\to C\) is a derivation.
Choose $x, y \in \widehat A$; we need to show that
$d(xy) = x \cdot d(y) + y \cdot d(x)$ in $C$. 
It suffices to prove this modulo $p^N$ for every $N \geq 2$.
We write
\[
x=x_{<n}+x_{\geq n},\qquad y=y_{<n}+y_{\geq n}\qquad (n\geq N),
\]
with $x_{<n},y_{<n}\in A$ and $x_{\geq n},y_{\geq n}\in \hV^n\widehat A$. Then
the product \(x_{\geq n}y_{\geq n}\) belongs to \(p^N\widehat A\), and hence
\begin{align*}
d(xy)
& = d( (x_{<n}+x_{\geq n})(y_{<n}+y_{\geq n})) \\
&\equiv d(x_{<n}y_{\geq n}+y_{<n}x_{\geq n}) \\
&\equiv x_{<n}\cdot d(y)+y_{<n}\cdot d(x)\\
&\equiv x\cdot d(y)+y\cdot d(x)
\pmod {p^NC}.
\end{align*}

Next, we check that 
$d: \widehat A \to C$ is a $\delta$-derivation.
Let $a \in \widehat A$. 
We need to show that
\begin{equation} \label{keydeltaderivationinproof} F( d(a)) = d( \delta(a)) + a^{p-1} d(a) \in C. \end{equation}

Suppose first $p > 2$, which implies (by freeness of $A$ as above) that     $C$ is $p$-torsionfree. Thus,   it suffices to check that 
if $a \in \widehat{A}$, then 
\begin{equation} 
    d( F(a)) = pF(d(a))  \in C. \end{equation}
In fact, if $a = a_{<1} + a_{\geq 1}$ with $a_{<1} \in A$ and $a_{\geq 1} \in \hV \widehat{A}$, then 
$F(a) = F(a_{<1}) + F(a_{\geq 1})$ with $F(a_{<1}) \in A$ and $F(a_{\geq 1}) \in \widehat{A}$, so
$$ 
d(F(a)) =  d(F(a_{\geq 1})) =  pd( \hV^{-1}(a_{\geq 1})) = p F d(a). $$

Finally, let us finish the proof in case $p = 2$.
To this end, fix $n \geq 3$ and write $a = a_{<n } + a_{\geq n}$ with $a_{<n} \in A$ and $a_{\geq n} \in \hV^n \widehat{A}$, and set $x = \hV^{-1}(a_{\geq n}) \in \hV^{n-1}(\widehat{A})$. Then
$F(d(a)) = d( x)$ by definition.
Moreover, \begin{equation} \label{deltaaexpansion}\delta(a) = \delta(a_{<n}) + \delta(a_{\geq n}) - a_{< n}a_{\geq n} . \end{equation}
Next, 
by \Cref{hVversusVdcartidentity} and \eqref{productsofVpowers}, we have
\begin{align*} 
    \delta(a_{\geq n}) &  =
x - \hV( x^2) \\
& \equiv x \pmod{ 2^{n-1} \widehat A} . \end{align*}
Applying $d$ to \eqref{deltaaexpansion},  we get 
from the above that
\begin{align*}
d( \delta(a))  & \equiv
d(x) - d( a_{<n}a_{\geq n}) \pmod{2^{n-1} C} \\
& \equiv 
F( d(a)) - a d(a) \pmod{ 2^{n-1} C}.
\end{align*}
Since $C$ is $2$-adically separated, letting $n \to \infty$ gives the desired identity.
\end{proof}

\begin{definition}[Taut square-zero extensions of \dcart and \dcartwo rings]
    \label{tautsquarezeroextensiondcartring}
A \emph{taut square-zero extension} of derived $p$-complete \dcart (resp.~\dcartwo) rings $B \to A$
is a surjective map 
of  $p$-complete \dcart (resp.~\dcartwo) rings with kernel $I$ such that
$I^2 =0$, $I$ is derived $p$-complete, and $V: I \to I$ is an isomorphism of abelian groups
(resp. $\hV: I \to I$ is an isomorphism of abelian groups).
We note that the condition implies that 
$$ B/\hV B \xrightarrow{\sim} A/\hV A .$$
\end{definition}

\begin{remark} 
If $B \to A$ is a taut square-zero extension of derived $p$-complete \dcart (resp.~\dcartwo) rings, then the 
map of underlying $\delta$-rings is a taut square-zero extension, as a consequence of \eqref{deltaVidentity} and \Cref{hVversusVdcartidentity}. 
\end{remark}

\begin{proposition} 
\label{liftcartierstructuretautsquarezero}
Let $A$ be a \dcart ring (resp.~a \dcartwo ring if $p = 2$).
Let $0 \to I \to B \to A \to 0$ be a taut square-zero extension of underlying $\delta$-rings. 
Then there exists a unique \dcart (resp.~\dcartwo) structure on $B$ such that 
the map $B \to A$ is a map of \dcart (resp.~\dcartwo) rings; it becomes a 
taut square-zero extension of \dcart (resp.~\dcartwo) rings.
\end{proposition} 
\begin{proof} We treat the case $p = 2$; the case $p > 2$ is analogous (and easier).
The category of \dcartwo rings is monadic over the category of $\delta$-rings. 
Thus, to promote $B$ into a \dcartwo ring, it suffices to produce a map of $\delta$-rings 
\begin{equation} B[\hV] \to B \end{equation}
satisfying the monad identities; in particular, it should restrict on $B \subset B[\hV]$ to the identity, and the two natural maps $(B[\hV])[\hV] \to B[\hV] \to B$ should coincide.
To construct this map, we consider the lifting problem
\[ \xymatrix{
B \ar[d]  \ar[r] &  B \ar[d]  \\
B[\hV] \ar@{-->}[ru] \ar[r] &  
A
},  \]
where the horizontal map is 
$B[\hV] \to A[\hV] \to A$ given by the \dcartwo structure on $A$.
By taut \'etaleness of $B \to B[\hV]$, this admits a unique lift $B[\hV] \to B$ making the diagram commute.
Similarly, the two natural maps $(B[\hV])[\hV] \to B[\hV] \to B$ coincide, since they coincide after either precomposing with $B  \subset (B[\hV])[\hV]$ or postcomposing with the map $B \to A$.
\end{proof}

\begin{proposition} 
    \label{tautrigidityimpliesuniversalproperty}
Let $S$ be a derived $p$-complete \dcart  ring (resp. a derived $2$-complete  \dcartwo ring if $p = 2$). Suppose the underlying $\delta$-ring of $S$ is taut rigid in the sense of \Cref{delta-taut-rigidity}. 
Then
for every derived
\(p\)-complete \dcart (resp. \dcartwo) ring \(A\), the natural map
\[
\Hom_{\ddcart}(S,A)\longrightarrow \Hom(S/\hV,A/\hV)  \ \text{resp.}\  \Hom_{\ddcartwo}(S,A)\longrightarrow \Hom(S/\hV,A/\hV)
\]
is a bijection.  
\end{proposition} 
\begin{proof} 
We treat the case $p = 2$; the case $p > 2$ is analogous (and easier). The idea is to factor 
the map from $A$ to its completion as a composite of a taut square-zero extension and the kernel of a taut derivation. 

Let \(I=\hV^\infty A\) and \(A'=A/I\) be the universal separated quotient of $A$.  By
\Cref{infiniteVdivisibleideals}, the map \(A\to A'\) is a taut square-zero
extension of \dcartwo rings
in the sense of \Cref{tautsquarezeroextensiondcartring},
and \(A'/\hV\simeq A/\hV\).  The  taut rigidity assumption on \(S\) gives
\[
\Hom_{\ddcartwo}(S,A)\simeq \Hom_{\ddcartwo}(S,A');
\]
in fact, taut rigidity gives an isomorphism on maps of $\delta$-rings, but we invoke \Cref{automatichV} to conclude that any map of $\delta$-rings $S \to A$ such that the map $S \to A'$ is a map of \dcartwo rings is automatically a map of \dcartwo rings.

Moreover, \(\widehat A_{\hV}=\widehat {A'}_{\hV}\), and \(A'\) is the kernel of the
taut derivation
\[
d:\widehat A_{\hV}\longrightarrow \widehat A_{\hV}/A'
\]
from \Cref{Vcompletiontautfactorization}. This derivation vanishes on $S$ by the rigidity hypothesis, whence the image of $S$ is contained in $A'$. Thus
\[
\Hom_{\ddcartwo}(S,A)\simeq \Hom_{\ddcartwo}(S,\widehat A_{\hV}).
\]
Finally, \(\widehat A_{\hV}\) is \(\hV\)-complete, so
\Cref{mapsbetweendcartwoandVcomplete}
identifies the last set with
\(
\Hom(S/\hV,\widehat A_{\hV}/\hV)=\Hom(S/\hV,A/\hV). \) 
\end{proof} 

\begin{theorem}[Universal property of \(\chW\) as a \dcart or \dcartwo ring]
\label{chWuniversalcartier}
Let \(R\) be a \(p\)-complete ring with bounded $p$-power torsion, such that \((R/p)_{\mathrm{red}}\) is perfect.
Consider $\chW(R)$ as a \dcart ring (resp.~\dcartwo ring if $p = 2$). 
\begin{enumerate}
\item If \(p>2\),   for every derived \(p\)-complete
\dcart ring \(A\), the natural map \(R\to\chW(R)/\hV\) induces an isomorphism
\begin{equation}  \label{chWuniversalproperty1}
\Hom_{\ddcart}(\chW(R),A)\xrightarrow{\sim}\Hom(R,A/\hV).
\end{equation}
\item If \(p=2\), for every derived \(2\)-complete \dcartwo ring \(A\),
the natural map \(R\to\chW(R)/\hV\) induces an isomorphism
\begin{equation} \label{chWuniversalproperty2}
\Hom_{\ddcartwo}(\chW(R),A)\xrightarrow{\sim}\Hom(R,A/\hV).
\end{equation}
\end{enumerate}
\end{theorem}
\begin{proof}
In both cases, \Cref{prop:chWmodV} and passage to the inverse limit give
\(\chW(R)/\hV\simeq R\).  Moreover, 
by \Cref{tautrigidityofchWasdcart}, the underlying \(\delta\)-ring of \(\chW(R)\) is taut rigid.
Applying \Cref{tautrigidityimpliesuniversalproperty} to the source \(\chW(R)\) and the
target \(A\), we get
\[
\Hom(\chW(R),A)\simeq \Hom(\chW(R)/\hV,A/\hV)=\Hom(R,A/\hV),
\]
where the first Hom is taken in \(\ddcart\) if \(p>2\), and in \(\ddcartwo\) if
\(p=2\).
\end{proof}
\subsection{\dcart and \dcartwo envelopes; construction of $\chW$ in general}
In this subsection, 
we extend the definition and universal property of $\chW$ to all $p$-completely nilperfect rings, by enforcing the universal property of \Cref{chWuniversalcartier}. 
To prove that an object satisfying the desired universal property exists,
we first introduce the notion of \dcart and \dcartwo envelopes.

Consider the forgetful functor from \dcart rings (resp.~\dcartwo rings) to pairs of a $\delta$-ring and an ideal. 
We know that this functor is fully faithful (\Cref{automaticV,automatichV}). Moreover, this functor commutes with limits, since \(V\) (resp.~\(\hV\)) is always injective (\Cref{Visinjective,hVisinjective}).

\begin{construction}[\dcart and \dcartwo envelopes]
    Let $(A_0, I_0)$ be a pair consisting of a $\delta$-ring $A_0$ and an ideal $I_0 \subset A_0$. By the adjoint functor theorem, we can form the universal \dcart (resp.~\dcartwo) ring $A$ with a $\delta$-map $A_0 \to A$ such that the image of $I_0$ in $A$ is contained in $\hV(A)$.
    We will call this the \emph{\dcart envelope} (resp.~\emph{\dcartwo envelope}) of the pair $(A_0, I_0)$.
\end{construction}

\begin{example}
    The \dcart (resp.~\dcartwo) envelope of the pair $(A_0, 0)$ is the free \dcart (resp.~\dcartwo) ring on the $\delta$-ring $A_0$ (cf.~\Cref{cons:freedCartring} and \Cref{cons:freedCarttwo}).
\end{example}

\begin{example}
    \label{envelopeofcartentirering}
    Let $A_0$ be any $\delta$-ring, and consider the pair $(A_0, A_0)$.

    The \dcart (resp.~\dcartwo) envelope of $(A_0, A_0)$ is $(A_0)_{\mathrm{perf}}[1/p]$ with Verschiebung given by $p F^{-1}$. 
This is equivalent to the following observation: if $A$ is a \dcart (resp.~\dcartwo) ring under $A_0$ such that $1 \in \hV(A)$, then $p \in A^{\times}$ and $A$ is perfect as a $\delta$-ring.
In fact, $1 \in \hV(A)$ implies that $A = \hV(A)$, whence $\hV: A \to A$ is an isomorphism. 
Choosing $x$ such that $\hV(x) = 1$ and applying $F$ (resp. $F^2$ when $p = 2$), we find that $p$ is invertible in $A$.   
Then, the identity $FV = p$ (resp.~$F\hV = \two$, with $\two$ a unit) implies that $F$ is also an isomorphism, so $A$ is perfect as a $\delta$-ring.
\end{example}

\begin{example}
    \label{envelopeofcartidealpn}
   Consider the initial \dcart ring $A_0 = \ZZ[V]$ for $ p> 2$.  
   We let $I_0$ be the ideal $( V(A_0), p^n)$ for some $n \geq 1$. 
   By construction, the \dcart envelope of the pair $(A_0, I_0)$ is the universal \dcart  ring $A$ such that $A/V$ is annihilated by $p^n$.
   Unwinding \Cref{chWuniversalcartier}, the derived $p$-completion of this envelope is $\chW( \ZZ/p^n)$.

 We can give an explicit description of the envelope by hand (before $p$-completion). 
    Let $A$ be a \dcart ring  such that 
   $p^n = V(a)$ for some $a \in A$. We claim that $a = p^{n-1}$. 
   To see this, we apply $F$ to obtain that $p( a -p^{n-1}) =0 $. 
   Applying $\delta$, we find 
   \[ p^{n-1} - p^{np-1} = \delta(p^n) = a - p^{p-2} V(a^p), \]
    which easily gives the claim. 
    It follows that the \dcart  envelope of the pair $(A_0, I_0)$ is given by the quotient of $A_0$ by the ideal generated by $V^i( p^n - V(p^{n-1}))$ for $i \geq 0$ (one checks that this is stable under $\delta$).
\end{example}

\begin{example}
    \label{dcartenvelopewhencontainsV}
    One important class of cases to consider is the \dcart (resp.~\dcartwo) envelope
    of a pair $(A_0, I_0)$ where $A_0$ is already a \dcart (resp.~\dcartwo) ring and $I_0 \subset A_0$ contains $\hV(A_0)$. 
    In this case (thanks to \Cref{Visinjective} and \Cref{hVisinjective}), the envelope is the universal \dcart (resp.~\dcartwo) ring $A$ under $A_0$ such that the map $A_0/\hV(A_0) \to A/\hV$ annihilates the image of $I_0$. 

    Any construction of a    \dcart (resp.~\dcartwo) envelope can be reduced to this case: 
    the \dcart (resp.~\dcartwo) envelope of a pair $(A_0, I_0)$ is the same as the \dcart (resp.~\dcartwo) envelope of the pair $(A_0', I_0')$ where $A_0' $ is the free \dcart (resp.~\dcartwo) ring on $A_0$, and $I_0'$ is generated by $\hV(A_0')$ and the image of $I_0$ in $A_0'$.
\end{example}

\begin{example} 
Let $A$ be a \dcart ring (resp.~\dcartwo ring if $p = 2$). The \dcart (resp.~\dcartwo) envelope of the pair $(A, \hV(A) + (p))$ is called the \emph{Dieudonn\'eization} of $A$; it is the universal \dcart (resp.~\dcartwo) ring $A'$ under $A$ such that $p \in A$ is in the image of $\hV$.
Using \Cref{automaticV} (resp.~\Cref{automatichV}), one checks that
$\mathbb{Z} \to A'$ is a map of \dcart (resp.~\dcartwo) rings, so
$V(1) = p$ (resp. $\hV(1) = p$) in $A'$ and $A'$ is a Dieudonn\'e $\delta$-ring.
\end{example}

We now apply this notion to construct $\chW(R)$ for any $p$-completely nilperfect ring $R$.
First, we show that any $p$-completely nilperfect ring can be realized as a quotient of a $p$-torsionfree $p$-completely nilperfect ring, based on the following straightforward observation whose proof we omit.

\begin{lemma} 
\label{critforRredflat}
Let $R$ be any ring. Then the following are equivalent: 
\begin{enumerate}
\item  
$(R/p)_{\mathrm{red}}$ is perfect. 
\item
For any map $\mathbb{Z}[x] \to R$, there exists $m \geq 1$ such that the map
extends along
\[
\mathbb{Z}[x] \to
B_m \stackrel{\mathrm{def}}{=}
\mathbb{Z}[x, y, z]/\bigl((x-y^p)^m-pz\bigr).
\]
\end{enumerate}
\end{lemma} 

\begin{remark} 
\label{flatBm}
 Note that the maps 
$\mathbb{Z}[x] \to \mathbb{Z}[x, z] \to  B_m$ of \Cref{critforRredflat} are faithfully flat. 
\end{remark}

\begin{lemma}\label{pcompleteNilperfectQuotient}
Let $R$ be a ring such that $(R/p)_{\mathrm{red}}$ is perfect. 
Then there is a $p$-torsionfree ring $\widetilde{R}$ such that
$(\widetilde{R}/p)_{\mathrm{red}}$ is perfect, and a surjection $\widetilde{R}
\twoheadrightarrow R$. 
If $R$ is derived $p$-complete, we can arrange the same for $\widetilde{R}$. 
\end{lemma} 
\begin{proof} 
The proof is via a straightforward small object argument. Take $\widetilde{R}_0$ to be a polynomial ring surjecting onto $R$. 
By pushing out $\widetilde{R}_0$ along the maps $\mathbb{Z}[x] \to B_m$ (for various $m$),
we can find a factorization
$\widetilde{R}_0 \to \widetilde{R}_1 \to R$ such that: 
\begin{enumerate}
\item $\widetilde{R}_0 \to \widetilde{R}_1$ is faithfully flat (use
\Cref{flatBm}).  
\item Any map $\mathbb{Z}[x] \to \widetilde{R}_0$ fits into a commutative
diagram
\[ \xymatrix{
\mathbb{Z}[x] \ar[d] \ar[r] &  \widetilde{R}_0 \ar[d]  \\
B_m \ar[r] &  \widetilde{R}_1
}\]
for some $m$. 
\end{enumerate}
We can iterate this process to produce a sequence of faithfully flat maps
\[ \widetilde{R}_0 \to \widetilde{R}_1 \to \widetilde{R}_2 \to \dots \to R  \]
such that each map $\widetilde{R}_i \to \widetilde{R}_{i+1}$ has the above lifting
property, 
i.e., any map $\mathbb{Z}[x] \to \widetilde{R}_i$ has the property that the
composite $\mathbb{Z}[x] \to \widetilde{R}_i \to \widetilde{R}_{i+1}$ extends
over $B_m$ for some $m$. 

It follows that 
$\widetilde{R} \stackrel{\mathrm{def}}{=} \varinjlim \widetilde{R}_i$ is
$p$-torsionfree and  has the
property that $(\widetilde{R}/p)_{\mathrm{red}}$ is perfect. 
Finally, if $R$ is derived $p$-complete, then we can arrange the same for
$\widetilde{R}$ by replacing it with its $p$-completion. 
\end{proof}

\begin{definition}[\(\chW\) for $p$-completely nilperfect rings]
    Let $R$ be a
  $p$-completely nilperfect ring. We define $\chW(R)$ as the derived $p$-complete \dcart ring (resp.~\dcartwo ring if $p = 2$)
 satisfying the universal property of \Cref{chWuniversalcartier} (i.e., \eqref{chWuniversalproperty1} or \eqref{chWuniversalproperty2}), which we may construct as follows (the universal property shows that it is well-defined up to unique isomorphism). 
By \Cref{pcompleteNilperfectQuotient}, choose a surjection
\(\widetilde{R}\twoheadrightarrow R\), where \(\widetilde{R}\) is
\(p\)-torsionfree and \(p\)-completely nilperfect. 
We then form the derived $p$-completion of the \dcart (resp.~\dcartwo) envelope of 
the pair \((\chW(\widetilde{R}),\ker(\chW(\widetilde{R})\to R))\). This satisfies the desired universal property by \Cref{chWuniversalcartier} applied to $\widetilde{R}$. 
\end{definition}

    \begin{corollary}
The functor \(R\mapsto \chW(R)\), from $p$-completely nilperfect rings to
\dcart rings if \(p>2\) and to \dcartwo rings if \(p=2\), is fully faithful.
Moreover, for any \(p\)-completely nilperfect ring \(R\), the $\hV$-completion of $\chW(R)$ is canonically isomorphic to \(W(R)\).
\end{corollary}
\begin{proof}
The universal property of \(\chW(R)\) shows that maps into any derived $p$-complete and $\hV$-complete \dcart (resp.~\dcartwo) ring $A$ are determined by the induced map on $R \to A/\hV$. From this and \Cref{Vcompletecartierwitt,mapsbetweendcartwoandVcomplete} it follows easily that the $\hV$-completion of $\chW(R)$ is canonically isomorphic to \(W(R)\).

For $p$-completely nilperfect rings \(R\) and \(R'\), by definition we have
\[
\Hom(\chW(R),\chW(R'))\simeq \Hom(R,\chW(R')/\hV)\simeq \Hom(R,R'),
\]
where the first Hom is taken in \(\ddcart\) if \(p>2\), and in \(\ddcartwo\)
if \(p=2\).  
\end{proof}

We will not seriously consider examples of $p$-completely nilperfect rings that have unbounded $p$-power torsion in this paper. 
However, even if one starts with nilperfect rings, $\chW(R)$ has typically unbounded $p$-power torsion, so the above is necessary if one wants to iterate the $\chW$ construction.

\begin{corollary}
\label{universaltautsquarezero:general}
    Let $R$ be a $p$-completely nilperfect ring. 
    Then any taut square-zero extension of $\chW(R)$ in $\delta$-rings admits a unique splitting. 
    As a result, if the natural map $\chW(R) \to W(R)$ is surjective, then $\chW(R)$ is the 
    universal taut square-zero extension of $W(R)$. 
\end{corollary}
\begin{proof}
Any taut square-zero extension of $\chW(R)$ in $\delta$-rings automatically 
promotes to one in \dcart (resp.~\dcartwo) rings, thanks to \Cref{liftcartierstructuretautsquarezero}. 
The result then follows by the universal property of $\chW(R)$. 
\end{proof}

\section{$\chW$ for semiperfectoid rings}

In this final section, we discuss some further examples of \dcart and \dcartwo rings based on the notion of \dcart and \dcartwo envelopes; this leads to explicit descriptions of $\chW$ in several perfectoid and semiperfectoid cases.

\subsection{Computing envelopes as quotients}

We start with the following construction of $\chW(R)$ for a derived $p$-complete ring $R$ such that $R/p$ is semiperfect.
\begin{proposition}
\label{chWasperfectdeltaenvelope}
Let $R$ be a
derived $p$-complete ring written as a quotient $P/J$ where $P$ is a perfect $\delta$-ring and $J \subset P$ is an ideal. (For example, we could take $(P, J) = (W(R^{\flat}), \ker \theta)$ for any derived $p$-complete ring $R$ such that $R/p$ is semiperfect.)
Then $\chW(R)$ is the derived $p$-completion of the \dcart envelope (resp.~\dcartwo envelope if $p = 2$) of the pair $(P, J)$. 
\end{proposition}

\begin{proof}
Suppose we are given a 
derived $p$-complete pair $(A_0, I_0)$ consisting of a $\delta$-ring $A_0$ and an ideal $I_0 \subset A_0$  equipped with a map $R \to A_0/I_0$.
Suppose moreover that some power of $I_0$ is contained in $(p) \subset A_0$. 
In this case,  since $P$ is perfect, the map $$P \to  R \to  A_0/I_0$$ lifts uniquely to a map $P \to A_0$ of $\delta$-rings, and the image of $J$ is contained in $I_0$. Thus, we get a map of pairs $( P, J) \to (A_0, I_0)$.

It follows that the category of derived $p$-complete
\dcart rings (resp.~ \dcartwo rings if $p = 2$) $A$ with a map $R \to A/\hV$ is equivalent to the category of derived $p$-complete \dcart rings (resp.~ \dcartwo rings if $p = 2$) $A$ with a map of pairs $(P, J) \to (A, \hV(A))$. 
The result now follows from the universal property of $\chW(R)$ as a \dcart (resp.~\dcartwo) ring proved in \Cref{chWuniversalcartier}. 
\end{proof}

We now give a construction of the derived $p$-completion of the \dcart (resp.~\dcartwo) envelope of a pair $(A_0, I_0)$. 
If $A_0$ is already a derived $p$-complete \dcart (resp.~\dcartwo) ring and $I_0$ contains $V(A_0)$ (resp.~$\hV(A_0)$), then we show that the derived $p$-completed envelope is a quotient of $A_0$ (\Cref{envelopeisquotientexplicit}). Our main tool is a natural retraction of Verschiebung (\Cref{naturalretractionofV}). In special cases, we can give more explicit descriptions.

    \begin{lemma} 
        Let $C \subset A $ be an inclusion of \dcart (resp.~\dcartwo if $p = 2$) rings such that 
        $C/\hV C \to A/\hV A$ is surjective. 
        Then $C + pA \subset A$ is a \dcart (resp.~\dcartwo) subring. 
        Moreover, when $p = 2$, $C + pA$ contains $\two A$. 
        \label{CpluspAisacartsubring}
    \end{lemma}
    \begin{proof} 
 Clearly $C + p A$ is a subring stable under $\hV$, so it suffices to show stability 
 under $\delta$. 
 Given $x \in C + p A$ of the form $x = c + p a$ with $c \in C$ and $a \in A$, 
 we write $a = \hV^2(a_0) + c_0$ with $a_0 \in A$ and $c_0 \in C$ (using surjectivity of $C/\hV C \to A/\hV A$).
It follows that 
\begin{align*} 
\delta(x) \equiv  \delta(c + pc_0 )  + \delta( p  \hV^2(a_0) )\mod p A,
\end{align*}
The second term 
$\delta( p \hV^2(a_0))$ 
belongs to $pA$
(where we use \Cref{hVversusVdcartidentity} in case $p = 2$). 
For the last claim, if $a \in A$, we write 
$a = c_1 + \hV(a_2 )$ for $c_1 \in C$ and $a_2 \in A$, and then
$\two a = \two c_1 + 2 \hV(a_2) \in C + pA$. 
    \end{proof}

\begin{proposition}
    For any derived $p$-complete \dcart (resp.~\dcartwo if $p = 2$) ring $A$, 
    the map 
    $\hV: A \to A$ has a natural set-valued retraction $r: A \to A$, i.e., such that $r \circ \hV = \mathrm{id}_A$.\footnote{This result was suggested by ChatGPT.}
    \label{naturalretractionofV}
\end{proposition}

Since $\hV$ is injective, the main content of the result is that the retraction can be chosen naturally in $A$; we do not know if there is any canonical choice. 

\begin{proof}
By the Yoneda lemma, the statement is equivalent to the following assertion: 
	    Let $A$ be the $p$-completion of the free \dcart (resp.~\dcartwo) ring $\mathbb{Z}\{ x \} [\hV]$ on a single generator $x$.
		    Consider the endomorphism $f: A \to A$ of \dcart (resp.~\dcartwo) rings defined by $f(x) = \hV(x)$. 
		    Then $f$  admits a section $s: A \to A$ in the category of \dcart (resp.~\dcartwo) rings.
             Equivalently, by freeness, it suffices for $f$ to be surjective.

    We first show that $f$ is surjective after $\hV$-completion.
    In fact, again by the Yoneda lemma, this is equivalent to the statement that 
    for any $\hV$-complete \dcart (resp.~\dcartwo) ring $B$, the map $\hV: B \to B$ has a set-valued retraction. 
    But we have a natural equivalence $B \simeq W(B/\hV)$, and a natural retraction is given by $x \mapsto \hV^{-1}( x - [x_0])$ where $x_0$ is the image of $x$ in $B/\hV$.
 
    Let
    $A' = f(A) + p A \subset A$. By Nakayama's lemma, to prove surjectivity of $f$, it suffices to prove that $A' = A$. 
     By \Cref{CpluspAisacartsubring},  
    $A'  \subset A$ is a \dcart (resp.~\dcartwo) subring; when $p = 2$, it contains $\two A$.

    Finally, we claim that $A' = A$. 
    Note that $\hV(x) \in A'$ by definition. 
    It follows that 
    $\delta( \hV(x)) \in A'$. Using the expansion of $\delta(\hV(x))$ (and $pA \subset A'$ if $p > 2$ and  $\two A \subset A'$ if $p = 2$), we find 
    that $x \in A'$, as desired. Since $A$ is generated as a $p$-complete \dcart (resp.~\dcartwo) ring by $x$, 
    we conclude that $A' = A$.
\end{proof}

\begin{corollary}
    \label{inclusioninducesinclusiononquotientsmodV}
    Let $A \to B$ be an injective map of derived $p$-complete \dcart (resp.~\dcartwo if $p = 2$) rings. 
    Then the induced map $A/\hV(A) \to B/\hV(B)$ is also injective.
\end{corollary}
\begin{proof}
An element $a \in A$ belongs to $\hV(A)$ if and only if $a = \hV( r(a))$, where $r$ is the retraction from \Cref{naturalretractionofV}, and similarly for an element $b \in B$; the result follows.
\end{proof}

\begin{proposition}[\dcart and \dcartwo envelopes as quotients]
    \label{envelopeisquotientexplicit}
    Let $A$ be a derived $p$-complete \dcart ring (resp.~\dcartwo ring at $p = 2$). 
   Let $I \subset A$ be an ideal containing $\hV(A)$. 

   Let
   \[
   \mathfrak{a} =
   \ker\bigl(A \to W(A/\hV) \to W(A/I)\bigr).
   \]
   Then the derived $p$-complete \dcart (resp.~\dcartwo) envelope of the pair $(A, I)$ is the derived $p$-completion of the quotient of $A$ by the \dcart (resp. \dcartwo) ideal generated by 
$y - \hV( \delta(y))$ for $y \in \mathfrak{a}$.
\end{proposition}

Before derived $p$-completion, the map from $A$ to the \dcart (resp.~\dcartwo) envelope of $(A, I)$ as in the statement need not be surjective,  cf.~\Cref{envelopeofcartentirering}.

\begin{proof}
    We first show that $\mathfrak{a}$ is not too small. Consider the factorization of $A \to W(A/I)$ through its image $A/\mathfrak{a}$.  
    The map $A/\hV \to A/( \mathfrak{a}, \hV)$ is surjective. Moreover,
$A/(\mathfrak{a}, \hV) \to A/I \simeq W(A/I)/\hV$ is injective by
\Cref{inclusioninducesinclusiononquotientsmodV}. Hence
$A/(\mathfrak{a}, \hV) \simeq A/I$.

Let $A'$ be the derived $p$-completion of the quotient of $A$ by the \dcart (resp.~\dcartwo) ideal generated by  $y - \hV( \delta(y))$ for $y \in \mathfrak{a}$ as in the statement. 
It follows from the previous paragraph that $A'/\hV \simeq A/I$. 

To complete the proof, it suffices to show that any map $A \to B$ of derived $p$-complete \dcart (resp.~\dcartwo) rings such that $I$ maps to zero in $B/\hV$ factors through $A'$. 
First, $\mathfrak{a}$ maps to $\bigcap_{n \geq 1} \hV^n(B)$, as one sees by passing to $\hV$-completion, after which one obtains the map $W(A/\hV) \to W(A/I) \to W(B/\hV)$. 
Therefore, the classes
$y - \hV( \delta(y))$ for $y \in \mathfrak{a}$  map to zero in $B$ thanks to \Cref{infiniteVdivisibleideals} and the previous observation as well. The result follows. 
\end{proof}

In certain cases we can obtain simpler and more explicit descriptions. The strategy is to identify certain elements that necessarily map to zero in the \dcart (resp.~\dcartwo) envelope. In some cases, one may show that the resulting quotient is already a \dcart (resp.~\dcartwo) ring.

Recall that an element $y$ of a $\delta$-ring is said to be \emph{rank one} if $\delta(y) = 0$. 
\begin{proposition}
\label{norankoneinVideal}
Let $A$ be a derived $p$-complete \dcart ring (resp.~\dcartwo ring at $p = 2$). Then there are no nonzero rank one elements in the ideal $\hV(A) \subset A$.
\end{proposition}
\begin{proof}
    For any ring $R$, the rank one elements in the $\delta$-ring $W(R)$ are exactly the elements of the form $[r]$ for some $r \in R$; this follows easily from Joyal's theorem. In particular, there are no nonzero rank one elements in the image of $V$ (resp. $\hV$). 

    Now let $A$ be an arbitrary derived $p$-complete \dcart ring (resp.~\dcartwo ring at $p = 2$), and let $x \in A$ be such that $\hV(x)$ is rank one.
   It follows from the previous paragraph that $\hV(x)$ maps to zero in the $\hV$-completion of $A$, and hence belongs to $\bigcap_{n \geq 1} \hV^n(A)$.
    However, $\delta$ is an isomorphism on this ideal by \Cref{infiniteVdivisibleideals}, implying $\hV(x) =0 $ and then $x =0$. 
\end{proof}

\begin{corollary}
    \label{modoutbyrankoneelements}
    Let $A$ be a derived $p$-complete \dcart ring (resp.~\dcartwo ring at $p = 2$) such that Frobenius $F: A \to A$ is surjective. 
    Let $(x_i)_{i \in I}$ be a collection of rank one elements of $A$. Then the derived $p$-completion of the \dcart (resp.~\dcartwo) envelope of the pair $(A, (x_i)_{i \in I})$ is the derived $p$-completed quotient of $A$ by the ideal generated by $\hV^j(x_i)$ for $j \geq 0$ and $i \in I$.
\end{corollary}
\begin{proof} 
    By \Cref{norankoneinVideal}, 
   the elements $x_i, \hV(x_i), \hV^2(x_i), \ldots$  map to zero in the derived $p$-complete \dcart envelope of $(A, (x_i)_{i \in I})$. 
    Our assumptions imply that this ideal is stable under $\delta$ and $\hV$, and thus the quotient is a \dcart (resp.~\dcartwo) ring. The result follows. 
\end{proof} 

\subsection{Semiperfect \dcart and \dcartwo rings}

Recall that a $\delta$-ring is said to be \emph{semiperfect} if the Frobenius map is surjective. 

\begin{proposition}[Semiperfect \dcart rings] \label{semiperfectdeltaCartierringclassification}
The category of semiperfect \dcart rings  is equivalent, via the functor sending $A \in \ddcart $ to $(A, V(1))$,   to the category of pairs $(A, \xi)$ where $A$ is a semiperfect $\delta$-ring and $\xi \in A$ is an element satisfying the following conditions:
\begin{enumerate}
    \item $F( \xi) = p$. 
    \item $\delta(\xi) = 1 - p^{p-2} \xi$. 
    \item $\xi\cdot \ker(F) = 0 \subset A$
\end{enumerate}
In the $p$-torsionfree case, the second condition can be omitted. 
\end{proposition}
\begin{proof}
We construct the functor in the inverse direction, and leave it to the reader to check that both functors are inverse equivalences. 
Given a pair $(A, \xi)$ as above, we construct a \dcart ring structure on $A$ by defining $V: A \to A$ to send $x \in A$ to $\xi F^{-1}(x)$; this is well-defined as an additive map $A \to A$ by the third condition, and satisfies $FV = p$. 

The only condition to verify is the equation for $\delta ( V(a))$ for $a \in A$. 
This follows from the following calculation, via the expression for $\delta(xy)$ and $\xi^p = p^{p-1} \xi$:
\begin{align*}
    \delta(V(a)) & = \delta( \xi F^{-1}(a)) \\
    & = \delta(\xi) a + \xi^p \delta(F^{-1}(a)) \\
    & = (1 - p^{p-2} \xi) a + p^{p-1} \xi \delta(F^{-1}(a)) \\
   & = (1 - p^{p-2} \xi) a + p^{p-2} \xi ( a  - F^{-1}(a)^p) = a - p^{p-2} \xi F^{-1}(a)^p. 
\end{align*}
In the $p$-torsionfree case, the condition for $\delta V$ is redundant. 
\end{proof}

\begin{proposition}[Semiperfect \dcartwo rings] \label{semiperfectdeltaCartierringclassification2}
For $p = 2$, the category of semiperfect \dcartwo rings is equivalent, via the functor sending $A \in \ddcartwo$ to $(A, \hV(1))$,   to the category of pairs $(A, \xi)$ where $A$ is a semiperfect $\delta$-ring and $\xi \in A$ is an element satisfying the following conditions:
\begin{enumerate}
    \item $F( F( \xi)) = 2$. 
    \item $\delta(\xi) = -1$.
    \item $\xi\cdot \ker(F) = 0 \subset A$.

\end{enumerate}
In the 2-torsionfree case, the second condition can be replaced with $\xi^2 = 2 + F( \xi)$. 
\end{proposition}
\begin{proof} 
    Again, we construct the functor in the inverse direction, and leave it to the reader to check that both functors are inverse equivalences.
Given a pair \((A,\xi)\) as above, we construct a \dcartwo ring structure on \(A\) by defining \(\hV:A\to A\) to send \(a \in A\) to \(\xi F^{-1}(a)\); this is well-defined as an additive map \(A\to A\) by the third condition.  It satisfies \(F\hV=\two\), where \(\two=F(\xi)\), and the projection formula \(b\hV(a)=\hV(F(b)a)\).

The only remaining condition to verify is the equation for \(\delta(\hV(a))\) for \(a \in A\).  Since \(F\) commutes with \(\delta\), we have \(\delta(a)=F(\delta(F^{-1}(a)))\).  Hence
\[
\hV(\two\delta(a))=\hV(F(\xi)F(\delta(F^{-1}(a))))=\xi^2\delta(F^{-1}(a)).
\]
The desired identity now follows from the following calculation:
\[
\begin{aligned}
\delta(\hV(a))
&=\delta(\xi F^{-1}(a)) \\
&=\xi^2\delta(F^{-1}(a))+F^{-1}(a)^2\delta(\xi)+2\delta(\xi)\delta(F^{-1}(a)) \\
&=\xi^2\delta(F^{-1}(a))-F^{-1}(a)^2-2\delta(F^{-1}(a)) \\
&=-a+\xi^2\delta(F^{-1}(a)) \\
&=-a+\hV(\two\delta(a)).
\end{aligned}
\]
\end{proof}

\begin{example}
    \label{cyclotomicdCartring}
Let us consider the semiperfect $\delta$-ring 
$\mathbb{Z}[q^{\pm 1/p^\infty}]/(q-1)$ with the $\delta$-structure determined by $F( q^{1/p^i}) = q^{1/p^{i-1}}$ for $i \geq 1$. 
We take $\xi = [p]_{q^{1/p}} \stackrel{\mathrm{def}}{=} \frac{q-1}{q^{1/p}-1}$, which satisfies $F(\xi) = [p]_q \equiv p \mod (q-1)$. Moreover, $\xi$ annihilates $\ker F = (q^{1/p} - 1)$. 
Since this quotient is $p$-torsionfree, the second condition in
\Cref{semiperfectdeltaCartierringclassification} is automatic. Hence the
$\delta$-ring $\mathbb{Z}[q^{\pm 1/p^\infty}]/(q-1)$ 
acquires the structure of a \dcart ring, with  $V(x) = {[p]_{q^{1/p}}} ( F^{-1}(x))$. 
\end{example}

 \subsection{$\chW$ for perfectoid rings}

In the remainder of this section, we give an ``explicit'' formula for $\chW$ of a semiperfectoid ring. The starting point is the case of a perfectoid ring, which is most explicit when the ring contains a compatible system of $p$-power roots of unity.

\begin{construction}[Recollections on perfectoids]
Let $R$ be a perfectoid ring \cite[\S 3.2]{bms1}, so that $R \simeq W(R^{\flat})/\ker \theta$ for $\theta: W(R^{\flat}) \to R$ the natural map, and $\ker \theta$ is principal, generated by a distinguished element $d \in W(R^{\flat})$.
One can always choose $d$ of the form $d= [\pi] + p u $ for $\pi \in R^{\flat}$ and $u \in W(R^{\flat})^\times$. 

We write $\ainf(R) = W(R^{\flat})$.
Throughout, we let $\varpi \in R$ denote an element such that $\varpi^p$ is a unit multiple of $p$ and such that $\varpi$ admits a compatible system of $p$-power roots $\varpi^{1/p^n}, n \geq 0$.
Such an element always exists, e.g., if $d = [\pi] + p u$ is a distinguished element generating $\ker \theta$, then we can take $\varpi = \theta([\pi^{1/p}])$.
\end{construction}

\begin{lemma} 
    \label{perfectoidrootliftinglemma}
Let $R$ be a perfectoid ring. Let $z \in \varpi^i R$ for some $i \in \mathbb{Z}_{\geq 0}$. Then there exists $z' \in  \varpi^{i/p}R$ such that 
$z'^p \equiv z \mod \varpi^{i+p + 1} R$.
\end{lemma} 
\begin{proof} 
    Write $z = \varpi^i z_0$ for some $z_0 \in R$. There exists $z_0' \in R$ such that $z_0'^p \equiv z_0 \mod p \varpi R$ because $R$ is perfectoid, cf.~\cite[Lem.~3.9]{bms1}. 
    Then $z' = \varpi^{i/p} z_0'$ satisfies the desired congruence.
\end{proof}

In general, given a perfectoid ring $R$, the Frobenius $F: W(R) \to W(R)$ is not surjective, cf.~\cite[Ex.~5.4]{dk}. Nonetheless, we have the following result for $\chW(R)$ and $\hw(R)$. 
\begin{proposition}
\label{ainftochWsurjectiveperfectoid}
Let $R$ be a perfectoid ring. 
The natural map $\ainf(R) \to \chW(R)$ is surjective.
Moreover: 
\begin{enumerate}
    \item  $F: \hw(R) \to \hw(R)$ is surjective.  
    \item $F: \chW(R) \to \chW(R)$ is surjective. 
\end{enumerate}
\end{proposition} 
\begin{proof} 
We start by proving that $F: \hw(R) \to \hw(R)$ is surjective, via a successive approximation argument in three stages.
    Choose an element $\varpi \in R$ such that $\varpi^p$ is a unit multiple of $p$ and which admits a compatible system of $p$-power roots.
    
Let $x \in \hw(R)$ be given. 
We want to find $y \in \hw(R)$ such that $F(y) = x$.

First, since $R/p$ is semiperfect,  we can find $v \in \hw(R)$ such that $F(v) \equiv x \mod \hw( p R)$. 
Replacing $x$ by $x - F(v)$, we can assume without loss of generality that $x \in \hw(pR)=\hw(\varpi^pR)$.

For the second step, 
suppose that $x \in \hw( \varpi^p R)$  and let $x_0, x_1, \dots \in R$ be the sequence of Joyal coordinates of $x$. 
Since $x \in \hw( \varpi^p R)$, we have $x_0, x_1, \dots \in \varpi^p R$ and $x_n \to 0$ in the $\varpi$-adic topology as $n \to \infty$. 
Via descending induction on $n$ and starting with $w_n = 0$ for $n \gg 0$, we can find a sequence $w_n \in  \varpi R$ such that
$w_n^p + p w_{n+1} \equiv  x_n \mod \varpi^{2p +1} R$ for all $n \geq 0$, cf.~\Cref{perfectoidrootliftinglemma}. 
Replacing $x$ by $x-F(w)$, where $w \in \hw(R)$ has Joyal coordinates $(w_n)_{n \geq 0}$, we can assume that $x \in \hw( \varpi^{2p + 1} R)$.

For the third and  last step, 
suppose that $x \in \hw( \varpi^{2p + 1} R)$. 
We observe that $[ \varpi^{2p+1}] \in p \hw(R)$ 
by the fact that $p \mid [p^2]$ in $W( \mathbb{Z})$ 
\cite[Lem.~4.7.3]{drinfeld}; cf.~\Cref{pdividespthree}. 
Applying this to the $V$-adic expansion of $x$, 
 it follows that $x =  px'$ for some $x' \in  \hw(R)$. 
It follows that 
$x  = F( V(x'))$ and we conclude. This completes the proof of the surjectivity of $F$ on $\hw(R)$.

The surjectivity of $F: \chW(R) \to \chW(R)$ follows from the short exact sequence
$$ 0 \to \hw(R) \to \chW(R) \to \qperf(R/p) \to 0$$
obtained by applying the short exact sequence \eqref{secondexactseq} of
Section~3 to the $p$-nilpotent rings $R/p^n$ and passing to the inverse limit;
the transition maps are surjective, so no \(\varprojlim^1\) term appears
(using \Cref{nilinvarianceofqperf} to identify
$\qperf(R/p^n)=\qperf(R/p)$), together with the surjectivity of $F$ on
$\hw(R)$ from the preceding argument. 
Finally, the surjectivity of $\ainf(R) \to \chW(R)$ follows from the surjectivity of $F$ on $\chW(R)$ and the fact that $\ainf(R) $ is the inverse limit perfection 
of $\chW(R)$ by \Cref{completionofperfection}. 
\end{proof} 

In the rest of the subsection, we identify precisely the kernel of the surjection $\ainf(R) \to \chW(R)$ for a perfectoid ring $R$.

\begin{construction}[Setup of almost ring theory]
We have natural maps
$\ainf(R) = W(R^{\flat}) \to W(R) \to W( (R/p)_{\mathrm{red}})$. 
The composite map is surjective and exhibits the target as a $p$-completely idempotent algebra object of $D( \ainf(R) )$ as a consequence of \cite[Th.~3.5.1]{BhattLuriemodp} (cf.~also \cite[Rem.~3.5.5]{BhattLuriemodp}), whence we obtain a setup of $p$-complete almost ring theory with respect to the kernel $W( \mathfrak{m}^{\flat}) \subset \ainf(R)$ of the map $\ainf(R) \to W( (R/p)_{\mathrm{red}})$. In particular, given a derived $p$-complete $\ainf(R)$-module $M$, we 
write $M_!$ for the $p$-completion of $ W(\mathfrak{m}^{\flat}) \otimes_{\ainf(R)} M$, the associated cosaturated module.

Explicitly, there exists an element $\pi \in R^{\flat}$ such that $\pi^{\sharp} \in R$ is a unit multiple of $p$; then 
$W( \mathfrak{m}^{\flat})$ is $p$-completely generated by $[\pi^{1/p^n}]$ for $n \geq 0$. 
In this case, 
$M_!$ is realized as the derived $p$-completion of the  colimit 
\begin{equation} \label{lowershriekdesc} M_! = \left(  M \xrightarrow{ [\pi^{1-1/p}] } M \xrightarrow{ [\pi^{1/p-1/p^2}] } M \xrightarrow{ [\pi^{1/p^2-1/p^3}] } \dots \right)^{\wedge}_p . \end{equation}
This reduces to the case where $M = W( R^{\flat})$, and then one can check after reducing modulo $p$, in which case this follows from a general  statement about perfect $\mathbb{F}_p$-algebras, cf.~\cite[Th.~3.1]{Aberbach_Melvin} and \cite[Lem.~3.5.4]{BhattLuriemodp}.
\end{construction}

\begin{remark}
\label{shriekpreservesptorsionfreeinclusions}
Note that description \eqref{lowershriekdesc} also shows that $W( \mathfrak{m}^{\flat})$ is $p$-completely
flat. 
As a consequence, if $M \subset M'$ is an inclusion of $p$-complete, $p$-torsionfree $W( R^{\flat})$-modules such that the cokernel is $p$-torsionfree, then the induced map $M_! \to M'_!$ is also injective with $p$-torsionfree cokernel.
\end{remark} 

\begin{construction}[The map $\Theta$]
We will use the natural commutative diagram
\begin{equation} 
\xymatrix{
W(R^{\flat}) \ar[r]^{\Theta} \ar[rd]^{\theta} & W(R) \ar[d]^{\gh_0} \\
& R 
},
\end{equation} 
where $\Theta: \ainf(R) \to W(R)$ is the unique map of $\delta$-rings compatible with the projection to $R$. 
We write $J \subset \ainf(R)$ for $J = \ker \Theta$.
\end{construction}

\begin{proposition}[{Cf.~\cite[Lem.~3.23]{bms1}}]
\label{thetaalmostsurjectiveperfectoid}
For any perfectoid $R$, the map $\Theta: \ainf(R) \to W(R)$ is almost
surjective.
\end{proposition} 
\begin{proof} 
This follows because the cokernel of $\Theta$ is also the cokernel $C$ of the $\hV$-completion map $\chW(R) \to W(R)$ by \Cref{ainftochWsurjectiveperfectoid}. 
Since 
$\hV$ acts invertibly on $C$, it follows that 
$C$ (via $\hV^{-1}$) acquires the structure of a perfect Frobenius module over $\chW(R)$ 
and hence the structure of a module over the colimit perfection of $\chW(R)$, 
which is $W((R/p)_{\mathrm{red}})$ by \Cref{completionofperfection}. 
\end{proof}

\begin{remark} 
Let $V$ be a perfectoid valuation ring receiving a map from $\zpcycl$. 
In this case, the kernel $J$ of $\Theta: \ainf(V) \to W(V)$ is generated by $[\epsilon]-1$, where $\epsilon
\in ( \zpcycl)^{\flat}$ is the sequence $(1, \zeta_p, \zeta_{p^2}, \dots)$ of compatible $p$-power roots of unity
(cf.~\cite[Ex.~3.16 and Lem.~3.23]{bms1}). 
\end{remark} 
\begin{lemma}
\label{Jbanginjectiveperfectoid}
As above, let $J=\ker(\Theta: \ainf(R) \to W(R))$. Then the natural map $J_! \to \ainf(R)$ is injective, and the  ring $\ainf(R)/J_!$ is 
$p$-torsionfree. 
\end{lemma}
\begin{proof}
First, $W(R)$ is $p$-torsionfree. Indeed, if $pw=0$ for some $w\in W(R)$, then 
$V(w)\in \ker(W(R)^{F=0} \to \mathbb{G}_a(R))=\ker(\mathbb{G}_a^\sharp(R)\to R)$. 
An element of the latter kernel is a sequence $(x_i\in R)_{i\geq 0}$ such that
$px_{i+1}=x_i^p$ for all $i\geq 0$ and $x_0=0$. It follows that every $x_i$ is killed
by some power of $p$. Since every $p$-power torsion element of a perfectoid ring is
killed by $p$ \cite[Lem.~2.34]{bhattscholze}, every $x_i$ is nilpotent; since perfectoid rings
are reduced \cite[\S 2.1.3]{cs}, every $x_i$ is zero. Thus $W(R)$, and hence
$\ainf(R)/J\subset W(R)$, is $p$-torsionfree.

It follows from \Cref{shriekpreservesptorsionfreeinclusions} that
$J_! \to \ainf(R)_!$ is injective with $p$-torsionfree cokernel, and since $\ainf(R)_! \to \ainf(R)$ is injective  with $p$-torsionfree
cokernel $W( (R/p)_{\mathrm{red}})$, the result follows.
\end{proof}

Note that $J_!  \subset \ainf(R)$ is  a $\delta$-ideal 
as the $p$-completion of $W( \mathfrak{m}^{\flat}) J$.

\begin{theorem}
\label{thm:chWofperfectoid}
Let $R$ be a perfectoid ring. 
Then the \dcart (resp.~\dcartwo) ring $\chW(R)$ is 
the ring-theoretic cosaturation of the $\ainf(R)$-algebra $W(R)$. 
That is,  there is a natural identification of $\ainf(R)$-algebras 
$\chW(R) \simeq \ainf(R)/ J_!$, compatible with the natural projection $\chW(R) \to W(R)$.
\end{theorem}

The main strategy of this proof is to show that $\ainf(R)/J_!$ is a \dcart ring (resp. a \dcartwo ring if $p = 2$). 
We treat this by showing that $\ainf(R)/J$ is a \dcart (resp. a \dcartwo) ring and that $\ainf(R)/J_! \to \ainf(R)/J$ is a taut square-zero extension.

\begin{proposition}
\label{ainfmodJcartierperfectoid}
Notation as above, for any perfectoid ring $R$, the pair
$(\ainf(R)/J, \mathrm{im}(\ker(\theta)\to \ainf(R)/J))$ defines a \dcart ring
(resp.~a \dcartwo ring if $p = 2$). Equivalently, the image of
$\ainf(R) \to W(R)$ is a \dcart (resp. \dcartwo) subring of $W(R)$.
\end{proposition} 
\begin{proof} 
By \Cref{ainftochWsurjectiveperfectoid}, the image of $\ainf(R)\to W(R)$ is the
image of $\chW(R)\to W(R)$.
\end{proof} 

\begin{lemma}
\label{ainfJbangtoainfJtautsquarezero}
The map of $\delta$-rings
$\ainf(R)/ J_! \to \ainf(R)/J \subset W(R)$ is a taut square-zero extension.
\end{lemma} 
\begin{proof} 
    Since $J$ is contained in the kernel $W( \mathfrak{m}^{\flat})$ of $\ainf(R) \to W( (R/p)_{\mathrm{red}})$, we have $J^2 \subset J_!$, whence the map is a square-zero extension. 

By \Cref{Jbanginjectiveperfectoid}, \(J/J_!\) is \(p\)-torsionfree. It suffices to show that \(F\) is \(p\) times an isomorphism on \(J/J_!\). 

To this end, define 
an operator 
$\hV': \ainf(R)/J_! \to \ainf(R)/J_!$ as follows: choose a generator $d \in \ker(\theta) \subset \ainf(R)$ such that $d$ maps to $p \in W( (R/p)_{\mathrm{red}})$ and set $\hV'(x) = d F^{-1}(x)$; 
since $d F^{-1}(J) \subset J \subset \ainf(R)$, we get
$d F^{-1}(J_!) \subset J_!$, so $\hV'$ is well-defined.
Both $d$ and $F(d)$ act as $p$ on $J/J_!$, since their images in $W((R/p)_{\mathrm{red}})$ are $p$. Since $J/J_!$ is a $W((R/p)_{\mathrm{red}})$-module, we obtain $F \circ \hV' = \hV' \circ F = p$ on $J/J_!$, so $\hV'$ is an inverse to $\delta = F/p$ on $J/J_!$.
\end{proof}

\begin{proof}[Proof of \Cref{thm:chWofperfectoid}]
We have a natural, surjective (by \Cref{ainftochWsurjectiveperfectoid}) map
$\ainf(R) \to \chW(R)$. 
Let us first show that it annihilates $J_!$. 
This follows because the kernel and cokernel of 
$\chW(R) \to W(R)$ are almost zero:
almost surjectivity holds by \Cref{thetaalmostsurjectiveperfectoid}, while
$\hV$ is invertible on the kernel of $\chW(R) \to W(R)$, whence the $W(R)$-module
structure factors over a $W( (R/p)_{\mathrm{red}})$-module structure, so it is
almost zero.
Since $J = \ker( \ainf(R) \to W(R))$, 
it follows that $J_! \subset \ker( \ainf(R) \to \chW(R))$.

We next construct a \dcart (resp.~\dcartwo) structure on $\ainf(R)/J_!$ such that the map
$\ainf(R)/J_!\to R$ identifies $R$ with
$(\ainf(R)/J_!)/\hV(\ainf(R)/J_!)$.
In fact, this follows from \Cref{liftcartierstructuretautsquarezero}, the fact
that $\ainf(R)/J_! \to \ainf(R)/J$ is a taut square-zero extension of
$\delta$-rings by \Cref{ainfJbangtoainfJtautsquarezero}, and the \dcart
(resp.~\dcartwo) structure on $\ainf(R)/J$ from
\Cref{ainfmodJcartierperfectoid}.

Using the universal property of $\chW(R)$ from \Cref{chWuniversalcartier}, it follows that we obtain a section of $\ainf(R)/J_! \to \chW(R)$, 
necessarily an isomorphism since all maps are compatible with the maps from $\ainf(R)$. 
\end{proof}

\begin{proposition}[$\chW$ of semiperfectoid rings]
Let $R = P/I$ be a ring, where 
\begin{itemize} 
\item $P$ is a $p$-complete, perfect $\delta$-ring (e.g., we could take $P = W(R^{\flat})$). 
\item $I \subset P$ is a $p$-complete ideal containing an element $d \in I$ such that $(P, (d))$ defines a perfect prism (thus, $R$ is semiperfectoid). 
\end{itemize}
Let $\theta:P\twoheadrightarrow R=P/I$ denote the quotient map.

Then $\chW(R)$ is the derived $p$-completion of the quotient of $P$ obtained by the following two-step procedure: 

\begin{enumerate}
\item Let $\mathfrak{a} \subset P$ be the kernel of the map $P \to W( R)$, and let $\mu \subset P$ be the kernel of the map $P \to W( (R/p)_{\mathrm{red}})$. Let $\mathfrak a'=(\mathfrak a\mu)^{\wedge}_p$ be the $p$-completed ideal generated by $\mathfrak a\mu$.
 Then $(P/\mathfrak{a}', \ker(P/\mathfrak{a}' \to R))$ defines a semiperfect \dcart (resp.~\dcartwo at $p = 2$) ring. 

\item  Take the further quotient of $P/\mathfrak{a}'$ by the \dcart
(resp.~\dcartwo) ideal generated by the images of
\(y-\hV(\delta(y))\) for all \(y\in\mathfrak a\).
\end{enumerate}
\end{proposition}

\begin{proof}
Let $\mathfrak{a}'=(\mathfrak a\mu)^{\wedge}_p \subset P$ be the $p$-completed ideal generated by $\mathfrak a\mu$. 
We  first show that $P/\mathfrak{a}'$ naturally admits the structure of a \dcart (resp.~\dcartwo) ring with Verschiebung ideal generated by the image of $\ker \theta$.
Since $\mathfrak{a}$ and $\mu$ are both $\delta$-ideals, $\mathfrak{a}'$ is also a $\delta$-ideal. 

Choose $\xi \in P$ that maps to $V(1) \in \chW(P/(d))$ (resp.~$\hV(1) \in \chW(P/(d))$ if $p = 2$), which we can do since $\chW(P/(d))$ is a quotient of $P$. 
Since $\xi$ maps to $V(1)$ (resp.~$\hV(1)$ if $p = 2$) in $W(R)$ as well, it follows that 
$\xi \cdot F^{-1}(\mathfrak{a}) \subset \mathfrak{a}$. Moreover, $F^{-1}(\mu) = \mu$. 
This gives 
$$ \xi . F^{-1}( \mathfrak a') \subset \mathfrak a'  \subset P.$$
Let \(J_0=\ker(P\to W(P/(d)))\), and let
\(\mu_0=\ker(P\to W(((P/(d))/p)_{\mathrm{red}}))\). By
\Cref{thm:chWofperfectoid}, the kernel of \(P\to \chW(P/(d))\) is
\((J_0)_!\), i.e., the $p$-completed ideal generated by \(J_0\mu_0\). Since \(J_0\subset \mathfrak a\) and \(\mu_0\subset \mu\), this
kernel is contained in \(\mathfrak a'=(\mathfrak a\mu)^{\wedge}_p\). Hence
\(P/\mathfrak a'\) is a quotient of \(\chW(P/(d))\), so the identities required
of \(\xi\) descend to \(P/\mathfrak a'\). Together with
\(\xi F^{-1}(\mathfrak a')\subset \mathfrak a'\), the classifications
\Cref{semiperfectdeltaCartierringclassification,semiperfectdeltaCartierringclassification2}
give the desired \dcart (resp.~\dcartwo) structure.

By \Cref{chWasperfectdeltaenvelope}, $\chW(R)$ is the derived $p$-completion of the \dcart (resp.~\dcartwo) envelope of the pair $(P, \ker (P \to R))$. 
In order to prove the result, it thus suffices to show that any map of $\delta$-pairs $(P, \ker (P \to R)) \to (A, \hV(A))$ with $A$ a \dcart (resp.~\dcartwo) ring factors through $P/\mathfrak{a}'$ and annihilates the elements in (2).

Now given a derived $p$-complete \dcart ring $A$ (resp.~\dcartwo ring if $p = 2$) and a map of $\delta$-pairs $(P, \ker (P \to R)) \to (A, \hV(A))$, 
we know that
$\mathfrak{a}$ maps to $\bigcap_{n \geq 1} \hV^n(A)$; this follows because the map from $P$ to the $\hV$-completion of $A$ factors through $W(R)$.
But $\bigcap_{n \geq 1} \hV^n(A)$ has the structure of a perfect Frobenius module over $A$, 
whence it is naturally a module over the $p$-completed perfection of $A$; by \Cref{completionofperfection} and the factorization $P \to R \to A/\hV$, its $P$-action factors through $ W( (R/p)_{\mathrm{red}})$, hence it is annihilated by $\mu$. 
We conclude that $\mathfrak{a}'$ maps to zero in $A$. 
Thus, we obtain a map of 
\dcart (resp.~\dcartwo) rings $P/\mathfrak{a}' \to A$. 

Note that the elements
in $\mathfrak{a}$ map into $\bigcap_{n \geq 1} \hV^n(A)$, so the elements of the form $x - \xi F^{-1}(\delta(x))$ for $x \in \mathfrak{a}$ map to $x - \hV \delta(x) = 0$ in $A$ by \Cref{infiniteVdivisibleideals}. 
The result now follows. 
\end{proof}

\begin{proposition} 
    \label{modoutbyrankoneelementschW}
Let $R_0$ be a perfectoid ring, 
$d \in \ainf(R_0)$ a generator of the kernel of $\theta: \ainf(R_0) \to R_0$, 
and let $(x_\alpha)_{\alpha \in S}$ be a collection of elements of 
$R_0^{\flat}$. 
Let $R$ be the $p$-completed quotient of  $R_0$ by the ideal generated by the
elements \(x_\alpha^{\sharp}\) for \(\alpha \in S\).
Then 
the $\delta$-ring $\chW(R)$ is obtained from $\ainf(R_0)/J_!$ by taking the quotient by the $p$-complete ideal generated by the elements
\[
[x_\alpha],\ d[x_\alpha^{1/p}],\ dF^{-1}(d)[x_\alpha^{1/p^2}],\ \dots
\]
for all $\alpha \in S$.
\end{proposition} 
\begin{proof} 
This follows from the description of $\chW(R_0) = \ainf(R_0)/J_!$ from \Cref{thm:chWofperfectoid}, and then the description of $\chW(R)$ as a \dcart (resp.~\dcartwo) envelope from \Cref{chWasperfectdeltaenvelope}, and finally \Cref{modoutbyrankoneelements}. 
We note that $\hV: \chW(R_0) \to \chW(R_0)$ is given by $\hV(x) = dF^{-1}(x)$, where $d$ is a generator of $\ker \theta: \ainf(R_0) \to R_0$ which maps to $V(1)$ (resp.~$\hV(1)$ if $p = 2$) in $\chW(R_0)$; however, for forming the quotient, any choice of $d$ works.
\end{proof}

\subsection{Cyclotomic examples}
\label{subsec:cyclotomicexamples}

In this subsection, we work out explicitly $\chW$ of the perfectoid ring $\zpcycl$ and some related semiperfectoid examples.
We keep the notation and conventions from the previous subsection.
\begin{construction}[The perfectoid ring $\zpcycl$]
    Let $R = \zpcycl$. 
    Then $R$ is perfectoid, $R^{\flat} = \mathbb{F}_p[q^{1/p^\infty}]^{\wedge}_{(q-1)}$,
    $\ainf(R) = W(R^{\flat}) = \mathbb{Z}_p[q^{1/p^\infty}]^{\wedge}_{(p,q-1)}$ with $\theta: \ainf(R) \to R$ sending $q^{1/p^n}$ to $\zeta_{p^n}$,
    and $\ker \theta$ is generated by $[p]_{q^{1/p}}$. 
By \cite[Ex.~3.16 and Lem.~3.23]{bms1}, for the map
\[
\Theta:\mathbb{Z}_p[q^{1/p^\infty}]^{\wedge}_{(p,q-1)}\to W(\zpcycl),
\]
the ideal $\ker\Theta$ is generated by $q-1$.

    The ideal $W( \mathfrak{m}^{\flat})$ is $p$-completely  generated by $q^{1/p^n}-1$ for $n \geq 0$.
    Indeed, these classes map to zero in $W( (R/p)_{\mathrm{red}}) = \mathbb{Z}_p$, and one sees that they $p$-completely generate $W( \mathfrak{m}^{\flat})$ by reduction mod $p$, because they generate the kernel of the map $R^{\flat} \to (R/p)_{\mathrm{red}}$.
    \label{zpcyclexample}
\end{construction}

It follows that 
$\chW(\zpcycl) = \ainf(\zpcycl)/J_!$ is the $p$-completion of the semiperfect $\delta$-ring
\begin{equation}
\label{cosaturateddeltaring}
\mathbb{Z}[q^{\pm 1/p^\infty}]
/
\left( (q-1)(q^{1/p^n} - 1)\right)_{n \geq 1}.
\end{equation} 
The associated Verschiebung ideal is generated by $[p]_{q^{1/p}}$.

\begin{remark}[Unwinding the \dcart/\dcartwo structure on $\chW(\zpcycl)$] 
By \Cref{semiperfectdeltaCartierringclassification,semiperfectdeltaCartierringclassification2}, to specify the structure of a \dcart (resp.~a \dcartwo) ring on 
\eqref{cosaturateddeltaring}, it suffices to specify an element $\xi$ such that $\xi\cdot \ker F = 0$ and $F(\xi) = p$ (resp.~$F(F(\xi)) = 2$ and $\delta(\xi) = -1$ when $p = 2$); then, 
$\hV$ is determined by the formula $\hV(x) = \xi F^{-1}(x)$.
Moreover, $\xi$ is uniquely determined by these conditions and $\xi \in \ker \theta$, since the \dcart (resp.~\dcartwo) structure is uniquely determined 
by the Verschiebung ideal. 

A short calculation shows that the element
\[
\xi = q^{-(p-1)/(2p)}[p]_{q^{1/p}}
 = q^{-(p-1)/(2p)} + q^{-(p-3)/(2p)} + \dots + q^{(p-3)/(2p)} + q^{(p-1)/(2p)} \in \ker \theta
\]
satisfies the axioms.

\begin{itemize}
    \item  When $p > 2$, $F( \xi) - p$ is divisible by $(q-1)^2$ and thus vanishes.
 \item When $p = 2$, $F(F( \xi)) - 2$ and 
$\xi^2 - F( \xi) - 2$ are divisible by $(q-1)^2$ and thus vanish.
\item One checks that $\ker(F)$ is $p$-completely generated by $(q^{1/p} -1)( q^{1/p^n} - 1)$ for $n \geq 0$, from which one easily gets $\xi  \cdot \ker(F) = 0$.
\end{itemize}

\end{remark} 

\begin{example}
Consider the 
semiperfectoid ring
$\left(\zpcycl[x^{1/p^\infty}]\right)^{\wedge}_p/(x)$. 
    By \Cref{modoutbyrankoneelementschW}, 
  $\chW\left(\left(\zpcycl[x^{1/p^\infty}]\right)^{\wedge}_p/(x)\right)$ is the $p$-completion 
of the quotient of 
$ \mathbb{Z}[q^{\pm 1/p^\infty}, x^{1/p^\infty}]$ by the ideal generated by 
$$\{   (q-1)(q^{1/p^n}-1), n \geq 1\}, \quad \{ x, [p]_{q^{1/p}} x^{1/p}, [p]_{q^{1/p}}  [p]_{q^{1/p^2}} x^{1/p^2}, \dots \}.$$
This is a sort of ``$q$-analog''
of \Cref{semiperfectexample}. 
\end{example}

\raggedbottom
\bibliographystyle{bkmalpha}
\bibliography{bibfile}

\end{document}